\documentclass[10pt,reqno]{amsart}

\usepackage[utf8x]{inputenc}
\usepackage[english]{babel}
\usepackage{amsmath,amsfonts,amssymb,amsthm,shuffle}
\usepackage[T1]{fontenc}
\usepackage[math]{anttor}

\usepackage[top=3.5cm,bottom=3.5cm,left=3.6cm,right=3.6cm]{geometry}

\usepackage{xcolor}

\definecolor{ColBlack}{RGB}{0,0,0} % Black.
\definecolor{ColWhite}{RGB}{255,255,255} % White.
\definecolor{ColA}{RGB}{146,194,113} % Green meadow.
\definecolor{ColB}{RGB}{64,98,53} % Green swamp.
\definecolor{ColC}{RGB}{237,174,119} % Chestnut coral.
\definecolor{ColD}{RGB}{140,116,96} % Ghost chestnut.

\usepackage[hyperindex=true,frenchlinks=true,colorlinks=true,citecolor=ColC,linkcolor=ColD,
urlcolor=ColD,linktocpage,pagebackref=true]{hyperref}

\usepackage{tikz}
\usetikzlibrary{shapes}
\usetikzlibrary{fit}
\usetikzlibrary{decorations.pathmorphing}
\usetikzlibrary{arrows.meta}

\usepackage{mathtools}
\usepackage{dsfont}
\usepackage{wasysym}
\usepackage{stmaryrd}
\usepackage{mleftright}
\usepackage{cite}
\usepackage{caption}
\usepackage{subfig}
\usepackage{multirow}
\usepackage{enumitem}
\usepackage{multicol}
\usepackage{cancel}

\allowdisplaybreaks

\numberwithin{equation}{subsection}

\makeatletter
\def\l@section{\@tocline{1}{3pt}{1pc}{5pc}{}}
\def\l@subsection{\@tocline{2}{2pt}{2pc}{5pc}{}}
\makeatother

\newtheorem{Theorem}{Theorem}[subsection]
\newtheorem{Proposition}[Theorem]{Proposition}
\newtheorem{Lemma}[Theorem]{Lemma}

\renewcommand{\leq}{\leqslant}
\renewcommand{\geq}{\geqslant}

\newcommand{\ColA}[1]{\textcolor{ColA}{#1}}

\newcommand{\Hide}[1]{\ColA{\tt HIDEN}}
\newcommand{\Def}[1]{\ColA{\em #1}}
\newcommand{\Par}[1]{\mleft(#1\mright)}
\newcommand{\Bra}[1]{\mleft\{#1\mright\}}
\newcommand{\Han}[1]{\mleft[#1\mright]}
\newcommand{\Brr}[1]{\mleft|#1\mright|}
\newcommand{\Angle}[1]{\mleft\langle#1\mright\rangle}
\newcommand{\AAngle}[1]{\Angle{\Angle{#1}}}
\newcommand{\BPar}[1]{\boldsymbol{\mleft(\mright.}#1\boldsymbol{\mleft.\mright)}}

\newcommand{\OEIS}[1]{\href{http://oeis.org/#1}{{\bf #1}}}

\tikzstyle{Centering}=[{baseline={([yshift=-0.5ex]current bounding box.center)}}]

\tikzstyle{MarkA}=[draw=ColA!80,fill=ColA!8]
\tikzstyle{MarkB}=[draw=ColB!80,fill=ColB!8]
\tikzstyle{MarkC}=[draw=ColC!80,fill=ColC!8]
\tikzstyle{MarkD}=[draw=ColD!80,fill=ColD!8]

\tikzstyle{Node}=[circle,MarkA,inner sep=1pt,minimum size=2mm,thick,font=\scriptsize]
\tikzstyle{Edge}=[MarkC,cap=round,thick,rounded corners=2.5pt]
\tikzstyle{Leaf}=[rectangle,MarkD,inner sep=0pt,minimum size=1mm,thick]
\tikzstyle{NodeST}=[font=\scriptsize]

\tikzstyle{EdgeGraph}=[MarkA,cap=round,thick]
\tikzstyle{EdgeLabel}=[midway,inner sep=1pt,fill=ColWhite!0,font=\tiny]

\tikzstyle{Grid}=[color=ColBlack!30]
\tikzstyle{PathNode}=[circle,MarkB,thick,inner sep=0pt,minimum size=2.0mm]
\tikzstyle{PathStep}=[MarkC,thick]

\tikzstyle{Injection}=[ColBlack,draw,{>[scale=1.5,length=4,width=5]}-{>[scale=1.5,length=4,
width=5]}]
\tikzstyle{Surjection}=[ColBlack,draw,-{>[scale=1.5,length=4,width=5]>[scale=1.5,length=4,
width=5]}]
\tikzstyle{Map}=[ColBlack,draw,-{>[scale=1.5,length=4,width=5]}]

\tikzstyle{Box}=[rectangle,MarkD,rounded corners=2pt,inner sep=1pt,minimum size=.25cm,thick]

\newcommand{\N}{\mathbb{N}}
\newcommand{\Z}{\mathbb{Z}}

\newcommand{\K}{\mathbb{K}}

\newcommand{\GenA}{\mathtt{a}}
\newcommand{\GenB}{\mathtt{b}}
\newcommand{\GenC}{\mathtt{c}}

\newcommand{\GenE}{\mathtt{e}}

\newcommand{\TreeS}{\mathfrak{s}}
\newcommand{\TreeR}{\mathfrak{r}}
\newcommand{\TreeT}{\mathfrak{t}}
\newcommand{\TreeU}{\mathfrak{u}}
\newcommand{\TreeV}{\mathfrak{v}}

\newcommand{\PosetP}{\mathcal{P}}
\newcommand{\PosetQ}{\mathcal{Q}}
\newcommand{\LatticeL}{\mathcal{L}}

\newcommand{\Operad}{\mathcal{O}}

\newcommand{\SeriesF}{\mathbf{f}}

\DeclareMathOperator{\Leq}{\preccurlyeq}
\DeclareMathOperator{\OrderPrefixes}{\Leq_{\mathrm{p}}}
\DeclareMathOperator{\JJoin}{\vee}
\DeclareMathOperator{\Meet}{\wedge}
\DeclareMathOperator{\Hadamard}{\boxtimes}

\newcommand{\Unit}{\mathbf{1}}
\newcommand{\GeneratingSet}{\mathfrak{G}}
\newcommand{\InternalNode}{\bullet}
\newcommand{\SyntaxTrees}{\mathbf{S}}
\newcommand{\SyntaxTreesLeaf}{\SyntaxTrees_\Leaf}
\newcommand{\SyntaxTreesInternalNode}{\SyntaxTrees_\InternalNode}
\newcommand{\Deg}{\mathrm{deg}}
\newcommand{\RankSeries}{\mathcal{R}}
\newcommand{\IntervalSeries}{\mathcal{I}}
\newcommand{\Congr}{\equiv}
\newcommand{\Eval}{\mathrm{ev}}
\newcommand{\Up}{\mathbf{U}}
\newcommand{\Vp}{\mathbf{V}}
\newcommand{\Dual}{\star}
\newcommand{\Nodes}{\mathcal{N}}
\newcommand{\InternalNodes}{\Nodes_\InternalNode}
\newcommand{\Leaves}{\Nodes_\Leaf}
\newcommand{\LeavesNonFirst}{\Leaves^{\mathrm{nf}}}
\newcommand{\InternalNodesMax}{\InternalNodes^{\mathrm{m}}}

\newcommand{\InternalNodesQuasiMax}{\InternalNodes^{\mathrm{qm}}}
\newcommand{\Shadow}{\mathrm{sh}}
\newcommand{\Load}{\mathrm{ld}}
\newcommand{\Rank}{\mathrm{rk}}
\newcommand{\Product}{\star}
\newcommand{\YoungLattice}{\mathbb{Y}}
\newcommand{\Support}[1]{\mathrm{Supp}\Par{#1}}
\newcommand{\CharacteristicSeries}[1]{\mathrm{ch}\Par{#1}}
\newcommand{\NonFirstLeaves}[1]{\mathrm{nfl}\Par{#1}}
\newcommand{\GeneratingSeries}{\mathcal{G}}
\newcommand{\HilbertSeries}{\mathcal{H}}
\newcommand{\Corolla}{\mathfrak{c}}
\newcommand{\Covered}{\lessdot}
\newcommand{\Weight}[2]{\omega_{#1}\Par{#2}}
\newcommand{\Hook}[1]{\mathrm{h}\Par{#1}}
\newcommand{\TwistedHook}[1]{\mathrm{h}'\Par{#1}}
\newcommand{\Zero}{\mathbf{0}}
\newcommand{\Paths}[1]{\mathcal{P}_{#1}}
\newcommand{\ShadowPoset}[1]{\mathbb{P}\Par{#1}}
\newcommand{\JLattice}[1]{\mathbb{J}\Par{#1}}
\newcommand{\TreePoset}[1]{\mathbb{P}\Par{#1}}
\newcommand{\TwistedTreePoset}[1]{\mathbb{P}'\Par{#1}}

\newcommand{\HookSeries}[1]{\mathbf{h}_{#1}}
\newcommand{\Trace}{\mathrm{tr}}
\newcommand{\InitialPathsSeries}[1]{\mathbf{ip}_{#1}}
\newcommand{\ReturningHookSeries}[2]{\mathbf{rh}_{#1, #2}}
\newcommand{\ReturningInitialPathsSeries}[2]{\mathbf{rip}_{#1, #2}}
\newcommand{\DelNode}[2]{\mathrm{del}_{#2}\Par{#1}}
\newcommand{\ContractNode}[2]{\mathrm{con}_{#2}\Par{#1}}
\newcommand{\TreeLikeExpressions}{\mathcal{T}}
\newcommand{\OperadDegree}{\Operad_\InternalNode}

\newcommand{\As}{\mathbf{As}}
\newcommand{\Dias}{\mathbf{Dias}}
\newcommand{\Comp}{\mathbf{Comp}}
\newcommand{\Motz}{\mathbf{Motz}}
\newcommand{\FCat}[1]{\mathbf{FCat}_{#1}}
\newcommand{\TwoAs}{\mathbf{2As}}
\newcommand{\AsTwo}{\mathbf{As}_2}
\newcommand{\Dup}{\mathbf{Dup}}
\newcommand{\Dip}{\mathbf{Dip}}

\newcommand{\Leaf}{{
\begin{tikzpicture}[Centering,xscale=.2,yscale=.22]
    \draw[Edge,thick](0,0)--(0,-1);
\end{tikzpicture}
}}

\newcommand{\CorollaOne}[1]{{
\begin{tikzpicture}[Centering,xscale=0.2,yscale=0.35]
    \node(1)at(0.00,-1.00){};
    \node[NodeST](0)at(0.00,0.00){$#1$};
    \draw[Edge](1)--(0);
    \node(r)at(0.00,1.0){};
    \draw[Edge](r)--(0);
\end{tikzpicture}}}

\newcommand{\CorollaTwo}[1]{{
\begin{tikzpicture}[Centering,xscale=0.16,yscale=0.24]
    \node(0)at(0.00,-1.50){};
    \node(2)at(2.00,-1.50){};
    \node[NodeST](1)at(1.00,0.00){$#1$};
    \draw[Edge](0)--(1);
    \draw[Edge](2)--(1);
    \node(r)at(1.00,1.5){};
    \draw[Edge](r)--(1);
\end{tikzpicture}
}}

\newcommand{\CorollaThree}[1]{{
\begin{tikzpicture}[Centering,xscale=0.19,yscale=0.19]
    \node(0)at(0.00,-2.00){};
    \node(2)at(1.00,-2.00){};
    \node(3)at(2.00,-2.00){};
    \node[NodeST](1)at(1.00,0.00){$#1$};
    \draw[Edge](0)--(1);
    \draw[Edge](2)--(1);
    \draw[Edge](3)--(1);
    \node(r)at(1.00,1.75){};
    \draw[Edge](r)--(1);
\end{tikzpicture}
}}

\title[Graph duality for operads]{Duality of graded graphs through operads}
\date{\today}
\author{Samuele Giraudo}
\address{\scriptsize LIGM, Université Gustave Eiffel, CNRS, ESIEE Paris, F-$77454$
Marne-la-Vallée, France.}
\email{samuele.giraudo@u-pem.fr}

\begin{document}

%%%%%%%%%%%%%%%%%%%%%%%%%%%%%%%%%%%%%%%%%%%%%%%%%%%%%%%%%%%%%%%%%%%%%%%%%%%%%%%%%%%%%%%%%%%%
%%%%%%%%%%%%%%%%%%%%%%%%%%%%%%%%%%%%%%%%%%%%%%%%%%%%%%%%%%%%%%%%%%%%%%%%%%%%%%%%%%%%%%%%%%%%
%%%%%%%%%%%%%%%%%%%%%%%%%%%%%%%%%%%%%%%%%%%%%%%%%%%%%%%%%%%%%%%%%%%%%%%%%%%%%%%%%%%%%%%%%%%%
\begin{abstract}
    Pairs of graded graphs, together with the Fomin property of graded graph duality, are
    rich combinatorial structures providing among other a framework for enumeration. The
    prototypical example is the one of the Young graded graph of integer partitions,
    allowing us to connect number of standard Young tableaux and numbers of permutations.
    Here, we use operads, that algebraic devices abstracting the notion of composition of
    combinatorial objects, to build pairs of graded graphs. For this, we first construct a
    pair of graded graphs where vertices are syntax trees, the elements of free nonsymmetric
    operads. This pair of graphs is dual for a new notion of duality called $\phi$-diagonal
    duality, similar to the ones introduced by Fomin. We also provide a general way to build
    pairs of graded graphs from operads, wherein underlying posets are analogous to the
    Young lattice. Some examples of operads leading to new pairs of graded graphs involving
    integer compositions, Motzkin paths, and $m$-trees are considered.
\end{abstract}

\maketitle

\begin{small}
\tableofcontents
\end{small}

%%%%%%%%%%%%%%%%%%%%%%%%%%%%%%%%%%%%%%%%%%%%%%%%%%%%%%%%%%%%%%%%%%%%%%%%%%%%%%%%%%%%%%%%%%%%
%%%%%%%%%%%%%%%%%%%%%%%%%%%%%%%%%%%%%%%%%%%%%%%%%%%%%%%%%%%%%%%%%%%%%%%%%%%%%%%%%%%%%%%%%%%%
%%%%%%%%%%%%%%%%%%%%%%%%%%%%%%%%%%%%%%%%%%%%%%%%%%%%%%%%%%%%%%%%%%%%%%%%%%%%%%%%%%%%%%%%%%%%
\section*{Introduction}
The well-known formula
\begin{equation} \label{equ:sum_square_Young_lattice}
    \sum_{\lambda \vdash n} {f_\lambda}^2 = n!,
\end{equation}
relating the numbers $f_\lambda$ of standard Young tableaux of shape $\lambda$ and number of
permutations is one of the most fascinating identities appearing in algebraic combinatorics.
This formula, admitting a lot of different proofs~\cite{Sag01}, arises in the context of
representations of symmetric groups and the Robinson–Schensted correspondence. One of its
proofs is surprisingly beautiful and uses the Young lattice $\YoungLattice$ on integer
partitions and its structure of a differential poset~\cite{Sta88}. Such a poset satisfies
the relation
\begin{equation} \label{equ:differential_posets_duality}
    \Up^\Dual \Up - \Up \Up^\Dual = I
\end{equation}
where $I$ is the identity map, and $\Up$ (resp. $\Up^\Dual$) is the linear map sending each
element of the poset to the formal sum of its coverings (resp. of the elements it covers).
One can interpret~\eqref{equ:sum_square_Young_lattice} as an identity between Hasse walks of
length $n$ in $\YoungLattice$ starting from the empty integer partition to an integer
partition of rank $n$ and returning to the empty integer partition.
\smallbreak

A natural question concerns the generalization of this concept of differential posets, in
order to obtain new combinatorial proofs similar to the previous formula or to discover new
ones. In this context, the notion of graded graph duality~\cite{Fom94} makes sense. Here, we
work not only with posets but with multigraphs wherein analogs
of~\eqref{equ:differential_posets_duality} hold between two different graphs, called pairs
of dual graphs. All this maintains close connections with algebra since, from the origins,
the Hasse diagram of $\YoungLattice$ is in fact the Bratteli diagram (or multiplication
graph) of the algebra of the symmetric functions on the basis of Schur functions
$s_\lambda$, where edges encode the multiplication by $s_1$. A striking and nice fact is
that one can construct similar graphs for other algebraic structures~\cite{Nze06,LS07} like
the algebra of the noncommutative symmetric functions or the Hopf algebra of planar binary
trees~\cite{HNT05}.
\smallbreak

The starting motivation of this work was to link this theory of duality of graded graphs
with the theory of operads, with the aim to construct new pairs of dual graded graphs and
explore the combinatorial properties they offer. Operads~\cite{Men15,LV12,Gir18} are
algebraic structures wherein elements are themselves operations, and can be composed. From a
combinatorial point of view, operads allow to {\em insert} combinatorial objects inside
other ones to form bigger ones~\cite{Gir15}. Since operads enclose a rich combinatorics, we
can expect that these structures are good combinatorial sources to build interesting graphs.
\smallbreak

We begin by constructing a pair $\Par{\SyntaxTreesInternalNode(\GeneratingSet), \Up, \Vp}$
of graphs where vertices are the elements of free nonsymmetric operads, that are planar
rooted trees decorated on an alphabet $\GeneratingSet$. The graphs
$\Par{\SyntaxTreesInternalNode(\GeneratingSet), \Up}$ and
$\Par{\SyntaxTreesInternalNode(\GeneratingSet), \Vp}$ are dual with respect to a new notion
of duality called $\phi$-diagonal duality, generalizing in a certain way some of the
previous ones. Also, the poset for which $\Par{\SyntaxTreesInternalNode(\GeneratingSet),
\Up}$ is the Hasse diagram has some combinatorial properties like to be a meet-semilattice
and to have all intervals that are distributive lattices. Then, given an operad
$\Operad$ satisfying some not so restrictive properties, we extend the previous construction
to build a pair of graded graphs $\Par{\OperadDegree, \Up, \Vp}$, potentially
$\phi$-diagonal dual. We consider four examples: the pair corresponding with the associative
operad (which is called the {\em chain} in~\cite{Fom94}), with the diassociative
operad~\cite{Lod01} (which is not $\phi$-diagonal dual but is $\phi$-diagonal self-dual),
with the operad of integer compositions~\cite{Gir15} (which leads to the {\em composition
poset} introduced in~\cite{BS05}), and with the operad of Motzkin paths~\cite{Gir15} (which
is $\phi$-diagonal dual).
\smallbreak

This paper is organized as follows. Section~\ref{sec:preliminaries} contains the preliminary
definitions used in the rest of the document including graded graphs, formal power series on
combinatorial objects, syntax trees, and nonsymmetric operads. We also set here our
definition of $\phi$-diagonal duality. We then introduce in
Section~\ref{sec:graded_graphs_syntax_trees} two graded graphs of syntax trees: the prefix
graded graph $\Par{\SyntaxTreesInternalNode(\GeneratingSet), \Up}$ and the twisted prefix
graded graph $\Par{\SyntaxTreesInternalNode(\GeneratingSet), \Vp}$. We prove here in
particular that the pair of graded graphs $\Par{\SyntaxTreesInternalNode(\GeneratingSet),
\Up, \Vp}$ is $\phi$-diagonal dual for a certain linear map $\phi$. Some combinatorial
properties of these graphs are established: the numbers of Hasse walks in the prefix graded
graphs are related with the hook-length formula of trees~\cite{Knu98} and the numbers of
Hasse walks in the twisted ones are related with a variation of this formula.  In
Section~\ref{sec:tree_prefix_posets}, we study the posets associated with the prefix graded
graphs. In particular, we describe the structure of the intervals of these posets.  Finally,
Section~\ref{sec:operad_prefix_graphs} is devoted to generalize the previous constructions
of graded graphs in order to obtain pairs of graded graphs from nonsymmetric operads
subjected to some conditions. We apply these constructions to some operads introduced
in~\cite{Gir15}, involving integer compositions, Motzkin paths, and $m$-trees.
\medbreak

%%%%%%%%%%%%%%%%%%%%%%%%%%%%%%%%%%%%%%%%%%%%%%%%%%%%%%%%%%%%%%%%%%%%%%%%%%%%%%%%%%%%%%%%%%%%
\subsubsection*{General notations and conventions}
All the considered vector spaces are defined over a ground field $\K$ of characteristic
zero. For any integers $i$ and $j$, $[i, j]$ denotes the set $\{i, i + 1, \dots, j\}$. For
any integer $i$, $[i]$ denotes the set $[1, i]$. The empty word is denoted by~$\epsilon$.
\medbreak

%%%%%%%%%%%%%%%%%%%%%%%%%%%%%%%%%%%%%%%%%%%%%%%%%%%%%%%%%%%%%%%%%%%%%%%%%%%%%%%%%%%%%%%%%%%%
%%%%%%%%%%%%%%%%%%%%%%%%%%%%%%%%%%%%%%%%%%%%%%%%%%%%%%%%%%%%%%%%%%%%%%%%%%%%%%%%%%%%%%%%%%%%
%%%%%%%%%%%%%%%%%%%%%%%%%%%%%%%%%%%%%%%%%%%%%%%%%%%%%%%%%%%%%%%%%%%%%%%%%%%%%%%%%%%%%%%%%%%%
\section{Graded graphs, trees, and operads} \label{sec:preliminaries}
We start by setting up our context by providing definitions about graded sets, series,
graded graphs, and nonsymmetric operads. We introduce also the notion of $\phi$-diagonal
duality.
\medbreak

%%%%%%%%%%%%%%%%%%%%%%%%%%%%%%%%%%%%%%%%%%%%%%%%%%%%%%%%%%%%%%%%%%%%%%%%%%%%%%%%%%%%%%%%%%%%
%%%%%%%%%%%%%%%%%%%%%%%%%%%%%%%%%%%%%%%%%%%%%%%%%%%%%%%%%%%%%%%%%%%%%%%%%%%%%%%%%%%%%%%%%%%%
\subsection{Graded graphs and diagonal duality}
The aim of this section is to make some recalls about graded graphs and their associated
formal power series, about graded graph duality, and to introduce $\phi$-diagonal duality.
\medbreak

%%%%%%%%%%%%%%%%%%%%%%%%%%%%%%%%%%%%%%%%%%%%%%%%%%%%%%%%%%%%%%%%%%%%%%%%%%%%%%%%%%%%%%%%%%%%
\subsubsection{Graded sets}
A \Def{graded set} is a set expressed as a disjoint union
\begin{equation}
    G := \bigsqcup_{d \in \N} G(d)
\end{equation}
such that all $G(d)$, $d \in \N$, are sets. The \Def{rank} $\Rank(x)$ of an $x \in G$ is the
unique integer $d$ such that $x \in G(d)$. A graded set is \Def{combinatorial} if all
$G(d)$, $d \in \N$, are finite.  In this case, the \Def{generating series}
$\RankSeries_G(t)$ of $G$ is defined by
\begin{equation}
    \RankSeries_G(t) := \sum_{x \in G} t^{\Rank(x)}
\end{equation}
and counts the elements of $G$ with respect to their ranks.  If $G_1$ and $G_2$ are two
graded sets, a map $\psi : G_1 \to G_2$ is a \Def{graded set morphism} if for any $x \in
G_1$, $\Rank\Par{\psi(x)} = \Rank(x)$. Besides $G_2$ is a \Def{graded subset} of $G_1$ if
for any $d \in \N$, $G_2(d) \subseteq G_1(d)$.
\medbreak

%%%%%%%%%%%%%%%%%%%%%%%%%%%%%%%%%%%%%%%%%%%%%%%%%%%%%%%%%%%%%%%%%%%%%%%%%%%%%%%%%%%%%%%%%%%%
\subsubsection{Polynomials and series}
We shall consider in the sequel linear spans of graded sets $G$, denoted by $\K \Angle{G}$.
The dual space $\K \AAngle{G}$ of $\K \Angle{G}$ is by definition the space of the maps
$\SeriesF : G \to \K$, called \Def{$G$-series}.  Let $\SeriesF_1$ and $\SeriesF_2$ be two
$G$-series.  The \Def{scalar product} of $\SeriesF_1$ and $\SeriesF_2$ is the element
\begin{equation}
    \Angle{\SeriesF_1, \SeriesF_2} 
    := \sum_{x \in G} \SeriesF_1(x) \SeriesF_2(x) 
\end{equation}
of $\K$. Note that the scalar product may be not defined for some $G$-series. For any subset
$X$ of $G$,  the \Def{characteristic series} of $X$ is the $G$-series
$\CharacteristicSeries{X}$ satisfying, for any $x \in G$, $\CharacteristicSeries{X}(x) =
\Han{x \in X}$, where $\Han{-}$ is the Iverson bracket. By a slight abuse of notation, we
denote simply by $x$ the $G$-series $\CharacteristicSeries{\{x\}}$.  Let $\SeriesF \in \K
\AAngle{G}$.  Observe that for any $x \in G$, $\Angle{x, \SeriesF} = \SeriesF(x)$.  The
\Def{support} of $\SeriesF$ is the set
\begin{math}
    \Support{\SeriesF} := \Bra{x \in G : \Angle{x, \SeriesF} \ne 0}.
\end{math}
An element $x$ of $G$ \Def{appears} in $\SeriesF$ if $x \in \Support{\SeriesF}$. By a slight
abuse of notation, this property is denoted by $x \in \SeriesF$. By exploiting the vector
space structure of $\K \AAngle{G}$, any $G$-series $\SeriesF$ expresses as
\begin{equation} \label{equ:definition_series_as_sums}
    \SeriesF = \sum_{x \in G} \Angle{x, \SeriesF} x.
\end{equation}
This notation using potentially infinite sums of elements of $G$ accompanied with
coefficients of $\K$ is common in the context of formal power series. In the sequel, we
shall define and handle some $G$-series using the
notation~\eqref{equ:definition_series_as_sums}.  A $G$-series having a finite support is a
\Def{$G$-polynomial}. The space $\K \Angle{G}$ can be seen as the subspace of $\K
\AAngle{G}$ consisting in all $G$-polynomials.  The \Def{Hadamard product} of two $G$-series
$\SeriesF_1$ and $\SeriesF_2$ is the series $\SeriesF_1 \Hadamard \SeriesF_2$ defined, for
any $x \in G$, by
\begin{math}
    \Angle{x, \SeriesF_1 \Hadamard \SeriesF_2}
    := \Angle{x, \SeriesF_1} \Angle{x, \SeriesF_2}.
\end{math}
\medbreak

The space of all generating series on one formal parameter $t$ is denoted by $\K
\AAngle{t}$. The \Def{trace} of a $G$-series $\SeriesF$ is the generating series
$\Trace(\SeriesF)$ of $\K \AAngle{t}$ defined by
\begin{equation}
    \Trace(\SeriesF) :=
    \sum_{x \in G} \Angle{x, \SeriesF} t^{\Rank(x)}.
\end{equation}
This series might be ill-defined when $G$ is not combinatorial. Observe that if $\SeriesF$
is the characteristic series of $G$, then $\Trace(\SeriesF)$ is the generating series
of~$G$.
\medbreak

%%%%%%%%%%%%%%%%%%%%%%%%%%%%%%%%%%%%%%%%%%%%%%%%%%%%%%%%%%%%%%%%%%%%%%%%%%%%%%%%%%%%%%%%%%%%
\subsubsection{Graded graphs}
A \Def{graded graph}~\cite{Fom94} is a pair $(G, \Up)$ where $G$ is a combinatorial graded
set of \Def{vertices} and $\Up : \K \Angle{G} \to \K \Angle{G}$ is a linear map such that
$\Up(x) \in \K \Angle{G(d + 1)}$ for any $x \in G(d)$. In the sequel, $I$ is the identity
map on $\K \Angle{G}$.
\medbreak

Given a pair $(x, y) \in G^2$, let us set $\Weight{\Up}{x, y} := \Angle{y, \Up(x)}$. We say
that $(x, y)$ is an \Def{edge} of $(G, \Up)$ if $\Weight{\Up}{x, y} \ne 0$. In this case the
\Def{weight} of this edge is $\Weight{\Up}{x, y}$. A \Def{path} from $x_1 \in G$ to $x_\ell
\in G$ is a sequence $\Par{x_1, \dots, x_\ell}$, $\ell \geq 1$, of vertices of $G$ such that
for any $i \in [\ell - 1]$, $\Par{x_i, x_{i + 1}}$ is an edge of $(G, \Up)$. The
\Def{length} of $\Par{x_1, \dots, x_\ell}$ is $\ell - 1$ and its \Def{weight} is
\begin{equation}
    \Weight{\Up}{x_1, \dots, x_\ell} :=
    \prod_{i \in [\ell - 1]} \Weight{\Up}{x_i, x_{i + 1}}.
\end{equation}
As a particular case, the weight of any path of length $0$ is~$1$. The set of all paths of
$(G, \Up)$ from $x$ to $y$ is denoted by $\Paths{\Up}(x, y)$.  When for all $(x, y) \in
G^2$, the coefficients $\Weight{\Up}{x, y}$ are nonnegative integers, $(G, \Up)$ is
\Def{natural}. In this case, one can interpret any edge $(x, y) \in G^2$ as a bunch of
$\Weight{\Up}{x, y}$ multi-edges from $x$ to $y$.  Hence, for any $x, y \in G$, the sum of
the weights of all paths from $x$ to $y$ can be interpreted as the number of multipaths from
$x$ to $y$.  When moreover all coefficients $\Weight{\Up}{x, y}$ belong to $\{0, 1\}$, $(G,
\Up)$ is \Def{simple}.  Besides, when there is an element $\Zero$ of $G$ such that for any
$x \in G$, there is a path from $\Zero$ to $x$, $(G, \Up)$ is \Def{rooted} and $\Zero$ is
the \Def{root} of the graded graph.  Observe that if $(G, \Up)$ is rooted, its root is
unique.  In this case, for any $x \in G$, an \Def{initial path} to $x$ is a path from
$\Zero$ to~$x$ in~$(G, \Up)$.
\medbreak

The \Def{poset} of $(G, \Up)$ is the poset $(G, \Leq)$ wherein $x \Leq y$ if there is a path
in $(G, \Up)$ from the vertex $x$ to the vertex $y$ of~$G$. The covering relation of this
poset is denoted by $\Covered_\Up$ and it satisfies, for any $x, y \in G$,  $x \Covered_\Up
y$ if and only if $y$ appears in~$\Up(x)$.
\medbreak

We shall draw graded graphs where edges are implicitly oriented from top to bottom. The
weight of an edge is written onto it, with the convention that undecorated edges have
weight~$1$. For instance, Figure~\ref{fig:young_graded_graph} shows the Young graded graph
$\YoungLattice$.
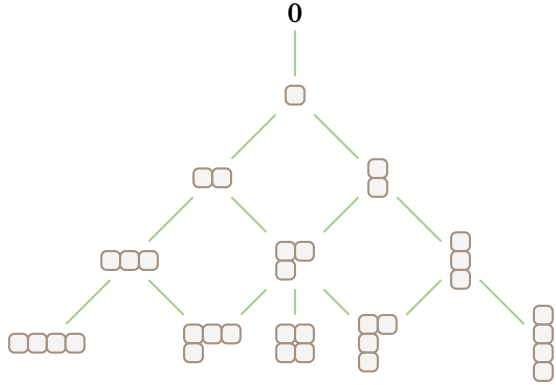
\begin{figure}[ht]
    \centering
    \begin{equation*}
        \begin{tikzpicture}[Centering,xscale=1.1,yscale=1.1]
            \node(Empty)at(0,0){$\Zero$};
            \node(1)at(0,-1){
                \begin{tikzpicture}[Centering]
                    \node[Box](1){};
                \end{tikzpicture}};
            \node(11)at(-1,-2){
                \begin{tikzpicture}[Centering]
                    \node[Box](11_1)at(0,0){};
                    \node[Box,right of=11_1,node distance=0.25cm](11_2){};
                \end{tikzpicture}};
            \node(2)at(1,-2){
                \begin{tikzpicture}[Centering]
                    \node[Box](2_1)at(0,0){};
                    \node[Box,below of=2_1,node distance=0.25cm](2_2){};
                \end{tikzpicture}};
            \node(111)at(-2,-3){
                \begin{tikzpicture}[Centering]
                    \node[Box](111_1)at(0,0){};
                    \node[Box,right of=111_1,node distance=0.25cm](111_2){};
                    \node[Box,right of=111_2,node distance=0.25cm](111_3){};
                \end{tikzpicture}};
            \node(21)at(0,-3){
                \begin{tikzpicture}[Centering]
                    \node[Box](21_1)at(0,0){};
                    \node[Box,right of=21_1,node distance=0.25cm](21_2){};
                    \node[Box,below of=21_1,node distance=0.25cm](21_3){};
                \end{tikzpicture}};
            \node(3)at(2,-3){
                \begin{tikzpicture}[Centering]
                    \node[Box](3_1)at(0,0){};
                    \node[Box,below of=3_1,node distance=0.25cm](3_2){};
                    \node[Box,below of=3_2,node distance=0.25cm](3_3){};
                \end{tikzpicture}};
            \node(1111)at(-3,-4){
                \begin{tikzpicture}[Centering]
                    \node[Box](1111_1)at(0,0){};
                    \node[Box,right of=1111_1,node distance=0.25cm](1111_2){};
                    \node[Box,right of=1111_2,node distance=0.25cm](1111_3){};
                    \node[Box,right of=1111_3,node distance=0.25cm](1111_4){};
                \end{tikzpicture}};
            \node(211)at(-1,-4){
                \begin{tikzpicture}[Centering]
                    \node[Box](211_1)at(0,0){};
                    \node[Box,right of=211_1,node distance=0.25cm](211_2){};
                    \node[Box,right of=211_2,node distance=0.25cm](211_3){};
                    \node[Box,below of=211_1,node distance=0.25cm](211_4){};
                \end{tikzpicture}};
            \node(22)at(0,-4){
                \begin{tikzpicture}[Centering]
                    \node[Box](22_1)at(0,0){};
                    \node[Box,right of=22_1,node distance=0.25cm](22_2){};
                    \node[Box,below of=22_1,node distance=0.25cm](22_3){};
                    \node[Box,below of=22_2,node distance=0.25cm](22_4){};
                \end{tikzpicture}};
            \node(31)at(1,-4){
                \begin{tikzpicture}[Centering]
                    \node[Box](31_1)at(0,0){};
                    \node[Box,right of=31_1,node distance=0.25cm](31_2){};
                    \node[Box,below of=31_1,node distance=0.25cm](31_3){};
                    \node[Box,below of=31_3,node distance=0.25cm](31_4){};
                \end{tikzpicture}};
            \node(4)at(3,-4){
                \begin{tikzpicture}[Centering]
                    \node[Box](4_1)at(0,0){};
                    \node[Box,below of=4_1,node distance=0.25cm](4_2){};
                    \node[Box,below of=4_2,node distance=0.25cm](4_3){};
                    \node[Box,below of=4_3,node distance=0.25cm](4_4){};
                \end{tikzpicture}};
            \draw[EdgeGraph](Empty)--(1);
            \draw[EdgeGraph](1)--(11);
            \draw[EdgeGraph](1)--(2);
            \draw[EdgeGraph](11)--(111);
            \draw[EdgeGraph](11)--(21);
            \draw[EdgeGraph](2)--(21);
            \draw[EdgeGraph](2)--(3);
            \draw[EdgeGraph](111)--(1111);
            \draw[EdgeGraph](111)--(211);
            \draw[EdgeGraph](21)--(211);
            \draw[EdgeGraph](21)--(22);
            \draw[EdgeGraph](21)--(31);
            \draw[EdgeGraph](3)--(31);
            \draw[EdgeGraph](3)--(4);
        \end{tikzpicture}
    \end{equation*}
    \caption{The Young graded graph up to integer partitions of size $4$.}
    \label{fig:young_graded_graph}
\end{figure}
The poset of $\YoungLattice$ is the Young lattice~\cite{Sta88}. Recall that in
$\YoungLattice$, vertices are integer partitions (represented as Young diagrams) and
$\Up(\lambda)$ is the sum of all partitions that can be obtained by adding one box to the
integer partition~$\lambda$.
\medbreak

Let
\begin{math}
    \Up^\Dual : \K \Angle{G}^\Dual \to \K \Angle{G}^\Dual
\end{math}
be the adjoint map of $\Up$. Due to the fact that $G$ is combinatorial and $\K \Angle{G}$ is
a graded space decomposing as
\begin{equation}
    \K \Angle{G} = \bigoplus_{d \in \N} \K \Angle{G(d)}
\end{equation}
with finite dimensional homogeneous components $\K \Angle{G(d)}$, $d \geq 0$, the space $\K
\Angle{G}$ can be identified with its graded dual $\K \Angle{G}^\Dual$. Therefore, for any
$y \in G$,
\begin{equation}
    \Up^\Dual(y)
    = \sum_{x \in G} \Angle{x, \Up(y)} x
    = \sum_{x \in G} \Weight{\Up}{x, y} x.
\end{equation}
\medbreak

In the case where $(G, \Up)$ is rooted, the \Def{hook series} of $(G, \Up)$ is the
$G$-series $\HookSeries{\Up}$ defined by the functional equation
\begin{equation}
    \Angle{x, \HookSeries{\Up}}
    = \Han{x = \Zero} + \Angle{\Up^\Dual(x), \HookSeries{\Up}}.
\end{equation}
For any $x \in G$, $\Angle{x, \HookSeries{\Up}}$ is the \Def{hook coefficient} of $x$ in
$(G, \Up)$.
\medbreak

\begin{Proposition} \label{prop:hook_series}
    Let $(G, \Up)$ be a rooted graded graph. For any $x \in G$,
    \begin{equation} \label{equ:hook_series}
        \Angle{x, \HookSeries{\Up}}
        = \sum_{p \in \Paths{\Up}(\Zero, x)} \Weight{\Up}{p}.
    \end{equation}
    Moreover, $\HookSeries{\Up} = (I - \Up)^{-1}(\Zero)$.
\end{Proposition}
\begin{proof}
    We proceed by induction on the rank $d$ of $x$. When $d = 0$, since $(G, \Up)$ is
    rooted, we necessarily have $x = \Zero$. Therefore, since $\Han{\Zero = \Zero} = 1$ and
    $\Up^\Dual(\Zero) = 0$, the property is satisfied in this case. Otherwise, $x \ne \Zero$
    and we have by definition of hook series,
    \begin{equation} \begin{split} \label{equ:hook_series_proof}
        \Angle{x, \HookSeries{\Up}} 
        & =
        \Angle{\Up^\Dual(x), \HookSeries{\Up}} \\
        & =
        \Angle{\sum_{y \in G} \Angle{x, \Up(y)} y, \HookSeries{\Up}} \\
        & =
        \sum_{y \in G} \Angle{x, \Up(y)} \Angle{y, \HookSeries{\Up}}.
    \end{split} \end{equation}
    Now, for any $y \in G$, if $x$ appears in $\Up(y)$, then the rank of $y$ is $d - 1$, and
    by induction hypothesis, $\Angle{y, \HookSeries{\Up}}$
    satisfies~\eqref{equ:hook_series}. Since all paths from $\Zero$ to $x$ in $(G, \Up)$
    decompose as paths from $\Zero$ to elements $y$ of rank $d - 1$ followed by edges from
    $y$ to $x$, the statement of the proposition follows.
    \smallbreak

    Let us establish the second part of the statement. For this, we prove that for any $x
    \in G$,
    \begin{equation}
        \Angle{x, \Up^{\Rank(x)}(\Zero)} = \Angle{x, \HookSeries{\Up}}
    \end{equation}
    by induction on the rank $d$ of $x$. If $d = 0$, since $(G, \Up)$ is rooted, $x =
    \Zero$, and since $\Angle{\Zero, \Up^0(\Zero)} = 1$ and $\Angle{\Zero, \HookSeries{\Up}}
    = 1$, the property is satisfied. Assume that $d \geq 1$.
    By~\eqref{equ:hook_series_proof} and by induction hypothesis,
    \begin{equation} \begin{split}
        \Angle{x, \HookSeries{\Up}}
        & =
        \sum_{y \in G} \Angle{x, \Up(y)} \Angle{y, \HookSeries{\Up}} \\
        & =
        \sum_{y \in G} \Angle{x, \Up(y)} \Angle{y, \Up^{\Rank(y)}(\Zero)} \\
        & =
        \Angle{x, \Up^{\Rank(x)}(\Zero)}.
    \end{split} \end{equation}
    Since finally $(I - \Up)^{-1} = \sum_{d \in \N} \Up^d$, the stated expression for
    $\HookSeries{\Up}$ follows.
\end{proof}
\medbreak

When $(G, \Up)$ is moreover natural, Proposition~\ref{prop:hook_series} says that the hook
coefficient of any $x \in G$ can be interpreted as the number of initial multipaths to $x$
in $(G, \Up)$.  These coefficients define a statistics on the elements of $G$ which can be
of independent combinatorial interest. For instance, the hook coefficient of a partition
$\lambda$ in $\YoungLattice$ is given by the hook-length formula~\cite{FRT54}, is also the
number of initial paths to $\lambda$, and is also the number of standard Young tableaux of
shape $\lambda$. Therefore, a standard Young tableau of shape $\lambda$ is to the integer
partition $\lambda$ what an initial path to $x \in G$ is to $x$ in the case where $(G, \Up)$
is a natural rooted graded graph.  Moreover, the \Def{initial paths series} of $(G, \Up)$ is
the generating series $\InitialPathsSeries{\Up} := \Trace\Par{\HookSeries{\Up}}$.  By
definition, for any $\ell \in \N$, the coefficient $\Angle{t^\ell,
\InitialPathsSeries{\Up}}$ can be interpreted as the number of initial multipaths of $(G,
\Up)$ of length $\ell$. In the case of $\YoungLattice$, we obtain the generating series
counting the standard Young tableaux as initial path series (see Sequence~\OEIS{A000085}
of~\cite{Slo} for its coefficients).
\medbreak

%%%%%%%%%%%%%%%%%%%%%%%%%%%%%%%%%%%%%%%%%%%%%%%%%%%%%%%%%%%%%%%%%%%%%%%%%%%%%%%%%%%%%%%%%%%%
\subsubsection{Pairs of graded graphs}
A \Def{pair of graded graphs} is a triple $(G, \Up, \Vp)$ such that both $(G, \Up)$ and $(G,
\Vp)$ are graded graphs. When $(G, \Up)$ and $(G, \Vp)$ are both natural, $(G, \Up, \Vp)$ is
\Def{natural}.  When $(G, \Up)$ and $(G, \Vp)$ are both rooted and share the same root
$\Zero$, $(G, \Up, \Vp)$ is \Def{rooted} and $\Zero$ is the \Def{root} of  $(G, \Up, \Vp)$.
A \Def{returning path} from $x \in G$ to $y \in G$ is a pair $\Par{p, p'}$ such that $p$ is
a path from $x$ to $y$ in $(G, \Up)$ and $p'$ is a path from $x$ to $y$ in $(G, \Vp)$.
The \Def{length} of $\Par{p, p'}$ is the length of $p$ (or equivalently, of $p'$) and its
\Def{weight} is
\begin{equation}
    \Weight{\Up, \Vp}{p, p'} := \Weight{\Up}{p} \Weight{\Vp}{p'}.
\end{equation}
When $(G, \Up, \Vp)$ is rooted, a \Def{returning initial path} to $x$ is a returning path
from $\Zero$ to~$x$ in~$(G, \Up, \Vp)$ and the \Def{returning hook series} of $(G, \Up,
\Vp)$ is the $G$-series $\ReturningHookSeries{\Up}{\Vp}$ defined by
\begin{equation}
    \ReturningHookSeries{\Up}{\Vp}
    :=
    \HookSeries{\Up} \Hadamard \HookSeries{\Vp}.
\end{equation}
For any $x \in G$, $\Angle{x, \ReturningHookSeries{\Up}{\Vp}}$ is the \Def{returning hook
coefficient} of $x$ in $(G, \Up, \Vp)$.
\medbreak

\begin{Proposition} \label{prop:returning_hook_series}
    Let $(G, \Up, \Vp)$ be a rooted pair of graded graphs. For any $x \in G$,
    \begin{equation}
        \Angle{x, \ReturningHookSeries{\Up}{\Vp}}
        =
        \sum_{\substack{
            p \in \Paths{\Up}(\Zero, x) \\
            p' \in \Paths{\Vp}(\Zero, x) \\
        }}
        \Weight{\Up, \Vp}{p, p'}.
    \end{equation}
\end{Proposition}
\begin{proof}
    By definition of returning hook series,
    \begin{equation}
        \Angle{x, \ReturningHookSeries{\Up}{\Vp}}
        = \Angle{x, \HookSeries{\Up} \Hadamard \HookSeries{\Vp}}
        = \Angle{x, \HookSeries{\Up}} \Angle{x, \HookSeries{\Vp}}.
    \end{equation}
    By Proposition~\ref{prop:hook_series}, the statement of the proposition follows.
\end{proof}
\medbreak

When $(G, \Up, \Vp)$ is moreover natural, Proposition~\ref{prop:returning_hook_series} says
that the returning hook coefficient of any $x \in G$ can be interpreted as the number of
returning initial paths to $x$ in $(G, \Up, \Vp)$. These coefficients define a statistics on
the elements of $G$ which can be of independent combinatorial interest.  For instance, by
seeing $\YoungLattice$ as a pair of graded graphs with $\Vp = \Up$, the returning hook
coefficient of a partition $\lambda$ in $\YoungLattice$ is given by the square of the
hook-length formula.  Therefore, a pair of standard Young tableaux of the same shape
$\lambda$ is to the integer partition $\lambda$  what a returning initial path to $x \in G$
is to $x$ in the case where $(G, \Up, \Vp)$ is a natural rooted pair of graded graphs.
Moreover, the \Def{returning initial paths series} of $(G, \Up, \Vp)$ is the generating
series
\begin{math}
    \ReturningInitialPathsSeries{\Up}{\Vp} := \Trace\Par{\ReturningHookSeries{\Up}{\Vp}}.
\end{math}
By definition, for any $\ell \in \N$, the coefficient of $\Angle{t^\ell,
\ReturningInitialPathsSeries{\Up}{\Vp}}$ can be interpreted as the number of returning
initial multipaths of $(G, \Up, \Vp)$ of length $\ell$.  In the case of $\YoungLattice$, we
obtain the generating series counting the permutations as returning initial paths series
(see~\eqref{equ:sum_square_Young_lattice} and~\cite{Sag01}).
\medbreak

%%%%%%%%%%%%%%%%%%%%%%%%%%%%%%%%%%%%%%%%%%%%%%%%%%%%%%%%%%%%%%%%%%%%%%%%%%%%%%%%%%%%%%%%%%%%
\subsubsection{Dual graded graphs}
Let $(G, \Up, \Vp)$ be a pair of graded graphs.  One says that $(G, \Up, \Vp)$ is
\Def{$\phi$-diagonal dual} if $\phi : \K \Angle{G} \to \K \Angle{G}$ is a diagonal linear
map that is, for any $x \in G$, $\phi(x) = \lambda_x x$ where $\lambda_x \in \K$, and
\begin{equation} \label{equ:phi_diagonal_duality}
    \Vp^\Dual \Up - \Up \Vp^\Dual = \phi.
\end{equation}
This notion is a generalization of $r_d$-duality, and hence, of $r$-duality and duality of
graded graphs (see~\cite{Fom94}). Indeed, in the case where $\phi(x) = r_{\Rank(x)} \; x$
for any $x \in G$, one recovers $r_d$-duality. Duality of graded graphs is very closely
connected with the theory of $r$-differential posets~\cite{Sta88}. In the case where
$(G, \Up, \Up)$ is $\phi$-diagonal dual, we say that $(G, \Up)$ is \Def{$\phi$-diagonal
self-dual}.
\medbreak

\begin{Proposition} \label{prop:relation_U_V_phi_diagonal_duality}
    Let $(G, \Up, \Vp)$ be a pair of $\phi$-diagonal dual graded graphs.  For any $n \geq
    0$,
    \begin{equation} \label{equ:relation_U_V_phi_diagonal_duality}
        \Vp^\Dual \Up^n
        = \Up^n \Vp^\Dual
        + \sum_{\substack{k_1, k_2 \geq 0 \\ k_1 + k_2 = n - 1}} \Up^{k_1} \phi \Up^{k_2}.
    \end{equation}
\end{Proposition}
\begin{proof}
    We proceed by induction on $n$. When $n = 0$, the property holds since the left-hand
    side of~\eqref{equ:relation_U_V_phi_diagonal_duality} is
    \begin{math}
        \Vp^\Dual \Up^0 = \Vp^\Dual I = \Vp^\Dual
    \end{math}
    while its right-hand side is
    \begin{math}
        \Up^0 \Vp^\Dual = I \Vp^\Dual = \Vp^\Dual.
    \end{math}
    Assume that the property holds for a $n \geq 0$. Hence, by induction hypothesis and by
    using Relation~\eqref{equ:phi_diagonal_duality} implied by of $\phi$-diagonal duality of
    $(G, \Up, \Vp)$, we have
    \begin{equation} \begin{split}
        \Vp^\Dual \Up^{n + 1}
        & =
        \Vp^\Dual \Up^n \Up \\
        & =
        \Up^n \Vp^\Dual \Up +
        \sum_{\substack{k_1, k_2 \geq 0 \\ k_1 + k_2 = n - 1}}
        \Up^{k_1} \phi \Up^{k_2 + 1} \\
        & =
        \Up^n \Par{\phi +  \Up \Vp^\Dual} +
        \sum_{\substack{k_1, k_2 \geq 0 \\ k_1 + k_2 = n - 1}}
        \Up^{k_1} \phi \Up^{k_2 + 1} \\
        & =
        \Up^{n + 1} \Vp^\Dual + \Up^n \phi +
        \sum_{\substack{k_1, k_2 \geq 0 \\ k_1 + k_2 = n - 1}}
        \Up^{k_1} \phi \Up^{k_2 + 1} \\
        & =
        \Up^{n + 1} \Vp^\Dual +
        \sum_{\substack{k_1, k_2 \geq 0 \\ k_1 + k_2 = n}}
        \Up^{k_1} \phi \Up^{k_2},
    \end{split} \end{equation}
    establishing~\eqref{equ:relation_U_V_phi_diagonal_duality}.
\end{proof}
\medbreak

Observe that when $(G, \Up, \Vp)$ is a pair of $\phi$-diagonal dual graded graphs such that
$\phi$ commutes with $\Up$, there exists an $r \in \K$ such that the map $\phi$ satisfies
$\phi(x) = r x$ for any $x \in G$.  This implies that in this case, $(G, \Up, \Vp)$ is a
pair of $r$-dual graphs and Proposition~\ref{prop:relation_U_V_phi_diagonal_duality} brings
us the well-known identity~\cite{Sta88}
\begin{equation}
    \Vp^\Dual \Up^n = \Up^n \Vp^\Dual + n r \Up^{n - 1}
\end{equation}
holding for any $n \geq 0$.
\medbreak

%%%%%%%%%%%%%%%%%%%%%%%%%%%%%%%%%%%%%%%%%%%%%%%%%%%%%%%%%%%%%%%%%%%%%%%%%%%%%%%%%%%%%%%%%%%%
%%%%%%%%%%%%%%%%%%%%%%%%%%%%%%%%%%%%%%%%%%%%%%%%%%%%%%%%%%%%%%%%%%%%%%%%%%%%%%%%%%%%%%%%%%%%
\subsection{Syntax trees}
We set here elementary definitions and notations about syntax trees and composition
operations on syntax trees. Most of these notions can be found in~\cite[Chapter 3.]{Gir18}.
\medbreak

%%%%%%%%%%%%%%%%%%%%%%%%%%%%%%%%%%%%%%%%%%%%%%%%%%%%%%%%%%%%%%%%%%%%%%%%%%%%%%%%%%%%%%%%%%%%
\subsubsection{Elementary definitions} \label{subsubsec:elementary_definitions_syntax_trees}
An \Def{alphabet} is a graded set $\GeneratingSet$ such that $\GeneratingSet(0) =
\emptyset$. The elements of $\GeneratingSet$ are \Def{letters}. The \Def{arity} $|\GenA|$ of
a letter $\GenA \in \GeneratingSet$ is its rank.  A \Def{$\GeneratingSet$-tree} (also called
\Def{$\GeneratingSet$-syntax tree}) is a planar rooted tree such that its internal nodes of
arity $k$ are decorated by letters of arity $k$ of $\GeneratingSet$. More precisely, a
$\GeneratingSet$-tree is either the \Def{leaf} $\Leaf$ or a pair $\Par{\GenA, \Par{\TreeT_1,
\dots, \TreeT_{|\GenA|}}}$ where $\GenA \in \GeneratingSet$ and $\TreeT_1$, \dots,
$\TreeT_{|\GenA|]}$ are $\GeneratingSet$-trees.  Unless otherwise specified, we use in the
sequel the standard terminology (such as \Def{node}, \Def{internal node}, \Def{leaf},
\Def{edge}, \Def{root}, \Def{child}, \Def{ancestor}, {\em etc.}) about planar rooted
trees~\cite{Knu97} (see also~\cite{Gir18}). Let us set here the most important definitions
employed in this work.
\medbreak

Let $\TreeT = \Par{\GenA, \Par{\TreeT_1, \dots, \TreeT_{|\GenA|}}}$ be a
$\GeneratingSet$-tree. For any word $u$ of positive integers, let $u \mapsto \TreeT(u)$ be
the partial map defined recursively as follows.
\begin{enumerate}[fullwidth,label={\em (\roman*)}]
    \item If $u = \epsilon$, then $\TreeT(u) := \TreeT$;
    \item If $u = u_1 u_2 \dots u_k$ with $k \geq 1$ and $u_1 \in [|\GenA|]$, then
    $\TreeT(u) := \TreeT_{u_1}\Par{u_2 \dots u_k}$;
    \item Otherwise, $\TreeT(u)$ is not defined.
\end{enumerate}
A \Def{node} of $\TreeT$ is any word $u$ of positive integers such that $\TreeT(u)$ is
well-defined. In this case, $\TreeT(u)$ is the \Def{$u$-suffix subtree} of $\TreeT$.
Moreover, for any $i \in [k]$ where $k$ is the arity of the node $u$ in $\TreeT$,
$\TreeT(ui)$ is the \Def{$i$-th subtree} of $u$ in $\TreeT$.  A node $u$ of $\TreeT$ is
\Def{internal} if $u$ is a proper prefix of an other node of $\TreeT$. A \Def{leaf} of
$\TreeT$ is a node of $\TreeT$ which is not internal. We denote by $\Nodes(\TreeT)$ (resp.
$\InternalNodes(\TreeT)$, $\Leaves(\TreeT)$) the set of all nodes (resp. internal nodes,
leaves) of $\TreeT$.
\medbreak

The \Def{degree} $\Deg(\TreeT)$ (resp. \Def{arity} $|\TreeT|$) of $\TreeT$ is its number of
internal nodes (resp. leaves). The only $\GeneratingSet$-tree of degree $0$ and arity $1$ is
the \Def{leaf} and is denoted by $\Leaf$.  For any $\GenA \in \GeneratingSet(k)$, the
\Def{corolla} decorated by $\GenA$ is the tree $\Corolla(\GenA)$ consisting in one internal
node decorated by $\GenA$ having as children $k$ leaves. The leaves of $\TreeT$ are totally
ordered by their position in $\Leaves(\TreeT)$ with respect to the lexicographic order. They
are thus implicitly indexed from $1$ to $|\TreeT|$.
\medbreak

For instance, if
\begin{math}
    \GeneratingSet :=  \GeneratingSet(2) \sqcup \GeneratingSet(3)
\end{math}
with $\GeneratingSet(2) := \{\GenA, \GenB\}$ and $\GeneratingSet(3) := \{\GenC\}$,
\begin{equation}
    \TreeT :=
    \begin{tikzpicture}[Centering,xscale=.26,yscale=.15]
        \node(0)at(0.00,-6.50){};
        \node(10)at(8.00,-9.75){};
        \node(12)at(10.00,-9.75){};
        \node(2)at(2.00,-6.50){};
        \node(4)at(3.00,-3.25){};
        \node(5)at(4.00,-9.75){};
        \node(7)at(5.00,-9.75){};
        \node(8)at(6.00,-9.75){};
        \node[NodeST](1)at(1.00,-3.25){$\GenB$};
        \node[NodeST](11)at(9.00,-6.50){$\GenA$};
        \node[NodeST](3)at(3.00,0.00){$\GenC$};
        \node[NodeST](6)at(5.00,-6.50){$\GenC$};
        \node[NodeST](9)at(7.00,-3.25){$\GenA$};
        \node(r)at(3.00,2.75){};
        \draw[Edge](0)--(1);
        \draw[Edge](1)--(3);
        \draw[Edge](10)--(11);
        \draw[Edge](11)--(9);
        \draw[Edge](12)--(11);
        \draw[Edge](2)--(1);
        \draw[Edge](4)--(3);
        \draw[Edge](5)--(6);
        \draw[Edge](6)--(9);
        \draw[Edge](7)--(6);
        \draw[Edge](8)--(6);
        \draw[Edge](9)--(3);
        \draw[Edge](r)--(3);
    \end{tikzpicture}
\end{equation}
is a $\GeneratingSet$-tree of degree $5$ and arity $8$. Its root is decorated by $\GenC$ and
has arity $3$. Moreover, we have
\begin{equation}
    \TreeT(1) =
    \begin{tikzpicture}[Centering,xscale=.25,yscale=.29]
        \node(0)at(0.00,-1.50){};
        \node(2)at(2.00,-1.50){};
        \node[NodeST](1)at(1.00,0.00){$\GenB$};
        \draw[Edge](0)--(1);
        \draw[Edge](2)--(1);
        \node(r)at(1.00,1.5){};
        \draw[Edge](r)--(1);
    \end{tikzpicture}
    = \Corolla(\GenB),
    \qquad
    \TreeT(2) = \Leaf,
    \qquad
    \TreeT(3) =
    \begin{tikzpicture}[Centering,xscale=.24,yscale=.17]
        \node(0)at(0.00,-5.33){};
        \node(2)at(1.00,-5.33){};
        \node(3)at(2.00,-5.33){};
        \node(5)at(4.00,-5.33){};
        \node(7)at(6.00,-5.33){};
        \node[NodeST](1)at(1.00,-2.67){$\GenC$};
        \node[NodeST](4)at(3.00,0.00){$\GenA$};
        \node[NodeST](6)at(5.00,-2.67){$\GenA$};
        \draw[Edge](0)--(1);
        \draw[Edge](1)--(4);
        \draw[Edge](2)--(1);
        \draw[Edge](3)--(1);
        \draw[Edge](5)--(6);
        \draw[Edge](6)--(4);
        \draw[Edge](7)--(6);
        \node(r)at(3.00,2.50){};
        \draw[Edge](r)--(4);
    \end{tikzpicture},
    \qquad
    \TreeT(32) =
    \begin{tikzpicture}[Centering,xscale=.25,yscale=.29]
        \node(0)at(0.00,-1.50){};
        \node(2)at(2.00,-1.50){};
        \node[NodeST](1)at(1.00,0.00){$\GenA$};
        \draw[Edge](0)--(1);
        \draw[Edge](2)--(1);
        \node(r)at(1.00,1.5){};
        \draw[Edge](r)--(1);
    \end{tikzpicture}
    = \Corolla(\GenA),
\end{equation}
and
\begin{math}
    \Nodes(\TreeT) = \{\epsilon, 1, 11, 12, 2, 3, 31, 311, 312, 313, 32, 321, 322\},
\end{math}
\begin{math}
    \InternalNodes(\TreeT) = \{\epsilon, 1, 3, 31, 32\},
\end{math}
and
\begin{math}
    \Leaves(\TreeT) = \{11, 12, 2, 311, 312, 313, 321, 322\}.
\end{math}
\medbreak

%%%%%%%%%%%%%%%%%%%%%%%%%%%%%%%%%%%%%%%%%%%%%%%%%%%%%%%%%%%%%%%%%%%%%%%%%%%%%%%%%%%%%%%%%%%%
\subsubsection{Graded sets of syntax trees and partial compositions}
Given an alphabet $\GeneratingSet$, we denote by $\SyntaxTreesLeaf(\GeneratingSet)$ (resp.
$\SyntaxTreesInternalNode(\GeneratingSet)$) the graded set of all the $\GeneratingSet$-trees
where the rank of a tree is its arity (resp. its degree).  When $\GeneratingSet$ is finite,
the graded set $\SyntaxTreesInternalNode(\GeneratingSet)$ is combinatorial and its
generating series $\RankSeries_{\SyntaxTreesInternalNode(\GeneratingSet)}(t)$, counting its
elements with respect to their degrees, satisfies
\begin{equation}
    \RankSeries_{\SyntaxTreesInternalNode(\GeneratingSet)}(t)
    = 1 + t \RankSeries_{\GeneratingSet}\Par{
        \RankSeries_{\SyntaxTreesInternalNode(\GeneratingSet)}(t)}.
\end{equation}
\medbreak

Given $\TreeT, \TreeS \in \SyntaxTreesInternalNode(\GeneratingSet)$ and $i \in [|\TreeT|]$,
the \Def{partial composition} $\TreeT \circ_i \TreeS$ is the $\GeneratingSet$-tree obtained
by grafting the root of $\TreeS$ onto the $i$-th leaf of $\TreeT$. For instance, by
considering the previous graded set $\GeneratingSet$ of
Section~\ref{subsubsec:elementary_definitions_syntax_trees}, one has
\begin{equation} \label{equ:example_partial_composition_trees}
    \begin{tikzpicture}[Centering,xscale=.2,yscale=.19]
        \node(0)at(0.00,-5.00){};
        \node(2)at(2.00,-7.50){};
        \node(4)at(4.00,-7.50){};
        \node(6)at(6.00,-5.00){};
        \node(8)at(7.00,-5.00){};
        \node(9)at(8.00,-5.00){};
        \node[NodeST](1)at(1.00,-2.50){$\GenB$};
        \node[NodeST](3)at(3.00,-5.00){$\GenA$};
        \node[NodeST](5)at(5.00,0.00){$\GenA$};
        \node[NodeST](7)at(7.00,-2.50){$\GenC$};
        \draw[Edge](0)--(1);
        \draw[Edge](1)--(5);
        \draw[Edge](2)--(3);
        \draw[Edge](3)--(1);
        \draw[Edge](4)--(3);
        \draw[Edge](6)--(7);
        \draw[Edge](7)--(5);
        \draw[Edge](8)--(7);
        \draw[Edge](9)--(7);
        \node(r)at(5.00,2){};
        \draw[Edge](r)--(5);
    \end{tikzpicture}
    \enspace \circ_5 \enspace
    \begin{tikzpicture}[Centering,xscale=.25,yscale=.21]
        \node(0)at(0.00,-2.00){};
        \node(2)at(1.00,-2.00){};
        \node(3)at(2.00,-4.00){};
        \node(5)at(4.00,-4.00){};
        \node[NodeST](1)at(1.00,0.00){$\GenC$};
        \node[NodeST](4)at(3.00,-2.00){$\GenB$};
        \draw[Edge](0)--(1);
        \draw[Edge](2)--(1);
        \draw[Edge](3)--(4);
        \draw[Edge](4)--(1);
        \draw[Edge](5)--(4);
        \node(r)at(1.00,2){};
        \draw[Edge](r)--(1);
    \end{tikzpicture}
    \enspace = \enspace
    \begin{tikzpicture}[Centering,xscale=.22,yscale=.16]
        \node(0)at(0.00,-6.00){};
        \node(10)at(8.00,-9.00){};
        \node(11)at(9.00,-12.00){};
        \node(13)at(11.00,-12.00){};
        \node(14)at(12.00,-6.00){};
        \node(2)at(2.00,-9.00){};
        \node(4)at(4.00,-9.00){};
        \node(6)at(6.00,-6.00){};
        \node(8)at(7.00,-9.00){};
        \node[NodeST](1)at(1.00,-3.00){$\GenB$};
        \node[NodeST](12)at(10.00,-9.00){$\GenB$};
        \node[NodeST](3)at(3.00,-6.00){$\GenA$};
        \node[NodeST](5)at(5.00,0.00){$\GenA$};
        \node[NodeST](7)at(9.00,-3.00){$\GenC$};
        \node[NodeST](9)at(8.00,-6.00){$\GenC$};
        \draw[Edge](0)--(1);
        \draw[Edge](1)--(5);
        \draw[Edge](10)--(9);
        \draw[Edge](11)--(12);
        \draw[Edge](12)--(9);
        \draw[Edge](13)--(12);
        \draw[Edge](14)--(7);
        \draw[Edge](2)--(3);
        \draw[Edge](3)--(1);
        \draw[Edge](4)--(3);
        \draw[Edge](6)--(7);
        \draw[Edge](7)--(5);
        \draw[Edge](8)--(9);
        \draw[Edge](9)--(7);
        \node(r)at(5.00,2.5){};
        \draw[Edge](r)--(5);
    \end{tikzpicture}.
\end{equation}
Moreover, for any $u \in \Leaves(\TreeT)$, we shall denote by $\TreeT \circ^u \TreeS$ the
$\GeneratingSet$-tree obtained by grafting the root of $\TreeS$ into the leaf $u$ of
$\TreeT$. For instance, the partial composition shown
in~\eqref{equ:example_partial_composition_trees} is the same as the one obtained by
composing the two considered trees through $\circ^{22}$ since the $5$-th leaf of the first
tree is~$22$.
\medbreak

By a slight but convenient abuse of notation, given $\GenA, \GenB \in \GeneratingSet$ and $i
\in [|\GenA|]$, we shall in some cases simply write $\GenA \circ_i \GenB$ instead of
$\Corolla(\GenA) \circ_i \Corolla(\GenB)$.  Moreover, when the context is clear, we shall
even write $\GenA$ for~$\Corolla(\GenA)$.  In addition, given some $\GeneratingSet$-trees
$\TreeS_1$, \dots, $\TreeS_{|\GenA|}$, we shall write $\GenA \BPar{\TreeS_1, \dots,
\TreeS_{|\GenA|}}$ instead of $\Par{\GenA, \Par{\TreeS_1, \dots, \TreeS_{|\GenA|}}}$.
\medbreak

%%%%%%%%%%%%%%%%%%%%%%%%%%%%%%%%%%%%%%%%%%%%%%%%%%%%%%%%%%%%%%%%%%%%%%%%%%%%%%%%%%%%%%%%%%%%
%%%%%%%%%%%%%%%%%%%%%%%%%%%%%%%%%%%%%%%%%%%%%%%%%%%%%%%%%%%%%%%%%%%%%%%%%%%%%%%%%%%%%%%%%%%%
\subsection{Nonsymmetric operads}
We set here elementary definitions and notations about nonsymmetric operads, free
nonsymmetric operads, and presentations by generators and relations. Most of these notions
can be found in~\cite[Chapter~5.]{Gir18} or in~\cite{Men15}.
\medbreak

%%%%%%%%%%%%%%%%%%%%%%%%%%%%%%%%%%%%%%%%%%%%%%%%%%%%%%%%%%%%%%%%%%%%%%%%%%%%%%%%%%%%%%%%%%%%
\subsubsection{Elementary definitions}
A \Def{nonsymmetric operad in the category of sets}, or a \Def{nonsymmetric operad} for
short, is a graded set $\Operad$ together with maps
\begin{equation}
    \circ_i : \Operad(n) \times \Operad(m) \to \Operad(n + m - 1),
    \qquad 1 \leq i \leq n, \enspace 1 \leq m,
\end{equation}
called \Def{partial compositions}, and a distinguished element $\Unit \in \Operad(1)$, the
\Def{unit} of $\Operad$. This data has to satisfy, for any $x, y, z \in \Operad$, the three
relations
\begin{subequations}
\begin{equation} \label{equ:operad_axiom_1}
    \Par{x \circ_i y} \circ_{i + j - 1} z = x \circ_i \Par{y \circ_j z},
    \qquad i \in [|x|], \enspace j \in [|y|],
\end{equation}
\begin{equation} \label{equ:operad_axiom_2}
    \Par{x \circ_i y} \circ_{j + |y| - 1} z = \Par{x \circ_j z} \circ_i y,
    \qquad i, j \in [|x|], \enspace i < j,
\end{equation}
\begin{equation} \label{equ:operad_axiom_3}
    \Unit \circ_1 x = x = x \circ_i \Unit,
    \qquad i \in [|x|].
\end{equation}
\end{subequations}
Since we consider in this work only nonsymmetric operads, we shall call these simply
\Def{operads}. The \Def{arity} $|x|$ of any $x \in \Operad$ is its rank.  An operad
$\Operad$ is \Def{combinatorial} if $\Operad$ is combinatorial as a graded set.
\medbreak

Given an operad $\Operad$, one defines the \Def{full composition} maps of $\Operad$ as the
maps
\begin{equation}
    \circ :
    \Operad(n) \times \Operad\Par{m_1} \times \dots \times \Operad\Par{m_n}
    \to \Operad\Par{m_1 + \dots + m_n},
    \qquad 1 \leq n, \enspace 1 \leq m_1, \dots, 1 \leq m_n,
\end{equation}
defined, for any $x \in \Operad(n)$ and $y_1, \dots, y_n \in \Operad$, by
\begin{equation} \label{equ:full_composition_maps}
    x \circ \Han{y_1, \dots, y_n}
    := \Par{\dots \Par{\Par{x \circ_n y_n} \circ_{n - 1} y_{n - 1}} \dots} \circ_1 y_1.
\end{equation}
\medbreak

When $\Operad$ is combinatorial, the \Def{Hilbert series} $\HilbertSeries_{\Operad}(t)$ of
$\Operad$ is the generating series $\GeneratingSeries_{\Operad}(t)$. If $\Operad_1$ and
$\Operad_2$ are two operads, a graded set morphism $\psi : \Operad_1 \to \Operad_2$ is an
\Def{operad morphism} if it sends the unit of $\Operad_1$ to the unit of $\Operad_2$ and
commutes with partial composition maps. We say that $\Operad_2$ is a \Def{suboperad} of
$\Operad_1$ if $\Operad_2$ is a graded subset of $\Operad_1$, $\Operad_1$ and $\Operad_2$
have the same unit, and the partial compositions of $\Operad_2$ are the ones of $\Operad_1$
restricted on $\Operad_2$. For any subset $\GeneratingSet$ of $\Operad$, the \Def{operad
generated} by $\GeneratingSet$ is the smallest suboperad $\Operad^{\GeneratingSet}$ of
$\Operad$ containing $\GeneratingSet$.  When $\Operad^{\GeneratingSet} = \Operad$ and
$\GeneratingSet$ is minimal with respect to the inclusion among the subsets of
$\GeneratingSet$ satisfying this property, $\GeneratingSet$ is a \Def{minimal generating
set} of $\Operad$ and its elements are \Def{generators} of $\Operad$. An \Def{operad
congruence} of $\Operad$ is an equivalence relation $\Congr$ respecting the arities and such
that, for any $x, y, x', y' \in \Operad$, $x \Congr x'$ and $y \Congr y'$ implies $x \circ_i
y \Congr x' \circ_i y'$ for any $i \in [|x|]$.  The $\Congr$-equivalence class of any $x \in
\Operad$ is denoted by $[x]_{\Congr}$. Given an operad congruence $\Congr$, the
\Def{quotient operad} $\Operad/_{\Congr}$ is the operad on the set of all
$\Congr$-equivalence classes and defined in the usual way.
\medbreak

%%%%%%%%%%%%%%%%%%%%%%%%%%%%%%%%%%%%%%%%%%%%%%%%%%%%%%%%%%%%%%%%%%%%%%%%%%%%%%%%%%%%%%%%%%%%
\subsubsection{Free operads} \label{subsubsec:free_operads}
Let $\GeneratingSet$ be an alphabet. The \Def{free operad} on $\GeneratingSet$ is the operad
defined on the graded set $\SyntaxTreesLeaf(\GeneratingSet)$ wherein the partial
compositions $\circ_i$ are the partial compositions of $\GeneratingSet$-trees.  By
considering the previous graded set $\GeneratingSet$ of
Section~\ref{subsubsec:elementary_definitions_syntax_trees}, one has, as an example of a
full composition in $\SyntaxTreesLeaf(\GeneratingSet)$,
\begin{equation}
    \begin{tikzpicture}[Centering,xscale=.25,yscale=.26]
        \node(0)at(0.00,-1.67){};
        \node(2)at(2.00,-3.33){};
        \node(4)at(4.00,-3.33){};
        \node[NodeST](1)at(1.00,0.00){$\GenB$};
        \node[NodeST](3)at(3.00,-1.67){$\GenA$};
        \draw[Edge](0)--(1);
        \draw[Edge](2)--(3);
        \draw[Edge](3)--(1);
        \draw[Edge](4)--(3);
        \node(r)at(1.00,1.75){};
        \draw[Edge](r)--(1);
    \end{tikzpicture}
    \enspace \circ \enspace
    \Han{
    \begin{tikzpicture}[Centering,xscale=.21,yscale=.18]
        \node(0)at(0.00,-4.67){};
        \node(2)at(2.00,-4.67){};
        \node(4)at(4.00,-4.67){};
        \node(6)at(6.00,-4.67){};
        \node[NodeST](1)at(1.00,-2.33){$\GenA$};
        \node[NodeST](3)at(3.00,0.00){$\GenA$};
        \node[NodeST](5)at(5.00,-2.33){$\GenB$};
        \draw[Edge](0)--(1);
        \draw[Edge](1)--(3);
        \draw[Edge](2)--(1);
        \draw[Edge](4)--(5);
        \draw[Edge](5)--(3);
        \draw[Edge](6)--(5);
        \node(r)at(3.00,2){};
        \draw[Edge](r)--(3);
    \end{tikzpicture},
    \Leaf,
    \begin{tikzpicture}[Centering,xscale=.24,yscale=.23]
        \node(0)at(0.00,-2.00){};
        \node(2)at(1.00,-2.00){};
        \node(3)at(2.00,-2.00){};
        \node[NodeST](1)at(1.00,0.00){$\GenC$};
        \draw[Edge](0)--(1);
        \draw[Edge](2)--(1);
        \draw[Edge](3)--(1);
        \node(r)at(1.00,1.75){};
        \draw[Edge](r)--(1);
    \end{tikzpicture}}
    \enspace = \enspace
    \begin{tikzpicture}[Centering,xscale=.2,yscale=.14]
        \node(0)at(0.00,-10.50){};
        \node(10)at(10.00,-10.50){};
        \node(12)at(11.00,-10.50){};
        \node(13)at(12.00,-10.50){};
        \node(2)at(2.00,-10.50){};
        \node(4)at(4.00,-10.50){};
        \node(6)at(6.00,-10.50){};
        \node(8)at(8.00,-7.00){};
        \node[NodeST](1)at(1.00,-7.00){$\GenA$};
        \node[NodeST](11)at(11.00,-7.00){$\GenC$};
        \node[NodeST](3)at(3.00,-3.50){$\GenA$};
        \node[NodeST](5)at(5.00,-7.00){$\GenB$};
        \node[NodeST](7)at(7.00,0.00){$\GenB$};
        \node[NodeST](9)at(9.00,-3.50){$\GenA$};
        \draw[Edge](0)--(1);
        \draw[Edge](1)--(3);
        \draw[Edge](10)--(11);
        \draw[Edge](11)--(9);
        \draw[Edge](12)--(11);
        \draw[Edge](13)--(11);
        \draw[Edge](2)--(1);
        \draw[Edge](3)--(7);
        \draw[Edge](4)--(5);
        \draw[Edge](5)--(3);
        \draw[Edge](6)--(5);
        \draw[Edge](8)--(9);
        \draw[Edge](9)--(7);
        \node(r)at(7.00,3){};
        \draw[Edge](r)--(7);
    \end{tikzpicture}.
\end{equation}
When $\GeneratingSet$ is combinatorial and satisfies $\GeneratingSet(1) = \emptyset$, the
Hilbert series $\HilbertSeries_{\SyntaxTreesLeaf(\GeneratingSet)}(t)$ satisfies
\begin{equation}
    \HilbertSeries_{\SyntaxTreesLeaf(\GeneratingSet)}(t)
    = t
    + \RankSeries_\GeneratingSet\Par{\HilbertSeries_{\SyntaxTreesLeaf(\GeneratingSet)}(t)}.
\end{equation}
\medbreak

Free operads satisfy the following universality property. The free operad
$\SyntaxTreesLeaf(\GeneratingSet)$ is the unique operad (up to isomorphism) such that for
any operad $\Operad$ and any graded set morphism  $f : \GeneratingSet \to \Operad$, there
exists a unique operad morphism $\psi : \SyntaxTreesLeaf(\GeneratingSet) \to \Operad$ such
that the factorization $f = \psi \circ \Corolla$ holds. In other terms, the diagram
\begin{equation}
    \begin{tikzpicture}[Centering,xscale=1.3,yscale=1.1]
        \node(G)at(0,0){\begin{math}\GeneratingSet\end{math}};
        \node(O)at(2,0){$\Operad$};
        \node(AG)at(0,-2){$\SyntaxTreesLeaf(\GeneratingSet)$};
        \draw[Map](G)--(O)node[midway,above]{$f$};
        \draw[Injection](G)--(AG)node[midway,left]{$\Corolla$};
        \draw[Map,dashed](AG)--(O)node[midway,right]{$\psi$};
    \end{tikzpicture}
\end{equation}
commutes.
\medbreak

%%%%%%%%%%%%%%%%%%%%%%%%%%%%%%%%%%%%%%%%%%%%%%%%%%%%%%%%%%%%%%%%%%%%%%%%%%%%%%%%%%%%%%%%%%%%
\subsubsection{Presentations and treelike expressions}
A \Def{presentation} of an operad $\Operad$ is a pair $(\GeneratingSet, \Congr)$ such that
$\GeneratingSet$ is an alphabet, $\Congr$ is an operad congruence of
$\SyntaxTreesLeaf(\GeneratingSet)$, and $\Operad$ is isomorphic to
$\SyntaxTreesLeaf(\GeneratingSet)/_{\Congr}$. In most of the practical cases,
$\GeneratingSet$ is a subset of $\Operad$ such that $\GeneratingSet$ is a minimal generating
set of~$\Operad$.
\medbreak

When $\Operad$ satisfies $\Operad(0) = \emptyset$, $\Operad$ is in particular an alphabet.
For this reason, $\SyntaxTreesLeaf(\Operad)$ is a well-defined free operad. The
\Def{evaluation map} of $\Operad$ is the map $\Eval : \SyntaxTreesLeaf(\Operad) \to \Operad$
defined as the unique surjective operad morphism satisfying, for any $x \in \Operad$,
$\Eval(\Corolla(x)) = x$. Given $S$ a subset of $\Operad$, a \Def{treelike expression} on
$S$ of an element $x$ of $\Operad$ is a $\Operad$-tree $\TreeT$ belonging to the fiber
$\Eval^{-1}(x)$ and such that all internal nodes of $\TreeT$ are decorated by elements of~$S$.
The set of all treelike expressions on $S$ of $x$ is denoted by~$\TreeLikeExpressions_S(x)$.
\medbreak

%%%%%%%%%%%%%%%%%%%%%%%%%%%%%%%%%%%%%%%%%%%%%%%%%%%%%%%%%%%%%%%%%%%%%%%%%%%%%%%%%%%%%%%%%%%%
\subsubsection{Linear operads} \label{subsubsec:linear_operads}
The partial and the full composition operations of an operad $\Operad$ extend by linearity
on the space $\K \Angle{\Operad}$. This fact will be used implicitly in the sequel.
Moreover, it is convenient in what follows, when $x \in \Operad(n)$ and $y \in \Operad$, to
set $x \circ_i y := 0$ whenever $i \notin [n]$. This convention will be used also implicitly
in the sequel.  Besides, when $\Operad$ is in particular the free operad on
$\GeneratingSet$, by a slight abuse of notation, for any $\GenA \in \GeneratingSet$ and
$\SyntaxTreesLeaf(\GeneratingSet)$-polynomials $f_1$, \dots, $f_{|\GenA|}$, we shall write
$\GenA \BPar{f_1, \dots, f_{|\GenA|}}$ for $\Corolla(\GenA) \circ \Han{f_1, \dots,
f_{|\GenA|}}$.
\medbreak

%%%%%%%%%%%%%%%%%%%%%%%%%%%%%%%%%%%%%%%%%%%%%%%%%%%%%%%%%%%%%%%%%%%%%%%%%%%%%%%%%%%%%%%%%%%%
%%%%%%%%%%%%%%%%%%%%%%%%%%%%%%%%%%%%%%%%%%%%%%%%%%%%%%%%%%%%%%%%%%%%%%%%%%%%%%%%%%%%%%%%%%%%
%%%%%%%%%%%%%%%%%%%%%%%%%%%%%%%%%%%%%%%%%%%%%%%%%%%%%%%%%%%%%%%%%%%%%%%%%%%%%%%%%%%%%%%%%%%%
\section{Graded graphs of syntax trees} \label{sec:graded_graphs_syntax_trees}
The objective of this section is to introduce two graded graphs of syntax trees which are
$\phi$-diagonal dual for a certain map $\phi$. These graphs will be used as raw material in
the next sections in order to associate pairs of graded graphs with operads.
\medbreak

%%%%%%%%%%%%%%%%%%%%%%%%%%%%%%%%%%%%%%%%%%%%%%%%%%%%%%%%%%%%%%%%%%%%%%%%%%%%%%%%%%%%%%%%%%%%
%%%%%%%%%%%%%%%%%%%%%%%%%%%%%%%%%%%%%%%%%%%%%%%%%%%%%%%%%%%%%%%%%%%%%%%%%%%%%%%%%%%%%%%%%%%%
\subsection{Prefix graded graphs} \label{subsec:prefix_graded_graphs}
We begin by introducing prefix graded graphs and present some of their combinatorial
properties.
\medbreak

%%%%%%%%%%%%%%%%%%%%%%%%%%%%%%%%%%%%%%%%%%%%%%%%%%%%%%%%%%%%%%%%%%%%%%%%%%%%%%%%%%%%%%%%%%%%
\subsubsection{First definitions and properties}
\label{subsubsec:definitions_prefix_graded_graphs}
For any finite alphabet $\GeneratingSet$, let
$\Par{\SyntaxTreesInternalNode(\GeneratingSet), \Up}$ be the graded graph wherein, for any
$\TreeT \in \SyntaxTreesInternalNode(\GeneratingSet)$,
\begin{equation} \label{equ:definition_up_trees}
    \Up(\TreeT) :=
    \sum_{\substack{
        \GenA \in \GeneratingSet \\
        i \in [|\TreeT|]
    }}
    \TreeT \circ_i \GenA.
\end{equation}
We call $\Par{\SyntaxTreesInternalNode(\GeneratingSet), \Up}$ the
\Def{$\GeneratingSet$-prefix graph}. Since $\GeneratingSet$ is finite, $\Up(\TreeT)$ is a
$\SyntaxTreesInternalNode(\GeneratingSet)$-polynomial. Moreover, since all trees appearing
in $\Up(\TreeT)$ are of rank $\Deg(\TreeT) + 1$, the $\GeneratingSet$-prefix graph is a
well-defined graded graph. Observe that this graded graph is simple and that it admits
$\Leaf$ as root.  Figure~\ref{fig:examples_tree_graphs} shows examples of such graphs.
\begin{figure}[ht]
    \centering
    \subfloat[][For
    \begin{math}
        \GeneratingSet = \Bra{\GenA}
    \end{math}
    with $|\GenA| = 2$ up to degree $3$.]{
    \centering
    \begin{tikzpicture}[Centering,xscale=1.7,yscale=1.45]
        \node(Unit)at(0,0){$\Leaf$};
        \node(a00)at(0,-.75){$\CorollaTwo{\GenA}$};
        \node(aa000)at(-1,-1.5){
        \begin{tikzpicture}[Centering,xscale=0.17,yscale=0.21]
            \node(0)at(0.00,-3.33){};
            \node(2)at(2.00,-3.33){};
            \node(4)at(4.00,-1.67){};
            \node[NodeST](1)at(1.00,-1.67){$\GenA$};
            \node[NodeST](3)at(3.00,0.00){$\GenA$};
            \draw[Edge](0)--(1);
            \draw[Edge](1)--(3);
            \draw[Edge](2)--(1);
            \draw[Edge](4)--(3);
            \node(r)at(3.00,1.75){};
            \draw[Edge](r)--(3);
        \end{tikzpicture}};
        \node(a0a00)at(1,-1.5){
            \begin{tikzpicture}[Centering,xscale=0.17,yscale=0.21]
            \node(0)at(0.00,-1.67){};
            \node(2)at(2.00,-3.33){};
            \node(4)at(4.00,-3.33){};
            \node[NodeST](1)at(1.00,0.00){$\GenA$};
            \node[NodeST](3)at(3.00,-1.67){$\GenA$};
            \draw[Edge](0)--(1);
            \draw[Edge](2)--(3);
            \draw[Edge](3)--(1);
            \draw[Edge](4)--(3);
            \node(r)at(1.00,1.75){};
            \draw[Edge](r)--(1);
        \end{tikzpicture}};
        \node(aaa0000)at(-2,-2.75){
        \begin{tikzpicture}[Centering,xscale=0.14,yscale=0.2]
            \node(0)at(0.00,-5.25){};
            \node(2)at(2.00,-5.25){};
            \node(4)at(4.00,-3.50){};
            \node(6)at(6.00,-1.75){};
            \node[NodeST](1)at(1.00,-3.50){$\GenA$};
            \node[NodeST](3)at(3.00,-1.75){$\GenA$};
            \node[NodeST](5)at(5.00,0.00){$\GenA$};
            \draw[Edge](0)--(1);
            \draw[Edge](1)--(3);
            \draw[Edge](2)--(1);
            \draw[Edge](3)--(5);
            \draw[Edge](4)--(3);
            \draw[Edge](6)--(5);
            \node(r)at(5.00,1.75){};
            \draw[Edge](r)--(5);
        \end{tikzpicture}};
        \node(aa0a000)at(-1,-2.75){
        \begin{tikzpicture}[Centering,xscale=0.14,yscale=0.2]
            \node(0)at(0.00,-3.50){};
            \node(2)at(2.00,-5.25){};
            \node(4)at(4.00,-5.25){};
            \node(6)at(6.00,-1.75){};
            \node[NodeST](1)at(1.00,-1.75){$\GenA$};
            \node[NodeST](3)at(3.00,-3.50){$\GenA$};
            \node[NodeST](5)at(5.00,0.00){$\GenA$};
            \draw[Edge](0)--(1);
            \draw[Edge](1)--(5);
            \draw[Edge](2)--(3);
            \draw[Edge](3)--(1);
            \draw[Edge](4)--(3);
            \draw[Edge](6)--(5);
            \node(r)at(5.00,1.75){};
            \draw[Edge](r)--(5);
        \end{tikzpicture}};
        \node(aa00a00)at(0,-2.75){
            \begin{tikzpicture}[Centering,xscale=0.14,yscale=0.17]
            \node(0)at(0.00,-4.67){};
            \node(2)at(2.00,-4.67){};
            \node(4)at(4.00,-4.67){};
            \node(6)at(6.00,-4.67){};
            \node[NodeST](1)at(1.00,-2.33){$\GenA$};
            \node[NodeST](3)at(3.00,0.00){$\GenA$};
            \node[NodeST](5)at(5.00,-2.33){$\GenA$};
            \draw[Edge](0)--(1);
            \draw[Edge](1)--(3);
            \draw[Edge](2)--(1);
            \draw[Edge](4)--(5);
            \draw[Edge](5)--(3);
            \draw[Edge](6)--(5);
            \node(r)at(3.00,2){};
            \draw[Edge](r)--(3);
        \end{tikzpicture}};
        \node(a0aa000)at(1,-2.75){
            \begin{tikzpicture}[Centering,xscale=0.14,yscale=0.2]
            \node(0)at(0.00,-1.75){};
            \node(2)at(2.00,-5.25){};
            \node(4)at(4.00,-5.25){};
            \node(6)at(6.00,-3.50){};
            \node[NodeST](1)at(1.00,0.00){$\GenA$};
            \node[NodeST](3)at(3.00,-3.50){$\GenA$};
            \node[NodeST](5)at(5.00,-1.75){$\GenA$};
            \draw[Edge](0)--(1);
            \draw[Edge](2)--(3);
            \draw[Edge](3)--(5);
            \draw[Edge](4)--(3);
            \draw[Edge](5)--(1);
            \draw[Edge](6)--(5);
            \node(r)at(1.00,1.75){};
            \draw[Edge](r)--(1);
        \end{tikzpicture}};
        \node(a0a0a00)at(2,-2.75){
        \begin{tikzpicture}[Centering,xscale=0.14,yscale=0.2]
            \node(0)at(0.00,-1.75){};
            \node(2)at(2.00,-3.50){};
            \node(4)at(4.00,-5.25){};
            \node(6)at(6.00,-5.25){};
            \node[NodeST](1)at(1.00,0.00){$\GenA$};
            \node[NodeST](3)at(3.00,-1.75){$\GenA$};
            \node[NodeST](5)at(5.00,-3.50){$\GenA$};
            \draw[Edge](0)--(1);
            \draw[Edge](2)--(3);
            \draw[Edge](3)--(1);
            \draw[Edge](4)--(5);
            \draw[Edge](5)--(3);
            \draw[Edge](6)--(5);
            \node(r)at(1.00,1.75){};
            \draw[Edge](r)--(1);
        \end{tikzpicture}};
        \draw[EdgeGraph](Unit)--(a00);
        \draw[EdgeGraph](a00)--(aa000);
        \draw[EdgeGraph](a00)--(a0a00);
        \draw[EdgeGraph](aa000)--(aaa0000);
        \draw[EdgeGraph](aa000)--(aa0a000);
        \draw[EdgeGraph](aa000)--(aa00a00);
        \draw[EdgeGraph](a0a00)--(aa00a00);
        \draw[EdgeGraph](a0a00)--(a0aa000);
        \draw[EdgeGraph](a0a00)--(a0a0a00);
    \end{tikzpicture}
    \label{subfig:tree_graph_a2}}
    \medbreak
    \subfloat[][For
    \begin{math}
        \GeneratingSet = \Bra{\GenE, \GenC}
    \end{math}
    with $|\GenE| = 1$ and $|\GenC| = 3$ up to degree $2$ and with some trees of
    degree~$3$.]{
    \centering
    \begin{tikzpicture}[Centering,xscale=1.5,yscale=1.5]
        \node(Unit)at(.5,0){$\Leaf$};
        \node(e0)at(-1.5, -.5){$\CorollaOne{\GenE}$};
        \node(c000)at(2.5, -.5){$\CorollaThree{\GenC}$};
        \node(ee0)at(-2, -1.6){
        \begin{tikzpicture}[Centering,xscale=0.17,yscale=0.29]
            \node(2)at(0.00,-2.25){};
            \node[NodeST](0)at(0.00,0.00){$\GenE$};
            \node[NodeST](1)at(0.00,-1.00){$\GenE$};
            \draw[Edge](1)--(0);
            \draw[Edge](2)--(1);
            \node(r)at(0.00,1.25){};
            \draw[Edge](r)--(0);
        \end{tikzpicture}};
        \node(ec000)at(-1, -1.6){
        \begin{tikzpicture}[Centering,xscale=0.18,yscale=0.24]
            \node(1)at(0.00,-3.33){};
            \node(3)at(1.00,-3.33){};
            \node(4)at(2.00,-3.33){};
            \node[NodeST](0)at(1.00,0.00){$\GenE$};
            \node[NodeST](2)at(1.00,-1.67){$\GenC$};
            \draw[Edge](1)--(2);
            \draw[Edge](2)--(0);
            \draw[Edge](3)--(2);
            \draw[Edge](4)--(2);
            \node(r)at(1.00,1.5){};
            \draw[Edge](r)--(0);
        \end{tikzpicture}};
        \node(ce000)at(0, -1.6){
        \begin{tikzpicture}[Centering,xscale=0.29,yscale=0.24]
            \node(1)at(0.00,-3.33){};
            \node(3)at(1.00,-1.67){};
            \node(4)at(2.00,-1.67){};
            \node[NodeST](0)at(0.00,-1.67){$\GenE$};
            \node[NodeST](2)at(1.00,0.00){$\GenC$};
            \draw[Edge](0)--(2);
            \draw[Edge](1)--(0);
            \draw[Edge](3)--(2);
            \draw[Edge](4)--(2);
            \node(r)at(1.00,1.75){};
            \draw[Edge](r)--(2);
        \end{tikzpicture}};
        \node(c0e00)at(1, -1.6){
        \begin{tikzpicture}[Centering,xscale=0.29,yscale=0.24]
            \node(0)at(0.00,-1.67){};
            \node(3)at(1.00,-3.33){};
            \node(4)at(2.00,-1.67){};
            \node[NodeST](1)at(1.00,0.00){$\GenC$};
            \node[NodeST](2)at(1.00,-1.67){$\GenE$};
            \draw[Edge](0)--(1);
            \draw[Edge](2)--(1);
            \draw[Edge](3)--(2);
            \draw[Edge](4)--(1);
            \node(r)at(1.00,1.75){};
            \draw[Edge](r)--(1);
        \end{tikzpicture}};
        \node(c00e0)at(2, -1.6){
        \begin{tikzpicture}[Centering,xscale=0.29,yscale=0.24]
            \node(0)at(0.00,-1.67){};
            \node(2)at(1.00,-1.67){};
            \node(4)at(2.00,-3.33){};
            \node[NodeST](1)at(1.00,0.00){$\GenC$};
            \node[NodeST](3)at(2.00,-1.67){$\GenE$};
            \draw[Edge](0)--(1);
            \draw[Edge](2)--(1);
            \draw[Edge](3)--(1);
            \draw[Edge](4)--(3);
            \node(r)at(1.00,1.75){};
            \draw[Edge](r)--(1);
        \end{tikzpicture}};
        \node(cc00000)at(3, -1.6){
        \begin{tikzpicture}[Centering,xscale=0.19,yscale=0.17]
            \node(0)at(0.00,-4.67){};
            \node(2)at(1.00,-4.67){};
            \node(3)at(2.00,-4.67){};
            \node(5)at(3.00,-2.33){};
            \node(6)at(4.00,-2.33){};
            \node[NodeST](1)at(1.00,-2.33){$\GenC$};
            \node[NodeST](4)at(3.00,0.00){$\GenC$};
            \draw[Edge](0)--(1);
            \draw[Edge](1)--(4);
            \draw[Edge](2)--(1);
            \draw[Edge](3)--(1);
            \draw[Edge](5)--(4);
            \draw[Edge](6)--(4);
            \node(r)at(3.00,2.25){};
            \draw[Edge](r)--(4);
        \end{tikzpicture}};
        \node(c0c0000)at(4, -1.6){
        \begin{tikzpicture}[Centering,xscale=0.19,yscale=0.17]
            \node(0)at(0.00,-2.33){};
            \node(2)at(1.00,-4.67){};
            \node(4)at(2.00,-4.67){};
            \node(5)at(3.00,-4.67){};
            \node(6)at(4.00,-2.33){};
            \node[NodeST](1)at(2.00,0.00){$\GenC$};
            \node[NodeST](3)at(2.00,-2.33){$\GenC$};
            \draw[Edge](0)--(1);
            \draw[Edge](2)--(3);
            \draw[Edge](3)--(1);
            \draw[Edge](4)--(3);
            \draw[Edge](5)--(3);
            \draw[Edge](6)--(1);
            \node(r)at(2.00,2.25){};
            \draw[Edge](r)--(1);
        \end{tikzpicture}};
        \node(c00c000)at(5, -1.6){
        \begin{tikzpicture}[Centering,xscale=0.19,yscale=0.17]
            \node(0)at(0.00,-2.33){};
            \node(2)at(1.00,-2.33){};
            \node(3)at(2.00,-4.67){};
            \node(5)at(3.00,-4.67){};
            \node(6)at(4.00,-4.67){};
            \node[NodeST](1)at(1.00,0.00){$\GenC$};
            \node[NodeST](4)at(3.00,-2.33){$\GenC$};
            \draw[Edge](0)--(1);
            \draw[Edge](2)--(1);
            \draw[Edge](3)--(4);
            \draw[Edge](4)--(1);
            \draw[Edge](5)--(4);
            \draw[Edge](6)--(4);
            \node(r)at(1.00,2.25){};
            \draw[Edge](r)--(1);
        \end{tikzpicture}};
        \node(ce0e00)at(0, -2.8){
        \begin{tikzpicture}[Centering,xscale=0.36,yscale=0.2]
            \node(1)at(0.00,-4.00){};
            \node(4)at(1.00,-4.00){};
            \node(5)at(2.00,-2.00){};
            \node[NodeST](0)at(0.00,-2.00){$\GenE$};
            \node[NodeST](2)at(1.00,0.00){$\GenC$};
            \node[NodeST](3)at(1.00,-2.00){$\GenE$};
            \draw[Edge](0)--(2);
            \draw[Edge](1)--(0);
            \draw[Edge](3)--(2);
            \draw[Edge](4)--(3);
            \draw[Edge](5)--(2);
            \node(r)at(1.00,2){};
            \draw[Edge](r)--(2);
        \end{tikzpicture}};
        \node(ce00e0)at(1, -2.8){
        \begin{tikzpicture}[Centering,xscale=0.24,yscale=0.18]
            \node(1)at(0.00,-4.00){};
            \node(3)at(1.00,-2.00){};
            \node(5)at(2.00,-4.00){};
            \node[NodeST](0)at(0.00,-2.00){$\GenE$};
            \node[NodeST](2)at(1.00,0.00){$\GenC$};
            \node[NodeST](4)at(2.00,-2.00){$\GenE$};
            \draw[Edge](0)--(2);
            \draw[Edge](1)--(0);
            \draw[Edge](3)--(2);
            \draw[Edge](4)--(2);
            \draw[Edge](5)--(4);
            \node(r)at(1.00,2){};
            \draw[Edge](r)--(2);
        \end{tikzpicture}};
        \node(c0e0e0)at(2, -2.8){
        \begin{tikzpicture}[Centering,xscale=0.36,yscale=0.2]
            \node(0)at(0.00,-2.00){};
            \node(3)at(1.00,-4.00){};
            \node(5)at(2.00,-4.00){};
            \node[NodeST](1)at(1.00,0.00){$\GenC$};
            \node[NodeST](2)at(1.00,-2.00){$\GenE$};
            \node[NodeST](4)at(2.00,-2.00){$\GenE$};
            \draw[Edge](0)--(1);
            \draw[Edge](2)--(1);
            \draw[Edge](3)--(2);
            \draw[Edge](4)--(1);
            \draw[Edge](5)--(4);
            \node(r)at(1.00,2){};
            \draw[Edge](r)--(1);
        \end{tikzpicture}};
        \draw[EdgeGraph](Unit)--(e0);
        \draw[EdgeGraph](Unit)--(c000);
        \draw[EdgeGraph](e0)--(ee0);
        \draw[EdgeGraph](e0)--(ec000);
        \draw[EdgeGraph](c000)--(ce000);
        \draw[EdgeGraph](c000)--(c0e00);
        \draw[EdgeGraph](c000)--(c00e0);
        \draw[EdgeGraph](c000)--(cc00000);
        \draw[EdgeGraph](c000)--(c0c0000);
        \draw[EdgeGraph](c000)--(c00c000);
        \draw[EdgeGraph](ce000)--(ce0e00);
        \draw[EdgeGraph](ce000)--(ce00e0);
        \draw[EdgeGraph](c0e00)--(ce0e00);
        \draw[EdgeGraph](c0e00)--(c0e0e0);
        \draw[EdgeGraph](c00e0)--(ce00e0);
        \draw[EdgeGraph](c00e0)--(c0e0e0);
    \end{tikzpicture}
    \label{subfig:tree_graph_e1c3}}
    \caption{Two graded graphs $\Par{\SyntaxTreesInternalNode(\GeneratingSet), \Up}$.}
    \label{fig:examples_tree_graphs}
\end{figure}
\medbreak

An internal node $u$ of a $\GeneratingSet$-tree $\TreeT$ is \Def{maximal} if $u$ has only
leaves as children. The set of all maximal nodes of $\TreeT$ is denoted by
$\InternalNodesMax(\TreeT)$. For any $u \in \InternalNodesMax(\TreeT)$, the \Def{deletion}
of $u$ in $\TreeT$ is the $\GeneratingSet$-tree $\DelNode{\TreeT}{u}$ obtained by replacing
the node $u$ of $\TreeT$ by a leaf. By relying on these definitions, the adjoint map of
$\Up$ satisfies, for any $\TreeT \in \SyntaxTreesInternalNode(\GeneratingSet)$,
\begin{equation} \label{equ:adjoint_up_trees}
    \Up^\Dual(\TreeT) = \sum_{u \in \InternalNodesMax(\TreeT)} \DelNode{\TreeT}{u}.
\end{equation}
\medbreak

%%%%%%%%%%%%%%%%%%%%%%%%%%%%%%%%%%%%%%%%%%%%%%%%%%%%%%%%%%%%%%%%%%%%%%%%%%%%%%%%%%%%%%%%%%%%
\subsubsection{Diagonal self-duality}
We give here a necessary and sufficient condition on the alphabet $\GeneratingSet$ for the
fact that $\Par{\SyntaxTreesInternalNode(\GeneratingSet), \Up}$ is $\phi$-diagonal
self-dual.
\medbreak

\begin{Proposition} \label{prop:prefix_tree_graphs_self_dual}
    The graded graph $\Par{\SyntaxTreesInternalNode(\GeneratingSet), \Up}$ is
    $\phi$-diagonal self-dual if and only if $\GeneratingSet$ is the empty alphabet or a
    singleton alphabet. When $\GeneratingSet$ is a singleton,
    \begin{math}
        \phi : \K \Angle{\SyntaxTreesInternalNode(\GeneratingSet)}
        \to \K \Angle{\SyntaxTreesInternalNode(\GeneratingSet)}
    \end{math}
    satisfies
    \begin{equation}
        \phi(\TreeT) = \Par{|\TreeT| - \# \InternalNodesMax(\TreeT)} \, \TreeT
    \end{equation}
    for any $\GeneratingSet$-tree $\TreeT$.
\end{Proposition}
\begin{proof}
    First of all, when $\GeneratingSet$ is empty,
    $\Par{\SyntaxTreesInternalNode(\GeneratingSet), \Up}$ is immediately $\phi$-diagonal
    self-dual for the zero map $\phi$.  Assume that $\GeneratingSet$ is the singleton
    $\{\GenA\}$.  When $\TreeT = \Leaf$, since $\Par{\Up^\Dual \Up - \Up
    \Up^\Dual}\Par{\Leaf} = \Leaf$, the property is  satisfied. Assume now that $\TreeT$ has
    at least one internal node.  All terms of $\Par{\Up^\Dual \Up}(\TreeT)$ are obtained by
    changing a leaf of $\TreeT$ into an internal node decorated by $\GenA$, and then by
    suppressing a maximal internal node of the obtained tree.  Then, in particular when the
    suppressed internal node is the one which has been just added, $\TreeT$ occurs in
    $\Par{\Up^\Dual \Up}(\TreeT)$.  For this reason, the coefficient of the term $\TreeT$ is
    $|\TreeT|$.  Moreover, all terms of $\Par{\Up \Up^\Dual}(\TreeT)$ are obtained by
    suppressing a maximal internal node of $\TreeT$, and then by changing a leaf into an
    internal node decorated by $\GenA$ of the obtained tree. For this reason, the coefficient
    of the term $\TreeT$ is $\InternalNodesMax(\TreeT)$. Finally, since all trees different
    from $\TreeT$ appearing $\Par{\Up^\Dual \Up}(\TreeT)$ or in $\Par{\Up
    \Up^\Dual}(\TreeT)$ are the same and have all $1$ as coefficient, the statement of the
    proposition follows.
    \smallbreak

    Conversely, assume that $\GeneratingSet$ is not empty neither a singleton. Hence, there
    exist $\GenA, \GenB \in \GeneratingSet$ with $\GenA \ne \GenB$, and we have in
    particular
    \begin{equation} \label{equ:prefix_tree_graphs_self_dual}
        \Par{\Up^\Dual \Up - \Up \Up^\Dual}\Par{\GenA}
        =
        \Up^\Dual\Par{\sum_{\substack{
            \GenC \in \GeneratingSet \\
            i \in [|\GenA|]
        }}
        \GenA \circ_i \GenC}
        - \Up\Par{\Leaf}
        =
        (\# \GeneratingSet) |\GenA| \; \GenA - \sum_{\GenC \in \GeneratingSet} \GenC.
    \end{equation}
    Since $\GenB$ appears in~\eqref{equ:prefix_tree_graphs_self_dual}, this shows that
    $\Par{\SyntaxTreesInternalNode(\GeneratingSet), \Up}$ is not $\phi$-diagonal self-dual.
\end{proof}
\medbreak

%%%%%%%%%%%%%%%%%%%%%%%%%%%%%%%%%%%%%%%%%%%%%%%%%%%%%%%%%%%%%%%%%%%%%%%%%%%%%%%%%%%%%%%%%%%%
\subsubsection{Path enumeration} \label{subsubsec:prefix_graph_path_enumeration}
Recall that if $(\PosetP, \Leq)$ is a finite poset, a \Def{linear extension} of $\PosetP$ is
a bijective map $\sigma : \PosetP \to [\# \PosetP]$ such that for any $x, y \in \PosetP$, $x
\Leq y$ implies $\sigma(x) \leq \sigma(y)$, where $\leq$ is the natural total order on the
set of natural numbers $[\# \PosetP]$. The linear extension $\sigma$ can be encoded by the
permutation
\begin{equation} \label{equ:linear_extension_permutation}
    \Par{x_1, \dots, x_{\# \PosetP}}
\end{equation}
of elements of $\PosetP$, wherein for any $x \in \PosetP$, $\sigma(x)$ is the position of
$x$ in the word~\eqref{equ:linear_extension_permutation}. Observe that if $\PosetP$ is the
disjoint union of some posets $\PosetP_1$, \dots, $\PosetP_k$, $k \geq 0$, each permutation
encoding a linear extension $\sigma$ of $\PosetP$ is obtained by shuffling the permutations
encoding respectively linear extensions $\sigma_1$, \dots, $\sigma_k$ of $\PosetP_1$, \dots,
$\PosetP_k$. Therefore, if each $\PosetP_i$, $i \in [k]$, has $a_i$ linear extensions, then
$\PosetP$ has the multinomial
\begin{equation} \label{equ:shuffle_linear_extensions}
    \lbag a_1, \dots, a_k \rbag ! := \frac{\Par{a_1 + \dots + a_k}!}{a_1! \dots a_k!}
\end{equation}
as number of linear extensions.
\medbreak

The \Def{poset induced} by a $\GeneratingSet$-tree $\TreeT$ is the poset $\TreePoset{\TreeT}
:= \Par{\InternalNodes(\TreeT), \Leq}$ wherein for any $u, v \in \InternalNodes(\TreeT)$, $u
\Leq v$ if $u$ is an ancestor of $v$. The number $\Hook{\TreeT}$ of all linear extensions of
$\TreePoset{\TreeT}$ is given by the hook-length formula for rooted planar
trees~\cite{Knu98} and satisfies
\begin{equation}
    \Hook{\TreeT} =
    \frac{\Deg(\TreeT)!}{\prod\limits_{u \in \InternalNodes(\TreeT)}
    \Deg\Par{\TreeT(u)}}.
\end{equation}
By elementary computations involving multinomial coefficients, and by structural induction
on $\GeneratingSet$-trees, one can show that these numbers satisfy the recurrence relation
\begin{subequations}
\begin{equation}
    \Hook{\Leaf} = 1,
\end{equation}
\begin{equation}
    \Hook{\GenA \BPar{\TreeS_1, \dots, \TreeS_{|\GenA|}}}
    = \lbag \Deg\Par{\TreeS_1}, \dots, \Deg\Par{\TreeS_{|\GenA|}}\rbag !
    \prod_{i \in [|\GenA|]} \Hook{\TreeS_i}
\end{equation}
\end{subequations}
for any $\GenA \in \GeneratingSet$ and any $\GeneratingSet$-trees $\TreeS_1$, \dots,
$\TreeS_{|\GenA|}$.
\medbreak

\begin{Proposition} \label{prop:prefix_tree_graph_hook_series}
    For any finite alphabet $\GeneratingSet$, the hook series of
    $\Par{\SyntaxTreesInternalNode(\GeneratingSet), \Up}$ satisfies
    \begin{math}
        \Angle{\TreeT, \HookSeries{\Up}} = \Hook{\TreeT}
    \end{math}
    for any $\GeneratingSet$-tree~$\TreeT$.
\end{Proposition}
\begin{proof}
    Let us proceed by induction on the degree of the $\GeneratingSet$-trees.  By definition
    of $\HookSeries{\Up}$, $\Angle{\Leaf, \HookSeries{\Up}} = 1$, and since $\Hook{\Leaf} =
    1$, the property holds. Now, let $\TreeT$ be a $\GeneratingSet$-tree of degree $d \geq
    1$.  We have, by definition of $\HookSeries{\Up}$, by~\eqref{equ:adjoint_up_trees}, and
    by induction hypothesis,
    \begin{equation} \begin{split} 
        \label{equ:prefix_tree_graph_hook_series}
        \Angle{\TreeT, \HookSeries{\Up}}
        & =
        \Angle{\Up^\Dual(\TreeT), \HookSeries{\Up}}
        \\
        & =
        \Angle{\sum_{u \in \InternalNodesMax(\TreeT)} \DelNode{\TreeT}{u}, \HookSeries{\Up}}
        \\
        & =
        \sum_{u \in \InternalNodesMax(\TreeT)} \Angle{\DelNode{\TreeT}{u}, \HookSeries{\Up}}
        \\
        & =
        \sum_{u \in \InternalNodesMax(\TreeT)} \Hook{\DelNode{\TreeT}{u}}.
    \end{split} \end{equation}
    From the definition of $\TreePoset{\TreeT}$, it follows that any permutation encoding a
    linear extension of $\TreePoset{\TreeT}$ writes as
    \begin{equation}
        \Par{\epsilon, v^{(1)}, \dots, v^{(d - 2)}, u}
    \end{equation}
    where the permutation $\Par{\epsilon, v^{(1)}, \dots, v^{(d - 2)}}$ encodes a linear
    extension of $\TreePoset{\DelNode{\TreeT}{u}}$ and $u$ is a maximal node of $\TreeT$.
    This leads, by~\eqref{equ:prefix_tree_graph_hook_series}, to $\Angle{\TreeT,
    \HookSeries{\Up}} = \Hook{\TreeT}$.
\end{proof}
\medbreak

For any alphabet $\GeneratingSet$, let $m_\GeneratingSet$ be the arity of a letter having a
maximal arity in~$\GeneratingSet$.
\medbreak

\begin{Proposition} \label{prop:prefix_tree_graph_initial_paths_series}
    For any finite alphabet $\GeneratingSet$, the initial path series of
    $\Par{\SyntaxTreesInternalNode(\GeneratingSet), \Up}$ satisfies
    \begin{equation} \label{equ:multipaths_graph_U_free_operad}
        \Angle{t^d, \InitialPathsSeries{\Up}}
        = \sum_{n \in \Han{1 + \Par{m_\GeneratingSet - 1}d}} \theta(d, n)
    \end{equation}
    for any $d \geq 0$, where $\theta(d, n)$ satisfies the recurrence
    $\theta(d, n) = 0$ for any $n \leq 0$ and $d \in \Z$, $\theta(0, 1) = 1$, and
    \begin{equation}
        \theta(d, n) =
        \sum_{\substack{
            \GenA \in \GeneratingSet \\
            |\GenA| \leq n
        }}
        \Par{n + 1 - |\GenA|} \; \theta(d - 1, n + 1 - |\GenA|)
    \end{equation}
    for any $d \geq 1$ and $n \geq 1$.
\end{Proposition}
\begin{proof}
    Let us prove that $\theta(d, n)$ is the number of initial paths in
    $\Par{\SyntaxTreesInternalNode(\GeneratingSet), \Up}$ to elements of degree $d$ and
    arity $n$ by induction on $d \geq 0$. First, since there are no syntax trees of arity
    $0$ or less, $\theta(d, n) = 0$ for any $n \leq 0$ and $d \in \Z$.  Moreover, since
    $\theta(0, 1) = 1$, the property is satisfied because $\Leaf$ is the unique tree of
    degree $0$ and arity $1$, and there is exactly one initial path to~$\Leaf$. Assume now
    that $d \geq 1$. Any initial path to a tree $\TreeT$ of degree $d$ and arity $n$
    decomposes as an initial path to a tree $\TreeT'$ of degree $d - 1$ such that $\TreeT$
    appears in $\Up\Par{\TreeT'}$. Hence, there is an $i \in \Han{\left|\TreeT'\right|}$ and
    an $\GenA \in \GeneratingSet$ such that $\TreeT = \TreeT' \circ_i \GenA$. This implies
    that
    \begin{math}
        \left|\TreeT'\right| = n + 1 - |\GenA|
    \end{math}
    and $|\GenA| \leq n$.  By induction hypothesis, there are
    \begin{math}
        \theta\Par{d - 1, \left|\TreeT'\right|}
    \end{math}
    initial paths to trees of degree $d - 1$ and arity $\left|\TreeT'\right|$. Therefore,
    due to the fact that initial paths to trees of degree $d$ decompose as explained before,
    $\theta(d, n)$ satisfies the claimed property. Finally,
    \eqref{equ:multipaths_graph_U_free_operad} is a consequence of the fact that a
    $\GeneratingSet$-tree of degree $d$ has $1$ as minimal arity and $1 +
    \Par{m_\GeneratingSet - 1} d$ as maximal arity. Indeed, this maximal arity is reached
    for trees consisting only in internal nodes of a maximal arity~$m_\GeneratingSet$.
\end{proof}
\medbreak

Here are the sequences of the first coefficients of some initial paths series of
$\Par{\SyntaxTreesInternalNode(\GeneratingSet), \Up}$:
\begin{subequations}
\begin{equation}
    1, 1, 2, 6, 24, 120, 720, 5040,
    \qquad \mbox{ for } \GeneratingSet = \{\GenA\} \mbox{ with } |\GenA| = 2,
\end{equation}
\begin{equation}
    1, 1, 3, 15, 105, 945, 10395, 135135,
    \qquad \mbox{ for } \GeneratingSet = \{\GenC\} \mbox{ with } |\GenC| = 3,
\end{equation}
\begin{equation}
    1, 2, 8, 48, 384, 3840, 46080, 645120,
    \qquad \mbox{ for } \GeneratingSet = \{\GenA, \GenB\}
    \mbox{ with } |\GenA| = |\GenB| = 2,
\end{equation}
\begin{equation}
    1, 2, 10, 82, 938, 13778, 247210, 5240338,
    \qquad \mbox{ for } \GeneratingSet = \{\GenA, \GenC\}
    \mbox{ with } |\GenA| = 2, |\GenC| = 3.
\end{equation}
\end{subequations}
These are respectively Sequences~\OEIS{A000142}, \OEIS{A001147}, \OEIS{A000165},
and~\OEIS{A112487} of~\cite{Slo}.
\medbreak

%%%%%%%%%%%%%%%%%%%%%%%%%%%%%%%%%%%%%%%%%%%%%%%%%%%%%%%%%%%%%%%%%%%%%%%%%%%%%%%%%%%%%%%%%%%%
%%%%%%%%%%%%%%%%%%%%%%%%%%%%%%%%%%%%%%%%%%%%%%%%%%%%%%%%%%%%%%%%%%%%%%%%%%%%%%%%%%%%%%%%%%%%
\subsection{Twisted prefix graded graphs}
We now introduce a second sort of graded graphs and present some of their combinatorial
properties. The aim is to study this graded graph in order to show in the next section
that, together with the first kind of graded graphs, this forms
a pair of $\phi$-diagonal dual graded graphs.
\medbreak

%%%%%%%%%%%%%%%%%%%%%%%%%%%%%%%%%%%%%%%%%%%%%%%%%%%%%%%%%%%%%%%%%%%%%%%%%%%%%%%%%%%%%%%%%%%%
\subsubsection{First definitions and properties} \label{subsubsec:twisted_prefix_graphs}
For any finite alphabet $\GeneratingSet$, let
\begin{math}
    \Vp^\Dual : \K \Angle{\SyntaxTreesInternalNode(\GeneratingSet)}
    \to \K \Angle{\SyntaxTreesInternalNode(\GeneratingSet)}
\end{math}
be the linear map satisfying the recurrence
\begin{subequations}
\begin{equation} \label{equ:twisted_prefix_tree_adjoint_recursive_1}
    \Vp^\Dual(\Leaf) := 0,
\end{equation}
\begin{equation}\label{equ:twisted_prefix_tree_adjoint_recursive_2}
    \Vp^\Dual\Par{\GenA \BPar{\TreeS, \Leaf, \dots, \Leaf}} := \TreeS,
\end{equation}
\begin{equation}\label{equ:twisted_prefix_tree_adjoint_recursive_3}
    \Vp^\Dual\Par{\GenA \BPar{\TreeS_1, \dots, \TreeS_{|\GenA|}}}
    :=
    \sum_{j \in [2, |\GenA|]}
    \GenA \BPar{\TreeS_1, \dots, \TreeS_{j - 1},
    \Vp^\Dual\Par{\TreeS_j}, \TreeS_{j + 1}, \dots, \TreeS_{|\GenA|}},
\end{equation}
\end{subequations}
for any $\GenA \in \GeneratingSet$ and any $\GeneratingSet$-trees $\TreeS$, $\TreeS_1$,
\dots, $\TreeS_{|\GenA|}$ such that there is at least a $j \in [2, |\GenA|]$ such that
$\TreeS_j \ne \Leaf$. For instance, for
\begin{math}
    \GeneratingSet = \Bra{\GenE, \GenA, \GenC}
\end{math}
where $|\GenE| = 1$, $|\GenA| = 2$, and $|\GenC| = 3$, one has
\begin{equation}
    \Vp^\Dual\Par{
    \begin{tikzpicture}[Centering,xscale=0.18,yscale=0.1]
        \node(0)at(0.00,-9.00){};
        \node(11)at(7.00,-13.50){};
        \node(13)at(9.00,-13.50){};
        \node(14)at(10.00,-13.50){};
        \node(16)at(11.00,-13.50){};
        \node(17)at(12.00,-13.50){};
        \node(2)at(2.00,-9.00){};
        \node(5)at(3.00,-13.50){};
        \node(7)at(4.00,-9.00){};
        \node(8)at(5.00,-9.00){};
        \node(9)at(6.00,-9.00){};
        \node[NodeST](1)at(1.00,-4.50){$\GenA$};
        \node[NodeST](10)at(8.00,-4.50){$\GenC$};
        \node[NodeST](12)at(8.00,-9.00){$\GenA$};
        \node[NodeST](15)at(11.00,-9.00){$\GenC$};
        \node[NodeST](3)at(4.00,0.00){$\GenC$};
        \node[NodeST](4)at(3.00,-9.00){$\GenE$};
        \node[NodeST](6)at(4.00,-4.50){$\GenC$};
        \draw[Edge](0)--(1);
        \draw[Edge](1)--(3);
        \draw[Edge](10)--(3);
        \draw[Edge](11)--(12);
        \draw[Edge](12)--(10);
        \draw[Edge](13)--(12);
        \draw[Edge](14)--(15);
        \draw[Edge](15)--(10);
        \draw[Edge](16)--(15);
        \draw[Edge](17)--(15);
        \draw[Edge](2)--(1);
        \draw[Edge](4)--(6);
        \draw[Edge](5)--(4);
        \draw[Edge](6)--(3);
        \draw[Edge](7)--(6);
        \draw[Edge](8)--(6);
        \draw[Edge](9)--(10);
        \node(r)at(4.00,4){};
        \draw[Edge](r)--(3);
    \end{tikzpicture}}
    =
    \begin{tikzpicture}[Centering,xscale=0.17,yscale=0.11]
        \node(0)at(0.00,-7.50){};
        \node(10)at(7.00,-11.25){};
        \node(11)at(8.00,-11.25){};
        \node(13)at(9.00,-11.25){};
        \node(14)at(10.00,-11.25){};
        \node(2)at(2.00,-7.50){};
        \node(5)at(3.00,-7.50){};
        \node(6)at(4.00,-7.50){};
        \node(8)at(5.00,-11.25){};
        \node[NodeST](1)at(1.00,-3.75){$\GenA$};
        \node[NodeST](12)at(9.00,-7.50){$\GenC$};
        \node[NodeST](3)at(3.00,0.00){$\GenC$};
        \node[NodeST](4)at(3.00,-3.75){$\GenE$};
        \node[NodeST](7)at(6.00,-3.75){$\GenC$};
        \node[NodeST](9)at(6.00,-7.50){$\GenA$};
        \draw[Edge](0)--(1);
        \draw[Edge](1)--(3);
        \draw[Edge](10)--(9);
        \draw[Edge](11)--(12);
        \draw[Edge](12)--(7);
        \draw[Edge](13)--(12);
        \draw[Edge](14)--(12);
        \draw[Edge](2)--(1);
        \draw[Edge](4)--(3);
        \draw[Edge](5)--(4);
        \draw[Edge](6)--(7);
        \draw[Edge](7)--(3);
        \draw[Edge](8)--(9);
        \draw[Edge](9)--(7);
        \node(r)at(3.00,3.5){};
        \draw[Edge](r)--(3);
    \end{tikzpicture}
    +
    \begin{tikzpicture}[Centering,xscale=0.17,yscale=0.11]
        \node(0)at(0.00,-7.50){};
        \node(11)at(7.00,-11.25){};
        \node(13)at(9.00,-11.25){};
        \node(14)at(10.00,-7.50){};
        \node(2)at(2.00,-7.50){};
        \node(5)at(3.00,-11.25){};
        \node(7)at(4.00,-7.50){};
        \node(8)at(5.00,-7.50){};
        \node(9)at(6.00,-7.50){};
        \node[NodeST](1)at(1.00,-3.75){$\GenA$};
        \node[NodeST](10)at(8.00,-3.75){$\GenC$};
        \node[NodeST](12)at(8.00,-7.50){$\GenA$};
        \node[NodeST](3)at(4.00,0.00){$\GenC$};
        \node[NodeST](4)at(3.00,-7.50){$\GenE$};
        \node[NodeST](6)at(4.00,-3.75){$\GenC$};
        \draw[Edge](0)--(1);
        \draw[Edge](1)--(3);
        \draw[Edge](10)--(3);
        \draw[Edge](11)--(12);
        \draw[Edge](12)--(10);
        \draw[Edge](13)--(12);
        \draw[Edge](14)--(10);
        \draw[Edge](2)--(1);
        \draw[Edge](4)--(6);
        \draw[Edge](5)--(4);
        \draw[Edge](6)--(3);
        \draw[Edge](7)--(6);
        \draw[Edge](8)--(6);
        \draw[Edge](9)--(10);
        \node(r)at(4.00,3.5){};
        \draw[Edge](r)--(3);
    \end{tikzpicture}
    +
    \begin{tikzpicture}[Centering,xscale=0.18,yscale=0.1]
        \node(0)at(0.00,-8.00){};
        \node(11)at(7.00,-8.00){};
        \node(12)at(8.00,-12.00){};
        \node(14)at(9.00,-12.00){};
        \node(15)at(10.00,-12.00){};
        \node(2)at(2.00,-8.00){};
        \node(5)at(3.00,-12.00){};
        \node(7)at(4.00,-8.00){};
        \node(8)at(5.00,-8.00){};
        \node(9)at(6.00,-8.00){};
        \node[NodeST](1)at(1.00,-4.00){$\GenA$};
        \node[NodeST](10)at(7.00,-4.00){$\GenC$};
        \node[NodeST](13)at(9.00,-8.00){$\GenC$};
        \node[NodeST](3)at(4.00,0.00){$\GenC$};
        \node[NodeST](4)at(3.00,-8.00){$\GenE$};
        \node[NodeST](6)at(4.00,-4.00){$\GenC$};
        \draw[Edge](0)--(1);
        \draw[Edge](1)--(3);
        \draw[Edge](10)--(3);
        \draw[Edge](11)--(10);
        \draw[Edge](12)--(13);
        \draw[Edge](13)--(10);
        \draw[Edge](14)--(13);
        \draw[Edge](15)--(13);
        \draw[Edge](2)--(1);
        \draw[Edge](4)--(6);
        \draw[Edge](5)--(4);
        \draw[Edge](6)--(3);
        \draw[Edge](7)--(6);
        \draw[Edge](8)--(6);
        \draw[Edge](9)--(10);
        \node(r)at(4.00,4){};
        \draw[Edge](r)--(3);
    \end{tikzpicture}.
\end{equation}
\medbreak

This recursive definition for $\Vp^\Dual$ is convenient to set up proofs by induction of
properties involving this map. Nevertheless, we shall consider also a non-recursive
description relying on the following definitions. Let $\TreeT$ be a $\GeneratingSet$-tree
decomposing as $\TreeT = \TreeS \circ^u \Par{\GenA \circ_i \TreeS'}$ where $\TreeS$ and
$\TreeS'$ are $\GeneratingSet$-trees, $u \in \Leaves(\TreeS)$, $i \in [|\GenA|]$, and $\GenA
\in \GeneratingSet$. The \Def{contraction} of the internal node $u$ of $\TreeT$ is the
$\GeneratingSet$-tree $\ContractNode{\TreeT}{u} := \TreeS \circ^u \TreeS'$. For instance,
by considering the same alphabet $\GeneratingSet$ as in the previous example, the
contraction of the node $2$ of the $\GeneratingSet$-tree
\begin{equation} \label{equ:example_contraction}
    \TreeT :=
    \begin{tikzpicture}[Centering,xscale=0.17,yscale=0.15]
        \node(0)at(0.00,-6.00){};
        \node(11)at(7.00,-12.00){};
        \node(13)at(9.00,-12.00){};
        \node(14)at(8.00,-3.00){};
        \node(2)at(2.00,-6.00){};
        \node(4)at(3.00,-6.00){};
        \node(6)at(4.00,-6.00){};
        \node(7)at(3.00,-12.00){};
        \node(9)at(5.00,-12.00){};
        \node[NodeST](1)at(1.00,-3.00){$\GenA$};
        \node[NodeST](10)at(6.00,-6.00){$\GenA$};
        \node[NodeST](12)at(8.00,-9.00){$\GenA$};
        \node[NodeST](3)at(4.00,0.00){$\GenC$};
        \node[NodeST](5)at(4.00,-3.00){$\GenC$};
        \node[NodeST](8)at(4.00,-9.00){$\GenA$};
        \draw[Edge](0)--(1);
        \draw[Edge](1)--(3);
        \draw[Edge](10)--(5);
        \draw[Edge](11)--(12);
        \draw[Edge](12)--(10);
        \draw[Edge](13)--(12);
        \draw[Edge](14)--(3);
        \draw[Edge](2)--(1);
        \draw[Edge](4)--(5);
        \draw[Edge](5)--(3);
        \draw[Edge](6)--(5);
        \draw[Edge](7)--(8);
        \draw[Edge](8)--(10);
        \draw[Edge](9)--(8);
        \node(r)at(4.00,2.75){};
        \draw[Edge](r)--(3);
    \end{tikzpicture}
    =
    \begin{tikzpicture}[Centering,xscale=0.19,yscale=0.21]
        \node(0)at(0.00,-4.00){};
        \node(2)at(2.00,-4.00){};
        \node(4)at(3.00,-2.00){};
        \node(5)at(4.00,-2.00){};
        \node[NodeST](1)at(1.00,-2.00){$\GenA$};
        \node[NodeST](3)at(3.00,0.00){$\GenC$};
        \draw[Edge](0)--(1);
        \draw[Edge](1)--(3);
        \draw[Edge](2)--(1);
        \draw[Edge](4)--(3);
        \draw[Edge](5)--(3);
        \node(r)at(3.00,2){};
        \draw[Edge](r)--(3);
    \end{tikzpicture}
    \circ^2
    \Par{
    \CorollaThree{\GenC}
    \circ_3
    \begin{tikzpicture}[Centering,xscale=0.17,yscale=0.19]
        \node(0)at(0.00,-4.67){};
        \node(2)at(2.00,-4.67){};
        \node(4)at(4.00,-4.67){};
        \node(6)at(6.00,-4.67){};
        \node[NodeST](1)at(1.00,-2.33){$\GenA$};
        \node[NodeST](3)at(3.00,0.00){$\GenA$};
        \node[NodeST](5)at(5.00,-2.33){$\GenA$};
        \draw[Edge](0)--(1);
        \draw[Edge](1)--(3);
        \draw[Edge](2)--(1);
        \draw[Edge](4)--(5);
        \draw[Edge](5)--(3);
        \draw[Edge](6)--(5);
        \node(r)at(3.00,2.25){};
        \draw[Edge](r)--(3);
    \end{tikzpicture}},
\end{equation}
is
\begin{equation}
    \ContractNode{\TreeT}{2} =
    \begin{tikzpicture}[Centering,xscale=0.19,yscale=0.21]
        \node(0)at(0.00,-4.00){};
        \node(2)at(2.00,-4.00){};
        \node(4)at(3.00,-2.00){};
        \node(5)at(4.00,-2.00){};
        \node[NodeST](1)at(1.00,-2.00){$\GenA$};
        \node[NodeST](3)at(3.00,0.00){$\GenC$};
        \draw[Edge](0)--(1);
        \draw[Edge](1)--(3);
        \draw[Edge](2)--(1);
        \draw[Edge](4)--(3);
        \draw[Edge](5)--(3);
        \node(r)at(3.00,2){};
        \draw[Edge](r)--(3);
    \end{tikzpicture}
    \circ^2
    \begin{tikzpicture}[Centering,xscale=0.17,yscale=0.19]
        \node(0)at(0.00,-4.67){};
        \node(2)at(2.00,-4.67){};
        \node(4)at(4.00,-4.67){};
        \node(6)at(6.00,-4.67){};
        \node[NodeST](1)at(1.00,-2.33){$\GenA$};
        \node[NodeST](3)at(3.00,0.00){$\GenA$};
        \node[NodeST](5)at(5.00,-2.33){$\GenA$};
        \draw[Edge](0)--(1);
        \draw[Edge](1)--(3);
        \draw[Edge](2)--(1);
        \draw[Edge](4)--(5);
        \draw[Edge](5)--(3);
        \draw[Edge](6)--(5);
        \node(r)at(3.00,2.25){};
        \draw[Edge](r)--(3);
    \end{tikzpicture}
    =
    \begin{tikzpicture}[Centering,xscale=0.15,yscale=0.15]
        \node(0)at(0.00,-6.00){};
        \node(10)at(9.00,-9.00){};
        \node(11)at(10.00,-3.00){};
        \node(2)at(2.00,-6.00){};
        \node(4)at(3.00,-9.00){};
        \node(6)at(5.00,-9.00){};
        \node(8)at(7.00,-9.00){};
        \node[NodeST](1)at(1.00,-3.00){$\GenA$};
        \node[NodeST](3)at(6.00,0.00){$\GenC$};
        \node[NodeST](5)at(4.00,-6.00){$\GenA$};
        \node[NodeST](7)at(6.00,-3.00){$\GenA$};
        \node[NodeST](9)at(8.00,-6.00){$\GenA$};
        \draw[Edge](0)--(1);
        \draw[Edge](1)--(3);
        \draw[Edge](10)--(9);
        \draw[Edge](11)--(3);
        \draw[Edge](2)--(1);
        \draw[Edge](4)--(5);
        \draw[Edge](5)--(7);
        \draw[Edge](6)--(5);
        \draw[Edge](7)--(3);
        \draw[Edge](8)--(9);
        \draw[Edge](9)--(7);
        \node(r)at(6.00,2.5){};
        \draw[Edge](r)--(3);
    \end{tikzpicture}.
\end{equation}
Observe that when $u$ is a maximal internal node of $\TreeT$, one has
$\ContractNode{\TreeT}{u} = \DelNode{\TreeT}{u}$.
\medbreak

Besides, an internal node $u$ of $\TreeT$ is \Def{quasi-maximal} if $u$ admits no occurrence
of the letter $1$ and all $uj$ are leaves for all $j \in [2, k]$ where $k$ is the arity of
$u$ in $\TreeT$. In other words, the path connecting the root of $\TreeT$ with $u$ does
never go through a first edge of an internal node and $u$ has only leaves as children except
possibly at first position.  We denote by $\InternalNodesQuasiMax(\TreeT)$ the set of all
quasi-maximal nodes of $\TreeT$. For instance, by considering the same alphabet
$\GeneratingSet$ as in the previous examples, here is a $\GeneratingSet$-tree wherein its
quasi-maximal nodes are framed:
\begin{equation}
    \begin{tikzpicture}[Centering,xscale=0.18,yscale=0.1]
        \node(0)at(0.00,-9.00){};
        \node(11)at(7.00,-13.50){};
        \node(13)at(9.00,-13.50){};
        \node(14)at(10.00,-13.50){};
        \node(16)at(11.00,-13.50){};
        \node(17)at(12.00,-13.50){};
        \node(2)at(2.00,-9.00){};
        \node(5)at(3.00,-13.50){};
        \node(7)at(4.00,-9.00){};
        \node(8)at(5.00,-9.00){};
        \node(9)at(6.00,-9.00){};
        \node[NodeST](1)at(1.00,-4.50){$\GenA$};
        \node[NodeST](10)at(8.00,-4.50){$\GenC$};
        \node[NodeST](12)at(8.00,-9.00){$\GenA$};
        \node[NodeST](15)at(11.00,-9.00){$\GenC$};
        \node[NodeST](3)at(4.00,0.00){$\GenC$};
        \node[NodeST](4)at(3.00,-9.00){$\GenE$};
        \node[NodeST](6)at(4.00,-4.50){$\GenC$};
        \draw[Edge](0)--(1);
        \draw[Edge](1)--(3);
        \draw[Edge](10)--(3);
        \draw[Edge](11)--(12);
        \draw[Edge](12)--(10);
        \draw[Edge](13)--(12);
        \draw[Edge](14)--(15);
        \draw[Edge](15)--(10);
        \draw[Edge](16)--(15);
        \draw[Edge](17)--(15);
        \draw[Edge](2)--(1);
        \draw[Edge](4)--(6);
        \draw[Edge](5)--(4);
        \draw[Edge](6)--(3);
        \draw[Edge](7)--(6);
        \draw[Edge](8)--(6);
        \draw[Edge](9)--(10);
        \node(r)at(4.00,4){};
        \draw[Edge](r)--(3);
        \node[fit=(6),Box,inner sep=0pt,opacity=.4]{};
        \node[fit=(12),Box,inner sep=0pt,opacity=.4]{};
        \node[fit=(15),Box,inner sep=0pt,opacity=.4]{};
    \end{tikzpicture}.
\end{equation}
Observe that for any $u \in \InternalNodesQuasiMax(\TreeT)$, one has $\TreeT = \TreeS
\circ^u \Par{\GenA \circ_1 \TreeS'}$ where $\TreeS$ and $\TreeS'$ are
$\GeneratingSet$-trees, possibly leaves. Therefore, $\ContractNode{\TreeT}{u}$ is
well-defined.
\medbreak

By relying on these definitions, the map $\Vp^\Dual$ rephrases, in a non-recursive way, as
follows.
\medbreak

\begin{Proposition} \label{prop:twisted_prefix_tree_graph_adjoint_direct}
    For any finite alphabet $\GeneratingSet$ and any $\GeneratingSet$-tree $\TreeT$, the map
    $\Vp^\Dual$ satisfies
    \begin{equation}
        \label{equ:twisted_prefix_tree_graph_adjoint_direct}
        \Vp^\Dual(\TreeT)
        = \sum_{u \in \InternalNodesQuasiMax(\TreeT)}\ContractNode{\TreeT}{u}.
    \end{equation}
\end{Proposition}
\begin{proof}
    We proceed by induction on the degree $d$ of $\TreeT$. If $d = 0$, then $\TreeT = \Leaf$
    and we have $\Vp^\Dual(\Leaf) = 0$. Moreover, since $\Leaf$ has no internal node, the
    right-hand side of~\eqref{equ:twisted_prefix_tree_graph_adjoint_direct} is equal to $0$.
    Hence, the property holds here. Assume now that $d \geq 1$. If $\TreeT$ decomposes as
    $\TreeT = \GenA \BPar{\TreeS, \Leaf, \dots, \Leaf}$ where $\GenA \in \GeneratingSet$ and
    $\TreeS$ is a $\GeneratingSet$-tree, we have $\Vp^\Dual(\TreeT) = \TreeS$ and, since the
    root $\epsilon$ of $\TreeT$ is the only quasi-maximal node of $\TreeT$, the right-hand
    side of~\eqref{equ:twisted_prefix_tree_graph_adjoint_direct} is equal to
    $\ContractNode{\TreeT}{\epsilon} = \TreeS$. Hence, the property holds here also. It
    remains to explore the case where $\TreeT$ decomposes as $\TreeT = \GenA \BPar{\TreeS_1,
    \dots, \TreeS_{|\GenA|}}$ where $\GenA \in \GeneratingSet$ and $\TreeS_1$, \dots,
    $\TreeS_{|\GenA|}$ are $\GeneratingSet$-trees such that there is a $j \in [2, |\GenA|]$
    such that $\TreeS_j \ne \Leaf$. We have, by definition of $\Vp^\Dual$ and by induction
    hypothesis,
    \begin{equation} \begin{split}
        \label{equ:twisted_prefix_tree_graph_adjoint_direct_2}
        \Vp^\Dual(\TreeT) 
        & =
        \sum_{j \in [2, |\GenA|]}
        \GenA \BPar{\TreeS_1, \dots, \TreeS_{j - 1},
            \Vp^\Dual\Par{\TreeS_j}, \TreeS_{j + 1}, \dots, 
            \TreeS_{|\GenA|}} \\
        & =
        \sum_{j \in [2, |\GenA|]}
        \GenA \BPar{\TreeS_1, \dots, \TreeS_{j - 1},
            \sum_{u \in \InternalNodesQuasiMax\Par{\TreeS_j}}
            \ContractNode{\TreeS_j}{u}, \TreeS_{j + 1}, \dots,
            \TreeS_{|\GenA|}}.
    \end{split} \end{equation}
    Now, the last member of~\eqref{equ:twisted_prefix_tree_graph_adjoint_direct_2} is equal
    to the right-hand side of~\eqref{equ:twisted_prefix_tree_graph_adjoint_direct} since
    \begin{equation}
        \InternalNodesQuasiMax(\TreeT)
        = \bigcup_{j \in [2, |\GenA|]}
        \Bra{j u : u \in \InternalNodesQuasiMax\Par{\TreeS_j}}.
    \end{equation}
    This says that the quasi-maximal nodes of $\TreeT$ come from the quasi-maximal nodes of
    the $\TreeS_j$, $j \in [2, |\GenA|]$.
    Whence~\eqref{equ:twisted_prefix_tree_graph_adjoint_direct} is established.
\end{proof}
\medbreak

It follows from Proposition~\ref{prop:twisted_prefix_tree_graph_adjoint_direct} that for any
$\GeneratingSet$-tree $\TreeT$, all the trees appearing in $\Vp^\Dual(\TreeT)$ have
$\Deg(\TreeT) - 1$ as degree. For this reason, the graph
$\Par{\SyntaxTreesInternalNode(\GeneratingSet), \Vp}$ is graded and the rank of a
$\GeneratingSet$-tree is its degree.  We call
$\Par{\SyntaxTreesInternalNode(\GeneratingSet), \Vp}$ the \Def{$\GeneratingSet$-twisted
prefix graph}. Besides, again by
Proposition~\ref{prop:twisted_prefix_tree_graph_adjoint_direct}, for any
$\GeneratingSet$-tree $\TreeT$, the trees appearing in $\Vp^\Dual(\TreeT)$ have trivial
coefficients.  For this reason, the graph $\Par{\SyntaxTreesInternalNode(\GeneratingSet),
\Vp}$ is simple. Moreover, one can prove by structural induction on $\GeneratingSet$-trees
that any $\GeneratingSet$-tree $\TreeT$ different from the leaf admits at least one internal
node which is quasi-maximal. For this reason, if $\TreeT$ is a $\GeneratingSet$-tree
different from the leaf, $\Vp^\Dual(\TreeT) \ne 0$, implying that
$\Par{\SyntaxTreesInternalNode(\GeneratingSet), \Vp}$ admits $\Leaf$ as root.  Observe also
that when $\GeneratingSet$ contains only unary letters,
$\Par{\SyntaxTreesInternalNode(\GeneratingSet), \Vp}$ is the line.
\medbreak

Since $\Vp$ is the adjoint map of $\Vp^\Dual$, we can provide a recursive description of
$\Vp$ from the recursive definition of $\Vp^\Dual$
of~\eqref{equ:twisted_prefix_tree_adjoint_recursive_1},
\eqref{equ:twisted_prefix_tree_adjoint_recursive_2},
and~\eqref{equ:twisted_prefix_tree_adjoint_recursive_3}. Indeed, it is possible to show by
induction on the degrees of the $\GeneratingSet$-trees that $\Vp$ satisfies
\begin{subequations}
\begin{equation}
    \Vp(\Leaf) = \sum_{\GenA \in \GeneratingSet} \GenA,
\end{equation}
\begin{equation} \begin{split}
    \Vp\Par{\GenB \BPar{\TreeS_1, \dots, \TreeS_{|\GenB|}}}
    & =
    \Par{
    \sum_{\GenA \in \GeneratingSet}
    \GenA \circ_1 \GenB \BPar{\TreeS_1, \dots, \TreeS_{|\GenB|}}} \\
    & \quad +
    \Par{
    \sum_{j \in [2, |\GenB|]}
    \GenB \BPar{\TreeS_1, \dots, \TreeS_{j - 1}, \Vp\Par{\TreeS_j},
        \TreeS_{j + 1}, \dots, \TreeS_{|\GenB|}}},
\end{split} \end{equation}
\end{subequations}
for any $\GenB \in \GeneratingSet$ and any $\GeneratingSet$-trees $\TreeS_1$, \dots,
$\TreeS_{|\GenB|}$. For instance, by considering the same alphabet $\GeneratingSet$ as in
the previous example, one has for instance
\begin{subequations}
\begin{equation} \begin{split}
    \Vp\Par{\begin{tikzpicture}[Centering,xscale=0.19,yscale=0.24]
        \node(0)at(0.00,-1.67){};
        \node(2)at(2.00,-3.33){};
        \node(4)at(4.00,-3.33){};
        \node[NodeST](1)at(1.00,0.00){$\GenA$};
        \node[NodeST](3)at(3.00,-1.67){$\GenA$};
        \draw[Edge](0)--(1);
        \draw[Edge](2)--(3);
        \draw[Edge](3)--(1);
        \draw[Edge](4)--(3);
        \node(r)at(1.00,1.75){};
        \draw[Edge](r)--(1);
    \end{tikzpicture}}
    & =
    \begin{tikzpicture}[Centering,xscale=0.2,yscale=0.22]
        \node(0)at(0.00,-1.75){};
        \node(2)at(2.00,-3.50){};
        \node(4)at(4.00,-5.25){};
        \node(6)at(6.00,-5.25){};
        \node[NodeST](1)at(1.00,0.00){$\GenA$};
        \node[NodeST](3)at(3.00,-1.75){$\GenA$};
        \node[NodeST](5)at(5.00,-3.50){$\GenA$};
        \draw[Edge](0)--(1);
        \draw[Edge](2)--(3);
        \draw[Edge](3)--(1);
        \draw[Edge](4)--(5);
        \draw[Edge](5)--(3);
        \draw[Edge](6)--(5);
        \node(r)at(1.00,1.75){};
        \draw[Edge](r)--(1);
    \end{tikzpicture}
    +
    \begin{tikzpicture}[Centering,xscale=0.17,yscale=0.2]
        \node(0)at(0.00,-2.00){};
        \node(2)at(2.00,-4.00){};
        \node(4)at(4.00,-6.00){};
        \node(6)at(5.00,-6.00){};
        \node(7)at(6.00,-6.00){};
        \node[NodeST](1)at(1.00,0.00){$\GenA$};
        \node[NodeST](3)at(3.00,-2.00){$\GenA$};
        \node[NodeST](5)at(5.00,-4.00){$\GenC$};
        \draw[Edge](0)--(1);
        \draw[Edge](2)--(3);
        \draw[Edge](3)--(1);
        \draw[Edge](4)--(5);
        \draw[Edge](5)--(3);
        \draw[Edge](6)--(5);
        \draw[Edge](7)--(5);
        \node(r)at(1.00,1.75){};
        \draw[Edge](r)--(1);
    \end{tikzpicture}
    +
    \begin{tikzpicture}[Centering,xscale=0.21,yscale=0.26]
        \node(0)at(0.00,-1.50){};
        \node(2)at(2.00,-3.00){};
        \node(5)at(4.00,-4.50){};
        \node[NodeST](1)at(1.00,0.00){$\GenA$};
        \node[NodeST](3)at(3.00,-1.50){$\GenA$};
        \node[NodeST](4)at(4.00,-3.00){$\GenE$};
        \draw[Edge](0)--(1);
        \draw[Edge](2)--(3);
        \draw[Edge](3)--(1);
        \draw[Edge](4)--(3);
        \draw[Edge](5)--(4);
        \node(r)at(1.00,1.5){};
        \draw[Edge](r)--(1);
    \end{tikzpicture}
    +
    \begin{tikzpicture}[Centering,xscale=0.17,yscale=0.2]
        \node(0)at(0.00,-1.75){};
        \node(2)at(2.00,-5.25){};
        \node(4)at(4.00,-5.25){};
        \node(6)at(6.00,-3.50){};
        \node[NodeST](1)at(1.00,0.00){$\GenA$};
        \node[NodeST](3)at(3.00,-3.50){$\GenA$};
        \node[NodeST](5)at(5.00,-1.75){$\GenA$};
        \draw[Edge](0)--(1);
        \draw[Edge](2)--(3);
        \draw[Edge](3)--(5);
        \draw[Edge](4)--(3);
        \draw[Edge](5)--(1);
        \draw[Edge](6)--(5);
        \node(r)at(1.00,2){};
        \draw[Edge](r)--(1);
    \end{tikzpicture}
    +
    \begin{tikzpicture}[Centering,xscale=0.15,yscale=0.18]
        \node(0)at(0.00,-2.00){};
        \node(2)at(2.00,-6.00){};
        \node(4)at(4.00,-6.00){};
        \node(6)at(5.00,-4.00){};
        \node(7)at(6.00,-4.00){};
        \node[NodeST](1)at(1.00,0.00){$\GenA$};
        \node[NodeST](3)at(3.00,-4.00){$\GenA$};
        \node[NodeST](5)at(5.00,-2.00){$\GenC$};
        \draw[Edge](0)--(1);
        \draw[Edge](2)--(3);
        \draw[Edge](3)--(5);
        \draw[Edge](4)--(3);
        \draw[Edge](5)--(1);
        \draw[Edge](6)--(5);
        \draw[Edge](7)--(5);
        \node(r)at(1.00,2){};
        \draw[Edge](r)--(1);
    \end{tikzpicture}
    +
    \begin{tikzpicture}[Centering,xscale=0.21,yscale=0.26]
        \node(0)at(0.00,-1.50){};
        \node(3)at(2.00,-4.50){};
        \node(5)at(4.00,-4.50){};
        \node[NodeST](1)at(1.00,0.00){$\GenA$};
        \node[NodeST](2)at(3.00,-1.50){$\GenE$};
        \node[NodeST](4)at(3.00,-3.00){$\GenA$};
        \draw[Edge](0)--(1);
        \draw[Edge](2)--(1);
        \draw[Edge](3)--(4);
        \draw[Edge](4)--(2);
        \draw[Edge](5)--(4);
        \node(r)at(1.00,1.5){};
        \draw[Edge](r)--(1);
    \end{tikzpicture}
    \\
    & \quad +
    \begin{tikzpicture}[Centering,xscale=0.17,yscale=0.2]
        \node(0)at(0.00,-3.50){};
        \node(2)at(2.00,-5.25){};
        \node(4)at(4.00,-5.25){};
        \node(6)at(6.00,-1.75){};
        \node[NodeST](1)at(1.00,-1.75){$\GenA$};
        \node[NodeST](3)at(3.00,-3.50){$\GenA$};
        \node[NodeST](5)at(5.00,0.00){$\GenA$};
        \draw[Edge](0)--(1);
        \draw[Edge](1)--(5);
        \draw[Edge](2)--(3);
        \draw[Edge](3)--(1);
        \draw[Edge](4)--(3);
        \draw[Edge](6)--(5);
        \node(r)at(5.00,2){};
        \draw[Edge](r)--(5);
    \end{tikzpicture}
    +
    \begin{tikzpicture}[Centering,xscale=0.15,yscale=0.18]
        \node(0)at(0.00,-4.00){};
        \node(2)at(2.00,-6.00){};
        \node(4)at(4.00,-6.00){};
        \node(6)at(5.00,-2.00){};
        \node(7)at(6.00,-2.00){};
        \node[NodeST](1)at(1.00,-2.00){$\GenA$};
        \node[NodeST](3)at(3.00,-4.00){$\GenA$};
        \node[NodeST](5)at(5.00,0.00){$\GenC$};
        \draw[Edge](0)--(1);
        \draw[Edge](1)--(5);
        \draw[Edge](2)--(3);
        \draw[Edge](3)--(1);
        \draw[Edge](4)--(3);
        \draw[Edge](6)--(5);
        \draw[Edge](7)--(5);
        \node(r)at(5.00,2){};
        \draw[Edge](r)--(5);
    \end{tikzpicture}
    +
    \begin{tikzpicture}[Centering,xscale=0.21,yscale=0.26]
        \node(1)at(0.00,-3.00){};
        \node(3)at(2.00,-4.50){};
        \node(5)at(4.00,-4.50){};
        \node[NodeST](0)at(1.00,0.00){$\GenE$};
        \node[NodeST](2)at(1.00,-1.50){$\GenA$};
        \node[NodeST](4)at(3.00,-3.00){$\GenA$};
        \draw[Edge](1)--(2);
        \draw[Edge](2)--(0);
        \draw[Edge](3)--(4);
        \draw[Edge](4)--(2);
        \draw[Edge](5)--(4);
        \node(r)at(1.00,1.5){};
        \draw[Edge](r)--(0);
    \end{tikzpicture},
\end{split} \end{equation}
\begin{equation} \begin{split}
    \Vp\Par{
    \begin{tikzpicture}[Centering,xscale=0.2,yscale=0.18]
        \node(0)at(0.00,-4.67){};
        \node(2)at(2.00,-4.67){};
        \node(4)at(3.00,-2.33){};
        \node(6)at(4.00,-4.67){};
        \node[NodeST](1)at(1.00,-2.33){$\GenA$};
        \node[NodeST](3)at(3.00,0.00){$\GenC$};
        \node[NodeST](5)at(4.00,-2.33){$\GenE$};
        \draw[Edge](0)--(1);
        \draw[Edge](1)--(3);
        \draw[Edge](2)--(1);
        \draw[Edge](4)--(3);
        \draw[Edge](5)--(3);
        \draw[Edge](6)--(5);
        \node(r)at(3.00,2){};
        \draw[Edge](r)--(3);
    \end{tikzpicture}}
    & =
    \begin{tikzpicture}[Centering,xscale=0.2,yscale=0.18]
        \node(0)at(0.00,-6.75){};
        \node(2)at(2.00,-6.75){};
        \node(4)at(3.00,-4.50){};
        \node(6)at(4.00,-6.75){};
        \node(8)at(6.00,-2.25){};
        \node[NodeST](1)at(1.00,-4.50){$\GenA$};
        \node[NodeST](3)at(3.00,-2.25){$\GenC$};
        \node[NodeST](5)at(4.00,-4.50){$\GenE$};
        \node[NodeST](7)at(5.00,0.00){$\GenA$};
        \draw[Edge](0)--(1);
        \draw[Edge](1)--(3);
        \draw[Edge](2)--(1);
        \draw[Edge](3)--(7);
        \draw[Edge](4)--(3);
        \draw[Edge](5)--(3);
        \draw[Edge](6)--(5);
        \draw[Edge](8)--(7);
        \node(r)at(5.00,2){};
        \draw[Edge](r)--(7);
    \end{tikzpicture}
    +
    \begin{tikzpicture}[Centering,xscale=0.2,yscale=0.18]
        \node(0)at(0.00,-4.50){};
        \node(2)at(2.00,-4.50){};
        \node(4)at(3.00,-2.25){};
        \node(6)at(4.00,-6.75){};
        \node(8)at(6.00,-4.50){};
        \node[NodeST](1)at(1.00,-2.25){$\GenA$};
        \node[NodeST](3)at(3.00,0.00){$\GenC$};
        \node[NodeST](5)at(4.00,-4.50){$\GenE$};
        \node[NodeST](7)at(5.00,-2.25){$\GenA$};
        \draw[Edge](0)--(1);
        \draw[Edge](1)--(3);
        \draw[Edge](2)--(1);
        \draw[Edge](4)--(3);
        \draw[Edge](5)--(7);
        \draw[Edge](6)--(5);
        \draw[Edge](7)--(3);
        \draw[Edge](8)--(7);
        \node(r)at(3.00,2){};
        \draw[Edge](r)--(3);
    \end{tikzpicture}
    +
    \begin{tikzpicture}[Centering,xscale=0.18,yscale=0.17]
        \node(0)at(0.00,-5.00){};
        \node(2)at(2.00,-5.00){};
        \node(4)at(3.00,-2.50){};
        \node(6)at(4.00,-7.50){};
        \node(8)at(5.00,-5.00){};
        \node(9)at(6.00,-5.00){};
        \node[NodeST](1)at(1.00,-2.50){$\GenA$};
        \node[NodeST](3)at(3.00,0.00){$\GenC$};
        \node[NodeST](5)at(4.00,-5.00){$\GenE$};
        \node[NodeST](7)at(5.00,-2.50){$\GenC$};
        \draw[Edge](0)--(1);
        \draw[Edge](1)--(3);
        \draw[Edge](2)--(1);
        \draw[Edge](4)--(3);
        \draw[Edge](5)--(7);
        \draw[Edge](6)--(5);
        \draw[Edge](7)--(3);
        \draw[Edge](8)--(7);
        \draw[Edge](9)--(7);
        \node(r)at(3.00,2.25){};
        \draw[Edge](r)--(3);
    \end{tikzpicture}
    +
    \begin{tikzpicture}[Centering,xscale=0.19,yscale=0.2]
        \node(0)at(0.00,-4.00){};
        \node(2)at(2.00,-4.00){};
        \node(4)at(3.00,-2.00){};
        \node(7)at(4.00,-6.00){};
        \node[NodeST](1)at(1.00,-2.00){$\GenA$};
        \node[NodeST](3)at(3.00,0.00){$\GenC$};
        \node[NodeST](5)at(4.00,-2.00){$\GenE$};
        \node[NodeST](6)at(4.00,-4.00){$\GenE$};
        \draw[Edge](0)--(1);
        \draw[Edge](1)--(3);
        \draw[Edge](2)--(1);
        \draw[Edge](4)--(3);
        \draw[Edge](5)--(3);
        \draw[Edge](6)--(5);
        \draw[Edge](7)--(6);
        \node(r)at(3.00,2){};
        \draw[Edge](r)--(3);
    \end{tikzpicture}
    +
    \begin{tikzpicture}[Centering,xscale=0.17,yscale=0.14]
        \node(0)at(0.00,-6.00){};
        \node(2)at(2.00,-6.00){};
        \node(4)at(3.00,-6.00){};
        \node(6)at(5.00,-6.00){};
        \node(8)at(6.00,-6.00){};
        \node[NodeST](1)at(1.00,-3.00){$\GenA$};
        \node[NodeST](3)at(4.00,0.00){$\GenC$};
        \node[NodeST](5)at(4.00,-3.00){$\GenA$};
        \node[NodeST](7)at(6.00,-3.00){$\GenE$};
        \draw[Edge](0)--(1);
        \draw[Edge](1)--(3);
        \draw[Edge](2)--(1);
        \draw[Edge](4)--(5);
        \draw[Edge](5)--(3);
        \draw[Edge](6)--(5);
        \draw[Edge](7)--(3);
        \draw[Edge](8)--(7);
        \node(r)at(4.00,3){};
        \draw[Edge](r)--(3);
    \end{tikzpicture}
    +
    \begin{tikzpicture}[Centering,xscale=0.17,yscale=0.13]
        \node(0)at(0.00,-6.67){};
        \node(2)at(2.00,-6.67){};
        \node(4)at(3.00,-6.67){};
        \node(6)at(4.00,-6.67){};
        \node(7)at(5.00,-6.67){};
        \node(9)at(6.00,-6.67){};
        \node[NodeST](1)at(1.00,-3.33){$\GenA$};
        \node[NodeST](3)at(4.00,0.00){$\GenC$};
        \node[NodeST](5)at(4.00,-3.33){$\GenC$};
        \node[NodeST](8)at(6.00,-3.33){$\GenE$};
        \draw[Edge](0)--(1);
        \draw[Edge](1)--(3);
        \draw[Edge](2)--(1);
        \draw[Edge](4)--(5);
        \draw[Edge](5)--(3);
        \draw[Edge](6)--(5);
        \draw[Edge](7)--(5);
        \draw[Edge](8)--(3);
        \draw[Edge](9)--(8);
        \node(r)at(4.00,3){};
        \draw[Edge](r)--(3);
    \end{tikzpicture}
    \\
    & \quad +
    \begin{tikzpicture}[Centering,xscale=0.21,yscale=0.17]
        \node(0)at(0.00,-5.33){};
        \node(2)at(2.00,-5.33){};
        \node(5)at(3.00,-5.33){};
        \node(7)at(4.00,-5.33){};
        \node[NodeST](1)at(1.00,-2.67){$\GenA$};
        \node[NodeST](3)at(3.00,0.00){$\GenC$};
        \node[NodeST](4)at(3.00,-2.67){$\GenE$};
        \node[NodeST](6)at(4.00,-2.67){$\GenE$};
        \draw[Edge](0)--(1);
        \draw[Edge](1)--(3);
        \draw[Edge](2)--(1);
        \draw[Edge](4)--(3);
        \draw[Edge](5)--(4);
        \draw[Edge](6)--(3);
        \draw[Edge](7)--(6);
        \node(r)at(3.00,2.25){};
        \draw[Edge](r)--(3);
    \end{tikzpicture}
    +
    \begin{tikzpicture}[Centering,xscale=0.18,yscale=0.17]
        \node(0)at(0.00,-7.50){};
        \node(2)at(2.00,-7.50){};
        \node(4)at(3.00,-5.00){};
        \node(6)at(4.00,-7.50){};
        \node(8)at(5.00,-2.50){};
        \node(9)at(6.00,-2.50){};
        \node[NodeST](1)at(1.00,-5.00){$\GenA$};
        \node[NodeST](3)at(3.00,-2.50){$\GenC$};
        \node[NodeST](5)at(4.00,-5.00){$\GenE$};
        \node[NodeST](7)at(5.00,0.00){$\GenC$};
        \draw[Edge](0)--(1);
        \draw[Edge](1)--(3);
        \draw[Edge](2)--(1);
        \draw[Edge](3)--(7);
        \draw[Edge](4)--(3);
        \draw[Edge](5)--(3);
        \draw[Edge](6)--(5);
        \draw[Edge](8)--(7);
        \draw[Edge](9)--(7);
        \node(r)at(5.00,2.25){};
        \draw[Edge](r)--(7);
    \end{tikzpicture}
    +
    \begin{tikzpicture}[Centering,xscale=0.19,yscale=0.22]
        \node(1)at(0.00,-6.00){};
        \node(3)at(2.00,-6.00){};
        \node(5)at(3.00,-4.00){};
        \node(7)at(4.00,-6.00){};
        \node[NodeST](0)at(2.00,0.00){$\GenE$};
        \node[NodeST](2)at(1.00,-4.00){$\GenA$};
        \node[NodeST](4)at(3.00,-2.00){$\GenA$};
        \node[NodeST](6)at(4.00,-4.00){$\GenE$};
        \draw[Edge](1)--(2);
        \draw[Edge](2)--(4);
        \draw[Edge](3)--(2);
        \draw[Edge](4)--(0);
        \draw[Edge](5)--(4);
        \draw[Edge](6)--(4);
        \draw[Edge](7)--(6);
        \node(r)at(2.00,1.8){};
        \draw[Edge](r)--(0);
    \end{tikzpicture}.
\end{split} \end{equation}
\end{subequations}
\medbreak

Observe that when $\GeneratingSet$ contains at least one letter $\GenA$ such that $|\GenA|
\geq 2$, $\Par{\SyntaxTreesInternalNode(\GeneratingSet), \Vp}$ is not $\phi$-diagonal
self-dual. Indeed,
\begin{equation} \begin{split}
    \Par{\Vp^\Dual \Vp - \Vp \Vp^\Dual}\Par{\GenA \circ_1 \GenA}
    & =
    \Vp^\Dual\Par{
    \Par{\sum_{\GenB \in \GeneratingSet}
        \GenB \circ_1 \Par{\GenA \circ_1 \GenA}}
        +
    \Par{\sum_{\substack{
        j \in [2, |\GenA|] \\
        \GenB \in \GeneratingSet}} 
    \Par{\GenA \circ_j \GenB} \circ_1 \GenA}} - \Vp(\GenA)
    \\
    & =
    (\# \GeneratingSet) \; \GenA \circ_1 \GenA
    + (\# \GeneratingSet) (|\GenA| - 1) \; \GenA \circ_1 \GenA
    - \Par{\sum_{\GenB \in \GeneratingSet} \GenB \circ_1 \GenA}
    - \Par{\sum_{\substack{
        j \in [2, |\GenA|] \\
        \GenB \in \GeneratingSet}} 
    \GenA \circ_j \GenB} \\
    & =
    (\# \GeneratingSet) |\GenA| \; \GenA \circ_1 \GenA
    - \Par{\sum_{\GenB \in \GeneratingSet} \GenB \circ_1 \GenA}
    - \Par{\sum_{\substack{
        j \in [2, |\GenA|] \\
        \GenB \in \GeneratingSet}} 
    \GenA \circ_j \GenB}.
\end{split} \end{equation}
\medbreak

Besides, in the particular case where
\begin{math}
    \GeneratingSet = \GeneratingSet(2) = \Bra{\GenA},
\end{math}
the map $\Vp^\Dual$ satisfies the recurrence
\begin{subequations}
\begin{equation}
    \Vp^\Dual(\Leaf) = 0,
\end{equation}
\begin{equation}
    \Vp^\Dual\Par{\GenA \BPar{\TreeS, \Leaf}} = \TreeS,
\end{equation}
\begin{equation}
    \Vp^\Dual\Par{\GenA \BPar{\TreeS_1, \TreeS_2}}
    = \GenA \BPar{\TreeS_1, \Vp^\Dual\Par{\TreeS_2}},
\end{equation}
\end{subequations}
for any $\GeneratingSet$-trees $\TreeS$, $\TreeS_1$, and $\TreeS_2$ such that $\TreeS_2 \ne
\Leaf$. It follows by induction on the degrees of the $\GeneratingSet$-trees that for any
$\GeneratingSet$-tree $\TreeT$, there is at most one term appearing in $\Vp^\Dual(\TreeT)$.
For this reason, the graph $(\SyntaxTrees(\GeneratingSet), \Vp)$ is a tree.
\medbreak

%%%%%%%%%%%%%%%%%%%%%%%%%%%%%%%%%%%%%%%%%%%%%%%%%%%%%%%%%%%%%%%%%%%%%%%%%%%%%%%%%%%%%%%%%%%%
\subsubsection{Path enumeration}
The \Def{twisted poset induced} by a $\GeneratingSet$-tree $\TreeT$ is the poset
$\TwistedTreePoset{\TreeT} := \Par{\InternalNodes(\TreeT), \Leq'}$ wherein for any $u, v \in
\InternalNodes(\TreeT)$, one has $u \Leq' v$ if $u = v$, or $v = u i v'$ where $i \geq 2$
and $v'$ is any word of integers, or $u = v 1 u'$ where $u'$ is any word of integers. In
other words, this says that one has $u \Leq' v$ if $u = v$, or $u$ is an ancestor of $v$ but
$v$ is not in the first subtree of $u$, or $u$ is in the first subtree of $v$. For
instance, by considering the same alphabet $\GeneratingSet$ as in the previous examples, in
the twisted poset induced by the $\GeneratingSet$-tree
\begin{equation} \label{equ:example_tree_twisted_hook}
    \TreeT :=
    \begin{tikzpicture}[Centering,xscale=0.21,yscale=0.14]
        \node(0)at(0.00,-10.80){};
        \node(10)at(6.00,-7.20){};
        \node(11)at(7.00,-7.20){};
        \node(13)at(8.00,-14.40){};
        \node(15)at(10.00,-14.40){};
        \node(17)at(12.00,-10.80){};
        \node(2)at(1.00,-10.80){};
        \node(3)at(2.00,-10.80){};
        \node(6)at(4.00,-10.80){};
        \node(8)at(5.00,-7.20){};
        \node[NodeST](1)at(1.00,-7.20){$\GenC$};
        \node[NodeST](12)at(10.00,-3.60){$\GenE$};
        \node[NodeST](14)at(9.00,-10.80){$\GenA$};
        \node[NodeST](16)at(11.00,-7.20){$\GenA$};
        \node[NodeST](4)at(3.00,-3.60){$\GenA$};
        \node[NodeST](5)at(4.00,-7.20){$\GenE$};
        \node[NodeST](7)at(6.00,0.00){$\GenC$};
        \node[NodeST](9)at(6.00,-3.60){$\GenC$};
        \draw[Edge](0)--(1);
        \draw[Edge](1)--(4);
        \draw[Edge](10)--(9);
        \draw[Edge](11)--(9);
        \draw[Edge](12)--(7);
        \draw[Edge](13)--(14);
        \draw[Edge](14)--(16);
        \draw[Edge](15)--(14);
        \draw[Edge](16)--(12);
        \draw[Edge](17)--(16);
        \draw[Edge](2)--(1);
        \draw[Edge](3)--(1);
        \draw[Edge](4)--(7);
        \draw[Edge](5)--(4);
        \draw[Edge](6)--(5);
        \draw[Edge](8)--(9);
        \draw[Edge](9)--(7);
        \node(r)at(6.00,2.70){};
        \draw[Edge](r)--(7);
    \end{tikzpicture}
\end{equation}
we have $11 \Leq ' 1 \Leq' \epsilon$, $12 \Leq' \epsilon$, $1 \Leq' 12$, $\epsilon \Leq' 2$,
$\epsilon \Leq' 3$, $\epsilon \Leq' 31$, $\epsilon \Leq ' 311$, and $311 \Leq ' 31 \Leq' 3$.
\medbreak

For any $\GeneratingSet$ tree $\TreeT$, let the statistics $\TreeT \mapsto
\TwistedHook{\TreeT}$ satisfying the recurrence relation
\begin{subequations}
\begin{equation}
    \TwistedHook{\Leaf} = 1,
\end{equation}
\begin{equation} \label{equ:twisted_hook_2}
    \TwistedHook{\GenA \BPar{\TreeS_1, \dots, \TreeS_{|\GenA|}}}
    = \lbag \Deg\Par{\TreeS_2}, \dots, \Deg\Par{\TreeS_{|\GenA|}} \rbag !
    \prod_{i \in [|\GenA|]} \TwistedHook{\TreeS_i},
\end{equation}
\end{subequations}
for any $\GenA \in \GeneratingSet$ and any $\GeneratingSet$-trees $\TreeS_1$, \dots,
$\TreeS_{|\GenA|}$.  For instance, by considering the $\GeneratingSet$-tree $\TreeT$ defined
in~\eqref{equ:example_tree_twisted_hook}, one has $\TwistedHook{\TreeT} = 4$.
\medbreak

\begin{Lemma} \label{lem:number_linear_extensions_twisted_posets}
    For any alphabet $\GeneratingSet$ and any $\GeneratingSet$-tree $\TreeT$, the number of
    linear extensions of the poset $\TwistedTreePoset{\TreeT}$ is $\TwistedHook{\TreeT}$.
\end{Lemma}
\begin{proof}
    We proceed by induction on the degree $d$ of $\TreeT$. If $d = 0$, then $\TreeT =
    \Leaf$, and since $\TwistedTreePoset{\Leaf}$ has exactly one linear extension which is
    the empty one, and $\TwistedHook{\TreeT} = 1$, the property is satisfied. If $d \geq 1$,
    then $\TreeT = \GenA \BPar{\TreeS_1, \dots, \TreeS_{|\GenA|}}$ for an $\GenA \in
    \GeneratingSet$ and $\GeneratingSet$-trees $\TreeS_1$, \dots, $\TreeS_{|\GenA|}$.  From
    the definition of $\TwistedTreePoset{\TreeT}$, it follows that any permutation encoding
    a linear extension of $\TwistedTreePoset{\TreeT}$ writes as
    \begin{equation}
        \Par{1 u^{(1)}, \dots, 1 u^{(k)}, \epsilon, i_1 v^{(1)}, \dots, i_\ell v^{(\ell)}}
    \end{equation}
    where the permutation $\Par{u^{(1)}, \dots, u^{(k)}}$ encodes a linear extension of
    $\TwistedTreePoset{\TreeS_1}$ and the permutation $\Par{\epsilon, i_1 v^{(1)}, \dots,
    i_\ell v^{(\ell)}}$ encodes a linear extension of $\TwistedTreePoset{\GenA
    \BPar{\Leaf,\TreeS_2, \dots, \TreeS_{|\GenA|}}}$.  Indeed, by definition of the order
    relation $\Leq'$ of $\TwistedTreePoset{\TreeT}$, all the nodes $u^{(1)}$, \dots,
    $u^{(k)}$ of $\TreeS_1$ are smaller than the root $\epsilon$ of $\TreeT$, and the root
    of $\TreeT$ is itself smaller than all the nodes $v^{(1)}$, \dots, $v^{(\ell)}$ of,
    respectively, $\TwistedTreePoset{\TreeS_2}$, \dots,
    $\TwistedTreePoset{\TreeS_{|\GenA|}}$. Moreover, again by definition of $\Leq'$, for any
    $j, j' \in [2, |\GenA|]$, $v \in \InternalNodes\Par{\TreeS_j}$, and $v' \in
    \InternalNodes\Par{\TreeS_{j'}}$, if $j \ne j'$ then the nodes $jv$ and $j'v'$ of
    $\TreeT$ are incomparable in $\TwistedTreePoset{\TreeT}$. Thus,
    by~\eqref{equ:shuffle_linear_extensions}, $\TwistedTreePoset{\GenA \BPar{\Leaf,\TreeS_2,
    \dots, \TreeS_{|\GenA|}}}$ has
    \begin{math}
        \lbag \Deg\Par{\TreeS_2}, \dots, \Deg\Par{\TreeS_{|\GenA|}} \rbag !
    \end{math}
    linear extensions.  Finally, since each $\TwistedTreePoset{\TreeS_i}$, $i \in
    [|\GenA|]$, has by induction hypothesis $\TwistedHook{\TreeS_i}$ linear extensions, it
    follows from~\eqref{equ:twisted_hook_2} that $\TwistedTreePoset{\TreeT}$ admits
    $\TwistedHook{\TreeT}$ linear extensions.
\end{proof}
\medbreak

By Lemma~\ref{lem:number_linear_extensions_twisted_posets}, and in pursuit of the previous
example, the $\GeneratingSet$-tree defined in~\eqref{equ:example_tree_twisted_hook} has four
linear extensions. The four permutations encoding these are
\begin{subequations}
\begin{equation}
    \Par{11, 1, 12, \epsilon, 2, 311, 31, 3},
\end{equation}
\begin{equation}
    \Par{11, 1, 12, \epsilon, 311, 2, 31, 3},
\end{equation}
\begin{equation}
    \Par{11, 1, 12, \epsilon, 311, 31, 2, 3},
\end{equation}
\begin{equation}
    \Par{11, 1, 12, \epsilon, 311, 31, 3, 2}.
\end{equation}
\end{subequations}
\medbreak

\begin{Proposition} \label{prop:twisted_prefix_tree_graph_hook_series}
    For any finite alphabet $\GeneratingSet$, the hook series of
    $\Par{\SyntaxTreesInternalNode(\GeneratingSet), \Vp}$ satisfies $\Angle{\TreeT,
    \HookSeries{\Vp}} = \TwistedHook{\TreeT}$ for any $\GeneratingSet$-tree $\TreeT$.
\end{Proposition}
\begin{proof}
    Let us proceed by induction on the degree of the $\GeneratingSet$-trees. By definition
    of $\HookSeries{\Vp}$, $\Angle{\Leaf, \HookSeries{\Vp}} = 1$, and since
    $\TwistedHook{\Leaf} = 1$, the property holds. Now, let $\TreeT$ be a
    $\GeneratingSet$-tree of degree $d \geq 1$. We have, by definition of
    $\HookSeries{\Vp}$, by Proposition~\ref{prop:twisted_prefix_tree_graph_adjoint_direct},
    and by induction hypothesis,
    \begin{equation} \begin{split}
        \Angle{\TreeT, \HookSeries{\Vp}}
        & =
        \Angle{\Vp^\Dual(\TreeT), \HookSeries{\Vp}} \\
        & =
        \Angle{
            \sum_{u \in \InternalNodesQuasiMax(\TreeT)} \ContractNode{\TreeT}{u},
            \HookSeries{\Vp}}
        \\
        & =
        \sum_{u \in \InternalNodesQuasiMax(\TreeT)}
        \Angle{\ContractNode{\TreeT}{u}, \HookSeries{\Vp}}
        \\
        & =
        \sum_{u \in \InternalNodesQuasiMax(\TreeT)} \TwistedHook{\ContractNode{\TreeT}{u}}.
    \end{split} \end{equation}
    From the definition of $\TwistedTreePoset{\TreeT}$, it follows that any permutation
    encoding a linear extension of $\TwistedTreePoset{\TreeT}$ writes as
    \begin{equation}
        \Par{v^{(1)}, \dots, v^{(d - 1)}, u}
    \end{equation}
    where the permutation $\Par{v^{(1)}, \dots, v^{(d - 1)}}$ encodes a linear extension of
    $\TwistedTreePoset{\ContractNode{\TreeT}{u}}$ and $u$ is a quasi-maximal node of
    $\TreeT$.  This leads, by using Lemma~\ref{lem:number_linear_extensions_twisted_posets},
    to $\Angle{\TreeT, \HookSeries{\Vp}} = \TwistedHook{\TreeT}$.
\end{proof}
\medbreak

Here are the sequences of the first coefficients of some initial paths series of
$\Par{\SyntaxTreesInternalNode(\GeneratingSet), \Vp}$:
\begin{subequations}
\begin{equation}
    1, 1, 2, 5, 14, 42, 132, 429,
    \qquad \mbox{ for } \GeneratingSet = \{\GenA\} \mbox{ with } |\GenA| = 2,
\end{equation}
\begin{equation}
    1, 1, 3, 13, 71, 465, 3563, 31429,
    \qquad \mbox{ for } \GeneratingSet = \{\GenC\} \mbox{ with } |\GenC| = 3,
\end{equation}
\begin{equation}
    1, 2, 8, 40, 224, 1344, 8448, 54912,
    \qquad \mbox{ for } \GeneratingSet = \{\GenA, \GenB\} \mbox{ with }
    |\GenA| = |\GenB| = 2,
\end{equation}
\begin{equation}
    1, 2, 10, 70, 606, 6210, 73842, 1006318,
    \qquad \mbox{ for } \GeneratingSet = \{\GenA, \GenC\} \mbox{ with } |\GenA| = 2,
    |\GenC| = 3.
\end{equation}
\end{subequations}
The first and third ones are respectively Sequences~\OEIS{A000108} and~\OEIS{A151374}
of~\cite{Slo}. The two other ones do not appear for the time being in~\cite{Slo}.
\medbreak

%%%%%%%%%%%%%%%%%%%%%%%%%%%%%%%%%%%%%%%%%%%%%%%%%%%%%%%%%%%%%%%%%%%%%%%%%%%%%%%%%%%%%%%%%%%%
%%%%%%%%%%%%%%%%%%%%%%%%%%%%%%%%%%%%%%%%%%%%%%%%%%%%%%%%%%%%%%%%%%%%%%%%%%%%%%%%%%%%%%%%%%%%
\subsection{Diagonal duality}
We prove here that the pair of graded graphs consisting in the $\GeneratingSet$-prefix graph
and the $\GeneratingSet$-twisted prefix graph is $\phi$-diagonal dual. The description of
the map $\phi$ requires the use of a particular statistics on $\GeneratingSet$-trees
introduced here.
\medbreak

%%%%%%%%%%%%%%%%%%%%%%%%%%%%%%%%%%%%%%%%%%%%%%%%%%%%%%%%%%%%%%%%%%%%%%%%%%%%%%%%%%%%%%%%%%%%
\subsubsection{Non-first leaves statistics}
Let $\GeneratingSet$ be an alphabet and $\TreeT$ be a $\GeneratingSet$-tree.  A leaf $u$ of
$\TreeT$ is \Def{non-first} if $u$ admits no occurrence of the letter $1$.  In other words,
the path connecting the root of $\TreeT$ with $u$ does never go through a first edge of an
internal node. We denote by $\LeavesNonFirst(\TreeT)$ the set of all non-first leaves of
$\TreeT$. For instance, by considering the same alphabet $\GeneratingSet$ as in the previous
examples, here is a $\GeneratingSet$-tree wherein its non-first leaves are framed:
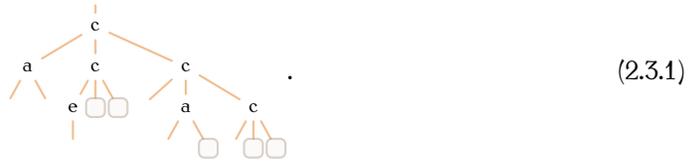
\begin{equation}
    \begin{tikzpicture}[Centering,xscale=0.30,yscale=0.12]
        \node(0)at(0.00,-9.00){};
        \node(11)at(7.00,-13.50){};
        \node(13)at(9.00,-13.50){};
        \node(14)at(10.00,-13.50){};
        \node(16)at(11.00,-13.50){};
        \node(17)at(12.00,-13.50){};
        \node(2)at(2.00,-9.00){};
        \node(5)at(3.00,-13.50){};
        \node(7)at(4.00,-9.00){};
        \node(8)at(5.00,-9.00){};
        \node(9)at(6.00,-9.00){};
        \node[NodeST](1)at(1.00,-4.50){$\GenA$};
        \node[NodeST](10)at(8.00,-4.50){$\GenC$};
        \node[NodeST](12)at(8.00,-9.00){$\GenA$};
        \node[NodeST](15)at(11.00,-9.00){$\GenC$};
        \node[NodeST](3)at(4.00,0.00){$\GenC$};
        \node[NodeST](4)at(3.00,-9.00){$\GenE$};
        \node[NodeST](6)at(4.00,-4.50){$\GenC$};
        \draw[Edge](0)--(1);
        \draw[Edge](1)--(3);
        \draw[Edge](10)--(3);
        \draw[Edge](11)--(12);
        \draw[Edge](12)--(10);
        \draw[Edge](13)--(12);
        \draw[Edge](14)--(15);
        \draw[Edge](15)--(10);
        \draw[Edge](16)--(15);
        \draw[Edge](17)--(15);
        \draw[Edge](2)--(1);
        \draw[Edge](4)--(6);
        \draw[Edge](5)--(4);
        \draw[Edge](6)--(3);
        \draw[Edge](7)--(6);
        \draw[Edge](8)--(6);
        \draw[Edge](9)--(10);
        \node(r)at(4.00,3.38){};
        \draw[Edge](r)--(3);
        \node[fit=(7),Box,inner sep=0pt,opacity=.4]{};
        \node[fit=(8),Box,inner sep=0pt,opacity=.4]{};
        \node[fit=(13),Box,inner sep=0pt,opacity=.4]{};
        \node[fit=(16),Box,inner sep=0pt,opacity=.4]{};
        \node[fit=(17),Box,inner sep=0pt,opacity=.4]{};
    \end{tikzpicture}.
\end{equation}
Moreover, let us define the \Def{non-first leaves} statistics $\TreeT \mapsto
\NonFirstLeaves{\TreeT}$ by setting $\NonFirstLeaves{\TreeT} = \# \LeavesNonFirst(\TreeT)$.
Immediately from the definition of non-first leaves, this statistics on
$\SyntaxTrees(\GeneratingSet)$ satisfies the recurrence
\begin{subequations}
\begin{equation}
    \NonFirstLeaves{\Leaf} = 1,
\end{equation}
\begin{equation}
    \NonFirstLeaves{\GenA \BPar{\TreeS, \Leaf, \dots, \Leaf}}  = |\GenA| - 1,
\end{equation}
\begin{equation} \label{equ:non_first_leaves_reccurrence_3}
    \NonFirstLeaves{\GenA \BPar{\TreeS_1, \dots, \TreeS_{|\GenA|}}}
    = \sum_{j \in [2, |\GenA|]} \NonFirstLeaves{\TreeS_j},
\end{equation}
\end{subequations}
for any $\GenA \in \GeneratingSet$ and any $\GeneratingSet$-trees $\TreeS$, $\TreeS_1$,
\dots, $\TreeS_{|\GenA|}$ such that there is at least a $j \in [2, |\GenA|]$
satisfying~$\TreeS_j \ne \Leaf$.
\medbreak

%%%%%%%%%%%%%%%%%%%%%%%%%%%%%%%%%%%%%%%%%%%%%%%%%%%%%%%%%%%%%%%%%%%%%%%%%%%%%%%%%%%%%%%%%%%%
\subsubsection{Diagonal duality}

\begin{Theorem} \label{thm:prefix_tree_graphs_dual}
    For any finite alphabet $\GeneratingSet$, the pair of graded graphs
    $(\SyntaxTreesInternalNode(\GeneratingSet), \Up, \Vp)$ is $\phi$-diagonal dual for the
    linear map
    \begin{math}
        \phi : \K \Angle{\SyntaxTreesInternalNode(\GeneratingSet)}
        \to \K \Angle{\SyntaxTreesInternalNode(\GeneratingSet)}
    \end{math}
    satisfying
    \begin{equation}
        \phi(\TreeT) = \Par{\# \GeneratingSet} \NonFirstLeaves{\TreeT} \; \TreeT
    \end{equation}
    for any $\GeneratingSet$-tree $\TreeT$.
\end{Theorem}
\begin{proof}
    Let us consider the $\GeneratingSet$-tree polynomial
    \begin{equation}
        f(\TreeT)
        :=
        \Par{\Vp^\Dual \Up - \Up \Vp^\Dual}\Par{\TreeT}
        =
        \sum_{\substack{
            \GenB \in \GeneratingSet \\
            i \geq 0
        }}
        \Vp^\Dual\Par{\TreeT \circ_i \GenB} - \Vp^\Dual(\TreeT) \circ_i \GenB.
    \end{equation}
    Notice that we use here the convention exposed in Section~\ref{subsubsec:linear_operads}
    about extension by linearity of the composition maps of operads in order to write the
    sum~\eqref{equ:definition_up_trees} without bounding $i$.  Nevertheless, this sum is
    finite.  We proceed by structural induction on $\GeneratingSet$-trees to show that
    $f(\TreeT) = \phi(\TreeT)$.
    \smallbreak

    We have to consider three cases following $\TreeT$.  First, when $\TreeT = \Leaf$, we
    immediately have $f(\TreeT) = \Par{\# \GeneratingSet} \Leaf$. Since
    $\NonFirstLeaves{\Leaf} = 1$, the property is satisfied.  Second, when $\TreeT$ is of
    the form $\TreeT = \GenA \BPar{\TreeS_1, \Leaf, \dots, \Leaf}$ for an $\GenA \in
    \GeneratingSet$ and a $\GeneratingSet$-tree $\TreeS_1$, we obtain
    \begin{equation} \begin{split}
        f(\TreeT) & =
        \sum_{\substack{
            \GenB \in \GeneratingSet \\
            i \geq 0
        }}
        \Vp^\Dual\Par{\GenA \BPar{\TreeS_1, \Leaf, \dots, \Leaf}
            \circ_i \GenB}
        - \TreeS_1 \circ_i \GenB \\
        & =
        \Par{
        \sum_{\substack{
            \GenB \in \GeneratingSet \\
            i \geq 0
        }}
        \Vp^\Dual\Par{\GenA \BPar{\TreeS_1 \circ_i \GenB, \Leaf, \dots, \Leaf}}
        - \TreeS_1 \circ_i \GenB}
        +
        \Par{
        \sum_{\GenB \in \GeneratingSet} \sum_{j \in [2, |\GenA|]}
        \Vp^\Dual\Par{\GenA \BPar{\TreeS_1,
        \underbrace{\Leaf, \dots, \Leaf}_{j - 2}, \GenB,
        \Leaf, \dots, \Leaf}}} \\
        & =
        \Par{
        \sum_{\substack{
            \GenB \in \GeneratingSet \\
            i \geq 0
        }}
        \TreeS_1 \circ_i \GenB - \TreeS_1 \circ_i \GenB}
        +
        \Par{
        \sum_{\GenB \in \GeneratingSet} \sum_{j \in [2, |\GenA|]}
        \GenA \BPar{\TreeS_1,
        \underbrace{\Leaf, \dots, \Leaf}_{j - 2}, \Vp^\Dual(\GenB),
        \Leaf, \dots, \Leaf}}\\
        & =
        \sum_{\GenB \in \GeneratingSet}
        \sum_{j \in [2, |\GenA|]}
        \GenA \BPar{\TreeS_1, \Leaf, \dots, \Leaf} \\
        & =  \Par{\# \GeneratingSet} (|\GenA| - 1) \TreeT.
    \end{split} \end{equation}
    Since $\TreeT$ has the considered form, $\NonFirstLeaves{\TreeT} = |\GenA| - 1$. Hence,
    $f(\TreeT) =\Par{\# \GeneratingSet} \NonFirstLeaves{\TreeT} \; \TreeT$, so that the
    property is satisfied. Finally, it remains to consider the case where $\TreeT$ is of the
    form
    \begin{math}
        \TreeT = \GenA \BPar{\TreeS_1, \dots, \TreeS_{|\GenA|}}
    \end{math}
    for an $\GenA \in \GeneratingSet$ and for $\GeneratingSet$-trees $\TreeS_1$, \dots,
    $\TreeS_{|\GenA|}$ where there is at least a $j \in [2, |\GenA|]$ such that
    $\TreeS_j \ne \Leaf$. In this case, we obtain
    \begin{equation}\begin{split}
    \label{equ:prefix_tree_graphs_dual_1}
        f(\TreeT) & =
        \sum_{\substack{
            \GenB \in \GeneratingSet \\
            i \geq 0
        }}
        \sum_{j \in [2, |\GenA|]}
        \GenA \BPar{\TreeS_1 \circ_i \GenB, \dots,
            \TreeS_{j - 1}, \Vp^\Dual\Par{\TreeS_j}, \TreeS_{j + 1},
            \dots, \TreeS_{|\GenA|}}
        + \cdots \\
        & \quad +
        \GenA \BPar{\TreeS_1, \dots,
            \TreeS_{j - 1}, \Vp^\Dual\Par{\TreeS_j \circ_i \GenB},
            \TreeS_{j + 1}, \dots, \TreeS_{|\GenA|}} + \cdots
        +
        \GenA \BPar{\TreeS_1, \dots,
            \TreeS_{j - 1}, \Vp^\Dual\Par{\TreeS_j},
            \TreeS_{j + 1}, \dots, \TreeS_{|\GenA|} \circ_i \GenB} \\
        & \quad -
        \GenA \BPar{\TreeS_1 \circ_i \GenB, \dots,
            \TreeS_{j - 1}, \Vp^\Dual\Par{\TreeS_j}, \TreeS_{j + 1},
            \dots, \TreeS_{|\GenA|}}
        - \cdots \\
        & \quad -
        \GenA \BPar{\TreeS_1, \dots,
            \TreeS_{j - 1}, \Vp^\Dual\Par{\TreeS_j} \circ_i \GenB,
            \TreeS_{j + 1}, \dots, \TreeS_{|\GenA|}}
        - \cdots
        -
        \GenA \BPar{\TreeS_1, \dots,
            \TreeS_{j - 1}, \Vp^\Dual\Par{\TreeS_j},
            \TreeS_{j + 1}, \dots, \TreeS_{|\GenA|} \circ_i \GenB} \\
        & =
        \sum_{\substack{
            \GenB \in \GeneratingSet \\
            i \geq 0
        }}
        \sum_{j \in [2, |\GenA|]}
        \GenA \BPar{\TreeS_1, \dots,
            \TreeS_{j - 1}, \Vp^\Dual\Par{\TreeS_j \circ_i \GenB},
            \TreeS_{j + 1}, \dots, \TreeS_{|\GenA|}}
        \\
        & \quad -
        \GenA \BPar{\TreeS_1, \dots,
            \TreeS_{j - 1}, \Vp^\Dual\Par{\TreeS_j} \circ_i \GenB,
            \TreeS_{j + 1}, \dots, \TreeS_{|\GenA|}} \\
        & =
        \sum_{\substack{
            \GenB \in \GeneratingSet \\
            i \geq 0
        }}
        \sum_{j \in [2, |\GenA|]}
        \GenA \BPar{\TreeS_1, \dots,
            \TreeS_{j - 1}, \Vp^\Dual\Par{\TreeS_j \circ_i \GenB}
            - \Vp^\Dual\Par{\TreeS_j} \circ_i \GenB,
            \TreeS_{j + 1}, \dots, \TreeS_{|\GenA|}}.
    \end{split}\end{equation}
    By induction hypothesis, we get
    \begin{equation} \label{equ:prefix_tree_graphs_dual_2}
        \sum_{\substack{
            \GenB \in \GeneratingSet \\
            i \geq 0
        }}
        \Vp^\Dual\Par{\TreeS_j \circ_i \GenB} - \Vp^\Dual\Par{\TreeS_j} \circ_i \GenB
        = \Par{\Vp^\Dual \Up - \Up \Vp^\Dual}\Par{\TreeS_j}
        = f\Par{\TreeS_j}
        = \phi\Par{\TreeS_j}.
    \end{equation}
    We now obtain from~\eqref{equ:prefix_tree_graphs_dual_1},
    \eqref{equ:prefix_tree_graphs_dual_2}, and~\eqref{equ:non_first_leaves_reccurrence_3}
    that
    \begin{equation}\begin{split}
        f(\TreeT)
        & =
        \sum_{j \in [2, |\GenA|]}
        \GenA \BPar{\TreeS_1, \dots, \TreeS_{j - 1},
            \phi\Par{\TreeS_j}, \TreeS_{j + 1},
            \dots, \TreeS_{|\GenA|}} \\
        & =
        \sum_{j \in [2, |\GenA|]}
        \Par{\# \GeneratingSet} \NonFirstLeaves{\TreeS_j} \;
        \GenA \BPar{\TreeS_1, \dots, \TreeS_{j - 1}, \TreeS_j,
            \TreeS_{j + 1}, \dots, \TreeS_{|\GenA|}} \\
        & =
        \Par{\# \GeneratingSet} \NonFirstLeaves{\TreeT} \; \TreeT.
    \end{split}\end{equation}
    Therefore,
    \begin{math}
        f(\TreeT) = \Par{\# \GeneratingSet} \NonFirstLeaves{\TreeT} \; \TreeT,
    \end{math}
    establishing the statement of theorem.
\end{proof}
\medbreak

Figure~\ref{fig:examples_dual_tree_graphs} shows an example of a pair of $\phi$-diagonal
dual graded graphs.
\begin{figure}[ht]
    \centering
    \subfloat[][The graph $(\SyntaxTreesInternalNode(\GeneratingSet), \Up)$ up to degree $2$
    and with some trees of degree~$3$.]{
    \centering
    \begin{tikzpicture}[Centering,xscale=1.2,yscale=1.6]
        \node(Unit)at(2,0){$\Leaf$};
        \node(a00)at(-.5,-.5){$\CorollaTwo{\GenA}$};
        \node(c000)at(4.5,-.5){$\CorollaThree{\GenC}$};
        \node(aa000)at(-2,-1.5){
        \begin{tikzpicture}[Centering,xscale=0.13,yscale=0.21]
            \node(0)at(0.00,-3.33){};
            \node(2)at(2.00,-3.33){};
            \node(4)at(4.00,-1.67){};
            \node[NodeST](1)at(1.00,-1.67){$\GenA$};
            \node[NodeST](3)at(3.00,0.00){$\GenA$};
            \draw[Edge](0)--(1);
            \draw[Edge](1)--(3);
            \draw[Edge](2)--(1);
            \draw[Edge](4)--(3);
            \node(r)at(3.00,1.75){};
            \draw[Edge](r)--(3);
        \end{tikzpicture}};
        \node(a0a00)at(-1,-1.5){
        \begin{tikzpicture}[Centering,xscale=0.13,yscale=0.21]
            \node(0)at(0.00,-1.67){};
            \node(2)at(2.00,-3.33){};
            \node(4)at(4.00,-3.33){};
            \node[NodeST](1)at(1.00,0.00){$\GenA$};
            \node[NodeST](3)at(3.00,-1.67){$\GenA$};
            \draw[Edge](0)--(1);
            \draw[Edge](2)--(3);
            \draw[Edge](3)--(1);
            \draw[Edge](4)--(3);
            \node(r)at(1.00,1.75){};
            \draw[Edge](r)--(1);
        \end{tikzpicture}};
        \node(ac0000)at(0,-1.5){
        \begin{tikzpicture}[Centering,xscale=0.13,yscale=0.18]
            \node(0)at(0.00,-4.00){};
            \node(2)at(1.00,-4.00){};
            \node(3)at(2.00,-4.00){};
            \node(5)at(4.00,-2.00){};
            \node[NodeST](1)at(1.00,-2.00){$\GenC$};
            \node[NodeST](4)at(3.00,0.00){$\GenA$};
            \draw[Edge](0)--(1);
            \draw[Edge](1)--(4);
            \draw[Edge](2)--(1);
            \draw[Edge](3)--(1);
            \draw[Edge](5)--(4);
            \node(r)at(3.00,2){};
            \draw[Edge](r)--(4);
        \end{tikzpicture}};
        \node(a0c000)at(1,-1.5){
        \begin{tikzpicture}[Centering,xscale=0.13,yscale=0.18]
            \node(0)at(0.00,-2.00){};
            \node(2)at(2.00,-4.00){};
            \node(4)at(3.00,-4.00){};
            \node(5)at(4.00,-4.00){};
            \node[NodeST](1)at(1.00,0.00){$\GenA$};
            \node[NodeST](3)at(3.00,-2.00){$\GenC$};
            \draw[Edge](0)--(1);
            \draw[Edge](2)--(3);
            \draw[Edge](3)--(1);
            \draw[Edge](4)--(3);
            \draw[Edge](5)--(3);
            \node(r)at(1.00,2){};
            \draw[Edge](r)--(1);
        \end{tikzpicture}};
        \node(ca0000)at(2,-1.5){
        \begin{tikzpicture}[Centering,xscale=0.13,yscale=0.18]
            \node(0)at(0.00,-4.00){};
            \node(2)at(2.00,-4.00){};
            \node(4)at(3.00,-2.00){};
            \node(5)at(4.00,-2.00){};
            \node[NodeST](1)at(1.00,-2.00){$\GenA$};
            \node[NodeST](3)at(3.00,0.00){$\GenC$};
            \draw[Edge](0)--(1);
            \draw[Edge](1)--(3);
            \draw[Edge](2)--(1);
            \draw[Edge](4)--(3);
            \draw[Edge](5)--(3);
            \node(r)at(3.00,2){};
            \draw[Edge](r)--(3);
        \end{tikzpicture}};
        \node(c0a000)at(3,-1.5){
        \begin{tikzpicture}[Centering,xscale=0.13,yscale=0.18]
            \node(0)at(0.00,-2.00){};
            \node(2)at(1.00,-4.00){};
            \node(4)at(3.00,-4.00){};
            \node(5)at(4.00,-2.00){};
            \node[NodeST](1)at(2.00,0.00){$\GenC$};
            \node[NodeST](3)at(2.00,-2.00){$\GenA$};
            \draw[Edge](0)--(1);
            \draw[Edge](2)--(3);
            \draw[Edge](3)--(1);
            \draw[Edge](4)--(3);
            \draw[Edge](5)--(1);
            \node(r)at(2.00,2){};
            \draw[Edge](r)--(1);
        \end{tikzpicture}};
        \node(c00a00)at(4,-1.5){
        \begin{tikzpicture}[Centering,xscale=0.13,yscale=0.18]
            \node(0)at(0.00,-2.00){};
            \node(2)at(1.00,-2.00){};
            \node(3)at(2.00,-4.00){};
            \node(5)at(4.00,-4.00){};
            \node[NodeST](1)at(1.00,0.00){$\GenC$};
            \node[NodeST](4)at(3.00,-2.00){$\GenA$};
            \draw[Edge](0)--(1);
            \draw[Edge](2)--(1);
            \draw[Edge](3)--(4);
            \draw[Edge](4)--(1);
            \draw[Edge](5)--(4);
            \node(r)at(1.00,2){};
            \draw[Edge](r)--(1);
        \end{tikzpicture}};
        \node(cc00000)at(5,-1.5){
        \begin{tikzpicture}[Centering,xscale=0.13,yscale=0.16]
            \node(0)at(0.00,-4.67){};
            \node(2)at(1.00,-4.67){};
            \node(3)at(2.00,-4.67){};
            \node(5)at(3.00,-2.33){};
            \node(6)at(4.00,-2.33){};
            \node[NodeST](1)at(1.00,-2.33){$\GenC$};
            \node[NodeST](4)at(3.00,0.00){$\GenC$};
            \draw[Edge](0)--(1);
            \draw[Edge](1)--(4);
            \draw[Edge](2)--(1);
            \draw[Edge](3)--(1);
            \draw[Edge](5)--(4);
            \draw[Edge](6)--(4);
            \node(r)at(3.00,2.25){};
            \draw[Edge](r)--(4);
        \end{tikzpicture}};
        \node(c0c0000)at(6,-1.5){
        \begin{tikzpicture}[Centering,xscale=0.13,yscale=0.16]
            \node(0)at(0.00,-2.33){};
            \node(2)at(1.00,-4.67){};
            \node(4)at(2.00,-4.67){};
            \node(5)at(3.00,-4.67){};
            \node(6)at(4.00,-2.33){};
            \node[NodeST](1)at(2.00,0.00){$\GenC$};
            \node[NodeST](3)at(2.00,-2.33){$\GenC$};
            \draw[Edge](0)--(1);
            \draw[Edge](2)--(3);
            \draw[Edge](3)--(1);
            \draw[Edge](4)--(3);
            \draw[Edge](5)--(3);
            \draw[Edge](6)--(1);
            \node(r)at(2.00,2.25){};
            \draw[Edge](r)--(1);
        \end{tikzpicture}};
        \node(c00c000)at(7,-1.5){
        \begin{tikzpicture}[Centering,xscale=0.13,yscale=0.16]
            \node(0)at(0.00,-2.33){};
            \node(2)at(1.00,-2.33){};
            \node(3)at(2.00,-4.67){};
            \node(5)at(3.00,-4.67){};
            \node(6)at(4.00,-4.67){};
            \node[NodeST](1)at(1.00,0.00){$\GenC$};
            \node[NodeST](4)at(3.00,-2.33){$\GenC$};
            \draw[Edge](0)--(1);
            \draw[Edge](2)--(1);
            \draw[Edge](3)--(4);
            \draw[Edge](4)--(1);
            \draw[Edge](5)--(4);
            \draw[Edge](6)--(4);
            \node(r)at(1.00,2.25){};
            \draw[Edge](r)--(1);
        \end{tikzpicture}};
        \node(c0a00a00)at(3.5,-2.75){
        \begin{tikzpicture}[Centering,xscale=0.16,yscale=0.14]
            \node(0)at(0.00,-2.67){};
            \node(2)at(1.00,-5.33){};
            \node(4)at(3.00,-5.33){};
            \node(5)at(4.00,-5.33){};
            \node(7)at(6.00,-5.33){};
            \node[NodeST](1)at(2.00,0.00){$\GenC$};
            \node[NodeST](3)at(2.00,-2.67){$\GenA$};
            \node[NodeST](6)at(5.00,-2.67){$\GenA$};
            \draw[Edge](0)--(1);
            \draw[Edge](2)--(3);
            \draw[Edge](3)--(1);
            \draw[Edge](4)--(3);
            \draw[Edge](5)--(6);
            \draw[Edge](6)--(1);
            \draw[Edge](7)--(6);
            \node(r)at(2.00,2.5){};
            \draw[Edge](r)--(1);
        \end{tikzpicture}};
        \node(cc0000c000)at(6,-2.75){
        \begin{tikzpicture}[Centering,xscale=0.16,yscale=0.12]
            \node(0)at(0.00,-6.67){};
            \node(2)at(1.00,-6.67){};
            \node(3)at(2.00,-6.67){};
            \node(5)at(3.00,-3.33){};
            \node(6)at(4.00,-6.67){};
            \node(8)at(5.00,-6.67){};
            \node(9)at(6.00,-6.67){};
            \node[NodeST](1)at(1.00,-3.33){$\GenC$};
            \node[NodeST](4)at(3.00,0.00){$\GenC$};
            \node[NodeST](7)at(5.00,-3.33){$\GenC$};
            \draw[Edge](0)--(1);
            \draw[Edge](1)--(4);
            \draw[Edge](2)--(1);
            \draw[Edge](3)--(1);
            \draw[Edge](5)--(4);
            \draw[Edge](6)--(7);
            \draw[Edge](7)--(4);
            \draw[Edge](8)--(7);
            \draw[Edge](9)--(7);
            \node(r)at(3.00,3){};
            \draw[Edge](r)--(4);
        \end{tikzpicture}};
        \draw[EdgeGraph](Unit)--(a00);
        \draw[EdgeGraph](Unit)--(c000);
        \draw[EdgeGraph](a00)--(aa000);
        \draw[EdgeGraph](a00)--(a0a00);
        \draw[EdgeGraph](a00)--(ac0000);
        \draw[EdgeGraph](a00)--(a0c000);
        \draw[EdgeGraph](c000)--(ca0000);
        \draw[EdgeGraph](c000)--(c0a000);
        \draw[EdgeGraph](c000)--(c00a00);
        \draw[EdgeGraph](c000)--(cc00000);
        \draw[EdgeGraph](c000)--(c0c0000);
        \draw[EdgeGraph](c000)--(c00c000);
        \draw[EdgeGraph](c0a000)--(c0a00a00);
        \draw[EdgeGraph](c00a00)--(c0a00a00);
        \draw[EdgeGraph](cc00000)--(cc0000c000);
        \draw[EdgeGraph](c00c000)--(cc0000c000);
    \end{tikzpicture}
    \label{subfig:tree_graph_a2c3_U}}
    \medbreak
    \subfloat[][The graph $(\SyntaxTreesInternalNode(\GeneratingSet), \Vp)$ up to degree $2$
    and with some trees of degree $3$.]{
    \centering
    \begin{tikzpicture}[Centering,xscale=1.2,yscale=1.6]
        \node(Unit)at(2,0){$\Leaf$};
        \node(a00)at(-.5,-.5){$\CorollaTwo{\GenA}$};
        \node(c000)at(4.5,-.5){$\CorollaThree{\GenC}$};
        \node(aa000)at(-2,-1.5){
        \begin{tikzpicture}[Centering,xscale=0.13,yscale=0.21]
            \node(0)at(0.00,-3.33){};
            \node(2)at(2.00,-3.33){};
            \node(4)at(4.00,-1.67){};
            \node[NodeST](1)at(1.00,-1.67){$\GenA$};
            \node[NodeST](3)at(3.00,0.00){$\GenA$};
            \draw[Edge](0)--(1);
            \draw[Edge](1)--(3);
            \draw[Edge](2)--(1);
            \draw[Edge](4)--(3);
            \node(r)at(3.00,1.75){};
            \draw[Edge](r)--(3);
        \end{tikzpicture}};
        \node(a0a00)at(-1,-1.5){
        \begin{tikzpicture}[Centering,xscale=0.13,yscale=0.21]
            \node(0)at(0.00,-1.67){};
            \node(2)at(2.00,-3.33){};
            \node(4)at(4.00,-3.33){};
            \node[NodeST](1)at(1.00,0.00){$\GenA$};
            \node[NodeST](3)at(3.00,-1.67){$\GenA$};
            \draw[Edge](0)--(1);
            \draw[Edge](2)--(3);
            \draw[Edge](3)--(1);
            \draw[Edge](4)--(3);
            \node(r)at(1.00,1.75){};
            \draw[Edge](r)--(1);
        \end{tikzpicture}};
        \node(ac0000)at(0,-1.5){
        \begin{tikzpicture}[Centering,xscale=0.13,yscale=0.18]
            \node(0)at(0.00,-4.00){};
            \node(2)at(1.00,-4.00){};
            \node(3)at(2.00,-4.00){};
            \node(5)at(4.00,-2.00){};
            \node[NodeST](1)at(1.00,-2.00){$\GenC$};
            \node[NodeST](4)at(3.00,0.00){$\GenA$};
            \draw[Edge](0)--(1);
            \draw[Edge](1)--(4);
            \draw[Edge](2)--(1);
            \draw[Edge](3)--(1);
            \draw[Edge](5)--(4);
            \node(r)at(3.00,2){};
            \draw[Edge](r)--(4);
        \end{tikzpicture}};
        \node(a0c000)at(1,-1.5){
        \begin{tikzpicture}[Centering,xscale=0.13,yscale=0.18]
            \node(0)at(0.00,-2.00){};
            \node(2)at(2.00,-4.00){};
            \node(4)at(3.00,-4.00){};
            \node(5)at(4.00,-4.00){};
            \node[NodeST](1)at(1.00,0.00){$\GenA$};
            \node[NodeST](3)at(3.00,-2.00){$\GenC$};
            \draw[Edge](0)--(1);
            \draw[Edge](2)--(3);
            \draw[Edge](3)--(1);
            \draw[Edge](4)--(3);
            \draw[Edge](5)--(3);
            \node(r)at(1.00,2){};
            \draw[Edge](r)--(1);
        \end{tikzpicture}};
        \node(ca0000)at(2,-1.5){
        \begin{tikzpicture}[Centering,xscale=0.13,yscale=0.18]
            \node(0)at(0.00,-4.00){};
            \node(2)at(2.00,-4.00){};
            \node(4)at(3.00,-2.00){};
            \node(5)at(4.00,-2.00){};
            \node[NodeST](1)at(1.00,-2.00){$\GenA$};
            \node[NodeST](3)at(3.00,0.00){$\GenC$};
            \draw[Edge](0)--(1);
            \draw[Edge](1)--(3);
            \draw[Edge](2)--(1);
            \draw[Edge](4)--(3);
            \draw[Edge](5)--(3);
            \node(r)at(3.00,2){};
            \draw[Edge](r)--(3);
        \end{tikzpicture}};
        \node(c0a000)at(3,-1.5){
        \begin{tikzpicture}[Centering,xscale=0.13,yscale=0.18]
            \node(0)at(0.00,-2.00){};
            \node(2)at(1.00,-4.00){};
            \node(4)at(3.00,-4.00){};
            \node(5)at(4.00,-2.00){};
            \node[NodeST](1)at(2.00,0.00){$\GenC$};
            \node[NodeST](3)at(2.00,-2.00){$\GenA$};
            \draw[Edge](0)--(1);
            \draw[Edge](2)--(3);
            \draw[Edge](3)--(1);
            \draw[Edge](4)--(3);
            \draw[Edge](5)--(1);
            \node(r)at(2.00,2){};
            \draw[Edge](r)--(1);
        \end{tikzpicture}};
        \node(c00a00)at(4,-1.5){
        \begin{tikzpicture}[Centering,xscale=0.13,yscale=0.18]
            \node(0)at(0.00,-2.00){};
            \node(2)at(1.00,-2.00){};
            \node(3)at(2.00,-4.00){};
            \node(5)at(4.00,-4.00){};
            \node[NodeST](1)at(1.00,0.00){$\GenC$};
            \node[NodeST](4)at(3.00,-2.00){$\GenA$};
            \draw[Edge](0)--(1);
            \draw[Edge](2)--(1);
            \draw[Edge](3)--(4);
            \draw[Edge](4)--(1);
            \draw[Edge](5)--(4);
            \node(r)at(1.00,2){};
            \draw[Edge](r)--(1);
        \end{tikzpicture}};
        \node(cc00000)at(5,-1.5){
        \begin{tikzpicture}[Centering,xscale=0.13,yscale=0.16]
            \node(0)at(0.00,-4.67){};
            \node(2)at(1.00,-4.67){};
            \node(3)at(2.00,-4.67){};
            \node(5)at(3.00,-2.33){};
            \node(6)at(4.00,-2.33){};
            \node[NodeST](1)at(1.00,-2.33){$\GenC$};
            \node[NodeST](4)at(3.00,0.00){$\GenC$};
            \draw[Edge](0)--(1);
            \draw[Edge](1)--(4);
            \draw[Edge](2)--(1);
            \draw[Edge](3)--(1);
            \draw[Edge](5)--(4);
            \draw[Edge](6)--(4);
            \node(r)at(3.00,2.25){};
            \draw[Edge](r)--(4);
        \end{tikzpicture}};
        \node(c0c0000)at(6,-1.5){
        \begin{tikzpicture}[Centering,xscale=0.13,yscale=0.16]
            \node(0)at(0.00,-2.33){};
            \node(2)at(1.00,-4.67){};
            \node(4)at(2.00,-4.67){};
            \node(5)at(3.00,-4.67){};
            \node(6)at(4.00,-2.33){};
            \node[NodeST](1)at(2.00,0.00){$\GenC$};
            \node[NodeST](3)at(2.00,-2.33){$\GenC$};
            \draw[Edge](0)--(1);
            \draw[Edge](2)--(3);
            \draw[Edge](3)--(1);
            \draw[Edge](4)--(3);
            \draw[Edge](5)--(3);
            \draw[Edge](6)--(1);
            \node(r)at(2.00,2.25){};
            \draw[Edge](r)--(1);
        \end{tikzpicture}};
        \node(c00c000)at(7,-1.5){
        \begin{tikzpicture}[Centering,xscale=0.13,yscale=0.16]
            \node(0)at(0.00,-2.33){};
            \node(2)at(1.00,-2.33){};
            \node(3)at(2.00,-4.67){};
            \node(5)at(3.00,-4.67){};
            \node(6)at(4.00,-4.67){};
            \node[NodeST](1)at(1.00,0.00){$\GenC$};
            \node[NodeST](4)at(3.00,-2.33){$\GenC$};
            \draw[Edge](0)--(1);
            \draw[Edge](2)--(1);
            \draw[Edge](3)--(4);
            \draw[Edge](4)--(1);
            \draw[Edge](5)--(4);
            \draw[Edge](6)--(4);
            \node(r)at(1.00,2.25){};
            \draw[Edge](r)--(1);
        \end{tikzpicture}};
        \node(c0a00a00)at(3.5,-2.75){
        \begin{tikzpicture}[Centering,xscale=0.16,yscale=0.14]
            \node(0)at(0.00,-2.67){};
            \node(2)at(1.00,-5.33){};
            \node(4)at(3.00,-5.33){};
            \node(5)at(4.00,-5.33){};
            \node(7)at(6.00,-5.33){};
            \node[NodeST](1)at(2.00,0.00){$\GenC$};
            \node[NodeST](3)at(2.00,-2.67){$\GenA$};
            \node[NodeST](6)at(5.00,-2.67){$\GenA$};
            \draw[Edge](0)--(1);
            \draw[Edge](2)--(3);
            \draw[Edge](3)--(1);
            \draw[Edge](4)--(3);
            \draw[Edge](5)--(6);
            \draw[Edge](6)--(1);
            \draw[Edge](7)--(6);
            \node(r)at(2.00,2.5){};
            \draw[Edge](r)--(1);
        \end{tikzpicture}};
        \node(cc0000c000)at(6,-2.75){
        \begin{tikzpicture}[Centering,xscale=0.16,yscale=0.12]
            \node(0)at(0.00,-6.67){};
            \node(2)at(1.00,-6.67){};
            \node(3)at(2.00,-6.67){};
            \node(5)at(3.00,-3.33){};
            \node(6)at(4.00,-6.67){};
            \node(8)at(5.00,-6.67){};
            \node(9)at(6.00,-6.67){};
            \node[NodeST](1)at(1.00,-3.33){$\GenC$};
            \node[NodeST](4)at(3.00,0.00){$\GenC$};
            \node[NodeST](7)at(5.00,-3.33){$\GenC$};
            \draw[Edge](0)--(1);
            \draw[Edge](1)--(4);
            \draw[Edge](2)--(1);
            \draw[Edge](3)--(1);
            \draw[Edge](5)--(4);
            \draw[Edge](6)--(7);
            \draw[Edge](7)--(4);
            \draw[Edge](8)--(7);
            \draw[Edge](9)--(7);
            \node(r)at(3.00,3){};
            \draw[Edge](r)--(4);
        \end{tikzpicture}};
        \draw[EdgeGraph](Unit)--(a00);
        \draw[EdgeGraph](Unit)--(c000);
        \draw[EdgeGraph](a00)--(aa000);
        \draw[EdgeGraph](a00)--(a0a00);
        \draw[EdgeGraph](a00)--(a0c000);
        \draw[EdgeGraph](a00)--(ca0000);
        \draw[EdgeGraph](c000)--(ac0000);
        \draw[EdgeGraph](c000)--(cc00000);
        \draw[EdgeGraph](c000)--(c0a000);
        \draw[EdgeGraph](c000)--(c00a00);
        \draw[EdgeGraph](c000)--(cc00000);
        \draw[EdgeGraph](c000)--(c0c0000);
        \draw[EdgeGraph](c000)--(c00c000);
        \draw[EdgeGraph](c0a000)--(c0a00a00);
        \draw[EdgeGraph](c00a00)--(c0a00a00);
        \draw[EdgeGraph](cc00000)--(cc0000c000);
    \end{tikzpicture}
    \label{subfig:tree_graph_a2c3_D}}
    \caption{The pair $(\SyntaxTreesInternalNode(\GeneratingSet), \Up, \Vp)$ of
    $\phi$-diagonal dual graded graphs where $\GeneratingSet := \Bra{\GenA, \GenC}$ with
    $|\GenA| = 2$ and $|\GenC| = 3$.}
    \label{fig:examples_dual_tree_graphs}
\end{figure}
\medbreak

%%%%%%%%%%%%%%%%%%%%%%%%%%%%%%%%%%%%%%%%%%%%%%%%%%%%%%%%%%%%%%%%%%%%%%%%%%%%%%%%%%%%%%%%%%%%
\subsubsection{Bracket tree}
As already noticed before, when
\begin{math}
    \GeneratingSet = \GeneratingSet(2) = \Bra{\GenA},
\end{math}
the graph $\Par{\SyntaxTreesInternalNode(\GeneratingSet), \Vp}$ is a tree. Moreover,
$\Par{\SyntaxTreesInternalNode(\GeneratingSet), \Up, \Vp}$ is a pair of dual graded graphs
isomorphic to the pair consisting in the finite order ideals of the infinite binary tree and
the \Def{Bracket tree}, known from~\cite{Fom94} (see also~\cite{HNT05}). One can see
Theorem~\ref{thm:prefix_tree_graphs_dual} as a generalization of this prototypical instance
for the present case of $\GeneratingSet$-trees and $\phi$-diagonal duality.
\medbreak

%%%%%%%%%%%%%%%%%%%%%%%%%%%%%%%%%%%%%%%%%%%%%%%%%%%%%%%%%%%%%%%%%%%%%%%%%%%%%%%%%%%%%%%%%%%%
%%%%%%%%%%%%%%%%%%%%%%%%%%%%%%%%%%%%%%%%%%%%%%%%%%%%%%%%%%%%%%%%%%%%%%%%%%%%%%%%%%%%%%%%%%%%
%%%%%%%%%%%%%%%%%%%%%%%%%%%%%%%%%%%%%%%%%%%%%%%%%%%%%%%%%%%%%%%%%%%%%%%%%%%%%%%%%%%%%%%%%%%%
\section{Posets of syntax trees} \label{sec:tree_prefix_posets}
We present here a combinatorial study of the posets of the $\GeneratingSet$-prefix graphs.
In particular we look at their lattice properties, the structure of their intervals,
enumerate the trees in a given interval, and enumerate all intervals with respect to the
degrees of theirs bounds.
\medbreak

%%%%%%%%%%%%%%%%%%%%%%%%%%%%%%%%%%%%%%%%%%%%%%%%%%%%%%%%%%%%%%%%%%%%%%%%%%%%%%%%%%%%%%%%%%%%
%%%%%%%%%%%%%%%%%%%%%%%%%%%%%%%%%%%%%%%%%%%%%%%%%%%%%%%%%%%%%%%%%%%%%%%%%%%%%%%%%%%%%%%%%%%%
\subsection{Posets and their intervals}
We begin by describing the order relation and covering relation of the posets of
$\GeneratingSet$-prefix graphs. We prove that any interval of these posets are distributive
lattices
\medbreak

%%%%%%%%%%%%%%%%%%%%%%%%%%%%%%%%%%%%%%%%%%%%%%%%%%%%%%%%%%%%%%%%%%%%%%%%%%%%%%%%%%%%%%%%%%%%
\subsubsection{Prefix posets}
Let $\GeneratingSet$ be a finite alphabet. The \Def{$\GeneratingSet$-prefix poset} is the
poset $\Par{\SyntaxTreesInternalNode(\GeneratingSet), \OrderPrefixes}$ of
$\Par{\SyntaxTreesInternalNode(\GeneratingSet), \Up}$. Besides, for any
$\GeneratingSet$-trees $\TreeS$ and $\TreeT$, $\TreeS$ is a \Def{prefix} of $\TreeT$ if
there exist $\GeneratingSet$-trees $\TreeR_1$, \dots, $\TreeR_{|\TreeS|}$ such that
\begin{math}
    \TreeT = \TreeS \circ \Han{\TreeR_1, \dots, \TreeR_{|\TreeS|}}.
\end{math}
\medbreak

\begin{Lemma} \label{lem:prefix_recursive}
    Let $\GeneratingSet$ be an alphabet, and $\TreeS$ and $\TreeT$ be two
    $\GeneratingSet$-trees. Then, $\TreeS$ is a prefix of $\TreeT$ if and only if $\TreeS =
    \Leaf$, or the roots of $\TreeS$ and $\TreeT$ are both decorated by the same letter
    $\GenA \in \GeneratingSet$, and for all $i \in [|\GenA|]$, $\TreeS(i)$ is a prefix of
    $\TreeT(i)$.
\end{Lemma}
\begin{proof}
    This follows directly from the definition of the notion of prefix just introduced.
\end{proof}
\medbreak

\begin{Proposition} \label{prop:prefix_tree_poset}
    For any finite alphabet $\GeneratingSet$, the order relation $\OrderPrefixes$ of the
    $\GeneratingSet$-prefix poset satisfies $\TreeS \OrderPrefixes \TreeT$ if and only if
    the $\GeneratingSet$-tree $\TreeS$ is a prefix of the $\GeneratingSet$-tree $\TreeT$.
    Moreover, the covering relation $\Covered_\Up$ of the $\GeneratingSet$-prefix poset
    satisfies $\TreeS \Covered_\Up \TreeT$ for any $\GeneratingSet$-trees $\TreeS$ and
    $\TreeT$ if and only if there is $u \in \InternalNodesMax(\TreeT)$ such that $\TreeS =
    \DelNode{\TreeT}{u}$.
\end{Proposition}
\begin{proof}
    Assume that $\TreeS \OrderPrefixes \TreeT$.  By definition of the
    $\GeneratingSet$-prefix poset, there exist an integer $k \geq 0$, letters $\GenA_1$,
    \dots, $\GenA_k$ of $\GeneratingSet$, and positive integers $i_1$, \dots, $i_k$ such
    that
    \begin{equation} \label{equ:prefix_tree_poset_1}
        \TreeT =
        \Par{\dots \Par{\Par{\TreeS \circ_{i_1} \GenA_1} \circ_{i_2}
        \GenA_2} \dots} \circ_{i_k} \GenA_k.
    \end{equation}
    It follows straightforwardly by induction on $k$ that there exist $\GeneratingSet$-trees
    $\TreeR_1$, \dots, $\TreeR_{|\TreeS|}$ such that
    \begin{math}
        \TreeT = \TreeS \circ \Han{\TreeR_1, \dots, \TreeR_{|\TreeS|}}.
    \end{math}
    Therefore, $\TreeS$ is a prefix of $\TreeT$.
    \smallbreak

    Conversely, assume that $\TreeS$ is a prefix of $\TreeT$. Hence, there exist
    $\GeneratingSet$-trees $\TreeR_1$, \dots, $\TreeR_{|\TreeS|}$ such that
    \begin{math}
        \TreeT = \TreeS \circ \Han{\TreeR_1, \dots, \TreeR_{|\TreeS|}}.
    \end{math}
    By~\eqref{equ:full_composition_maps},
    \begin{equation}
        \TreeT =
        \Par{\dots \Par{\Par{\TreeS \circ_{|\TreeS|} \TreeR_{|\TreeS|}}
        \circ_{|\TreeS| - 1} \TreeR_{|\TreeS| - 1}} \dots} \circ_1 \TreeR_1,
    \end{equation}
    and, by expressing now each tree $\TreeR_i$, $i \in [|\TreeS|]$, by means of partial
    compositions involving letters of $\GeneratingSet$, and  by arranging this expression so
    that it becomes bracketed from the left to the right by using
    Relation~\eqref{equ:operad_axiom_1}, one obtains an expression of the same form
    as~\eqref{equ:prefix_tree_poset_1} for $\TreeT$. Therefore, $\TreeS \OrderPrefixes
    \TreeT$.
    \smallbreak

    The second part of the statement is a direct consequence of the fact that the adjoint
    map $\Up^\Dual$ of the map $\Up$ of the $\GeneratingSet$-prefix graph
    satisfies~\eqref{equ:adjoint_up_trees}.
\end{proof}
\medbreak

%%%%%%%%%%%%%%%%%%%%%%%%%%%%%%%%%%%%%%%%%%%%%%%%%%%%%%%%%%%%%%%%%%%%%%%%%%%%%%%%%%%%%%%%%%%%
\subsubsection{Distributive lattices}
Let $\GeneratingSet$ be a finite alphabet. Let $\Meet$ be the binary operation on
$\SyntaxTreesInternalNode(\GeneratingSet)$, called \Def{intersection}, defined recursively
by
\begin{subequations}
\begin{equation}
    \TreeT \Meet \Leaf := \Leaf =: \Leaf \Meet \TreeT,
\end{equation}
\begin{equation}
    \GenA \BPar{\TreeT_1, \dots, \TreeT_{|\GenA|}}
    \Meet
    \GenB \BPar{\TreeS_1, \dots, \TreeS_{|\GenB|}}
    := \Leaf,
\end{equation}
\begin{equation}
    \GenA \BPar{\TreeT_1, \dots, \TreeT_{|\GenA|}}
    \Meet
    \GenA \BPar{\TreeT'_1, \dots, \TreeT'_{|\GenA|}}
    :=
    \GenA \BPar{\TreeT_1 \Meet \TreeT'_1, \dots,
    \TreeT_{|\GenA|} \Meet \TreeT'_{|\GenA|}},
\end{equation}
\end{subequations}
for any $\GenA, \GenB \in \GeneratingSet$ such that $\GenA \ne \GenB$, and any
$\GeneratingSet$-trees $\TreeT$, $\TreeT_1$, \dots, $\TreeT_{|\GenA|}$, $\TreeT'_1$, \dots,
$\TreeT'_{|\GenA|}$, and $\TreeS_1$, \dots, $\TreeS_{|\GenB|}$. From an intuitive point of
view, $\TreeT \Meet \TreeT'$ is the tree obtained by considering the largest common part
between the $\GeneratingSet$-trees $\TreeT$ and $\TreeT'$ starting from their roots. For
instance,
\begin{equation}
    \begin{tikzpicture}[Centering,xscale=0.22,yscale=0.19]
        \node(0)at(0.00,-4.50){};
        \node(2)at(2.00,-4.50){};
        \node(4)at(3.00,-2.25){};
        \node(6)at(4.00,-6.75){};
        \node(8)at(6.00,-4.50){};
        \node[NodeST](1)at(1.00,-2.25){$\GenA$};
        \node[NodeST](3)at(3.00,0.00){$\GenC$};
        \node[NodeST](5)at(4.00,-4.50){$\GenE$};
        \node[NodeST](7)at(5.00,-2.25){$\GenA$};
        \draw[Edge](0)--(1);
        \draw[Edge](1)--(3);
        \draw[Edge](2)--(1);
        \draw[Edge](4)--(3);
        \draw[Edge](5)--(7);
        \draw[Edge](6)--(5);
        \draw[Edge](7)--(3);
        \draw[Edge](8)--(7);
        \node(r)at(3.00,2.25){};
        \draw[Edge](r)--(3);
    \end{tikzpicture}
    \enspace \Meet \enspace
    \begin{tikzpicture}[Centering,xscale=0.26,yscale=0.15]
        \node(1)at(0.00,-6.00){};
        \node(3)at(1.50,-6.00){};
        \node(5)at(2.50,-6.00){};
        \node(6)at(3.50,-6.00){};
        \node(8)at(4.50,-6.00){};
        \node[NodeST](0)at(0.00,-3.00){$\GenE$};
        \node[NodeST](2)at(2.00,0.00){$\GenC$};
        \node[NodeST](4)at(2.00,-3.00){$\GenA$};
        \node[NodeST](7)at(4.00,-3.00){$\GenA$};
        \draw[Edge](0)--(2);
        \draw[Edge](1)--(0);
        \draw[Edge](3)--(4);
        \draw[Edge](4)--(2);
        \draw[Edge](5)--(4);
        \draw[Edge](6)--(7);
        \draw[Edge](7)--(2);
        \draw[Edge](8)--(7);
        \node(r)at(2.00,2.5){};
        \draw[Edge](r)--(2);
    \end{tikzpicture}
    \enspace = \enspace
    \begin{tikzpicture}[Centering,xscale=0.18,yscale=0.21]
        \node(0)at(0.00,-2.00){};
        \node(2)at(1.00,-2.00){};
        \node(3)at(2.00,-4.00){};
        \node(5)at(4.00,-4.00){};
        \node[NodeST](1)at(1.00,0.00){$\GenC$};
        \node[NodeST](4)at(3.00,-2.00){$\GenA$};
        \draw[Edge](0)--(1);
        \draw[Edge](2)--(1);
        \draw[Edge](3)--(4);
        \draw[Edge](4)--(1);
        \draw[Edge](5)--(4);
        \node(r)at(1.00,2){};
        \draw[Edge](r)--(1);
    \end{tikzpicture}.
\end{equation}
\medbreak

\begin{Lemma} \label{lem:prefix_tree_poset_meet}
    For any finite alphabet $\GeneratingSet$ and any $\GeneratingSet$-trees $\TreeT$ and
    $\TreeT'$, $\TreeT \Meet \TreeT'$ is greatest lower bound of $\Bra{\TreeT, \TreeT'}$ in
    the $\GeneratingSet$-prefix poset.
\end{Lemma}
\begin{proof}
    We use here Proposition~\ref{prop:prefix_tree_poset} and its description of the order
    relation $\OrderPrefixes$ of the $\GeneratingSet$-prefix poset in terms of prefixes of
    $\GeneratingSet$-trees.  Let us denote by $L\Par{\TreeT, \TreeT'}$ the set all lower
    bounds of $\Bra{\TreeT, \TreeT'}$. By structural induction on $\TreeT$ and $\TreeT'$, we
    show that $\max_{\OrderPrefixes} L\Par{\TreeT, \TreeT'}$ exists and that
    $\max_{\OrderPrefixes} L\Par{\TreeT, \TreeT'} = \TreeT \Meet \TreeT'$.  First,
    immediately from the definition of $\Meet$, we have
    \begin{math}
        \Leaf \Meet \TreeT = \TreeT \Meet \Leaf = \Leaf
    \end{math}
    and
    \begin{math}
        L\Par{\Leaf, \TreeT} = L\Par{\TreeT, \Leaf} = \Bra{\Leaf}
    \end{math}
    so that the statement of the lemma holds in this case.  Assume now that $\TreeT$ and
    $\TreeT'$ are both different from the leaf so that $\TreeT = \GenA \BPar{\TreeT_1,
    \dots, \TreeT_{|\GenA|}}$ and
    \begin{math}
        \TreeT' =
        \GenA' \BPar{\TreeT'_1, \dots, \TreeT'_{\left|\GenA'\right|}}
    \end{math}
    where $\GenA, \GenA' \in \GeneratingSet$ and $\TreeT_1$, \dots $\TreeT_{|\GenA|}$,
    $\TreeT'_1$, \dots, $\TreeT'_{\left|\GenA'\right|}$ are $\GeneratingSet$-trees. If
    $\GenA \ne \GenA'$, we have $\TreeT \Meet \TreeT' = \Leaf$ and
    \begin{math}
        L\Par{\TreeT, \TreeT'} = \Bra{\Leaf}
    \end{math}
    since the leaf is the only $\GeneratingSet$-tree which is a common prefix of both
    $\TreeT$ and $\TreeT'$.  Hence, the statement of the lemma holds in this case.  For the
    last case to consider, one has $\GenA = \GenA'$, and it follows by induction hypothesis
    that
    \begin{equation}
        \TreeT \Meet \TreeT'
        =
        \GenA \BPar{\TreeT_1 \Meet \TreeT'_1, \dots,
        \TreeT_{|\GenA|} \Meet \TreeT'_{|\GenA|}}
        =
        \GenA \BPar{\TreeS_1, \dots, \TreeS_{|\GenA|}},
    \end{equation}
    where for any $i \in [|\GenA|]$,
    \begin{math}
        \TreeS_i := \max_{\OrderPrefixes} L\Par{\TreeT_i, \TreeT'_i}.
    \end{math}
    Now, since for any $i \in [|\GenA|]$, $\TreeS_i$ is a prefix of both $\TreeT_i$ and
    $\TreeT'_i$, and since all trees different from the leaf of $L\Par{\TreeT, \TreeT'}$
    have a root decorated by $\GenA$, by Lemma~\ref{lem:prefix_recursive}, $\TreeT \Meet
    \TreeT'$ is a prefix of both $\TreeT$ and $\TreeT'$ so that $\TreeT \Meet \TreeT' \in
    L\Par{\TreeT, \TreeT'}$.  To show finally that $\TreeT \Meet \TreeT'$ is the greatest
    element of $L\Par{\TreeT, \TreeT'}$, assume that $\TreeR$ is a $\GeneratingSet$-tree of
    $L\Par{\TreeT, \TreeT'}$. First, the root of $\TreeR$ is decorated by $\GenA$.  Second,
    since for any $i \in [|\GenA|]$, $\TreeR(i)$ is a prefix of both $\TreeT_i$ and
    $\TreeT'_i$, and since $\TreeS_i$ is the greatest $\GeneratingSet$-tree which is a
    common prefix of $\TreeT_i$ and $\TreeT'_i$, $\TreeR(i)$ is a prefix of $\TreeS_i$.
    Therefore, by Lemma~\ref{lem:prefix_recursive}, this implies that $\TreeR$ is a prefix
    of $\TreeT \Meet \TreeT'$ and establishes the statement of the lemma.
\end{proof}
\medbreak

In the same way, let $\JJoin$ be the partial binary operation on
$\SyntaxTreesInternalNode(\GeneratingSet)$, called \Def{union}, defined recursively by
\begin{subequations}
    \begin{equation}
        \TreeT \JJoin \Leaf := \TreeT =: \Leaf \JJoin \TreeT,
    \end{equation}
    \begin{equation}
        \GenA \BPar{\TreeT_1, \dots, \TreeT_{|\GenA|}}
        \JJoin
        \GenA \BPar{\TreeT'_1, \dots, \TreeT'_{|\GenA|}}
        :=
        \GenA \BPar{\TreeT_1 \JJoin \TreeT'_1, \dots,
        \TreeT_{|\GenA|} \JJoin \TreeT'_{|\GenA|}},
    \end{equation}
\end{subequations}
and where
\begin{equation}
    \GenA \BPar{\TreeT_1, \dots, \TreeT_{|\GenA|}}
    \JJoin
    \GenB \BPar{\TreeS_1, \dots, \TreeS_{|\GenB|}}
\end{equation}
is not defined, for any $\GenA, \GenB \in \GeneratingSet$ such that $\GenA \ne \GenB$, and
any $\GeneratingSet$-trees $\TreeT$, $\TreeT_1$, \dots, $\TreeT_{|\GenA|}$, $\TreeT'_1$,
\dots, $\TreeT'_{|\GenA|}$, and $\TreeS_1$, \dots, $\TreeS_{|\GenB|}$. From an intuitive
point of view, $\TreeT \JJoin \TreeT'$ is the tree obtained by superimposing $\TreeT$ and
$\TreeT'$. For instance,
\begin{equation}
    \begin{tikzpicture}[Centering,xscale=0.18,yscale=0.24]
        \node(0)at(0.00,-1.67){};
        \node(2)at(2.00,-3.33){};
        \node(4)at(4.00,-3.33){};
        \node[NodeST](1)at(1.00,0.00){$\GenA$};
        \node[NodeST](3)at(3.00,-1.67){$\GenA$};
        \draw[Edge](0)--(1);
        \draw[Edge](2)--(3);
        \draw[Edge](3)--(1);
        \draw[Edge](4)--(3);
        \node(r)at(1.00,1.75){};
        \draw[Edge](r)--(1);
    \end{tikzpicture}
    \enspace \JJoin \enspace
    \begin{tikzpicture}[Centering,xscale=0.17,yscale=0.2]
        \node(0)at(0.00,-4.00){};
        \node(2)at(1.00,-4.00){};
        \node(3)at(2.00,-6.00){};
        \node(5)at(4.00,-6.00){};
        \node(7)at(6.00,-2.00){};
        \node[NodeST](1)at(1.00,-2.00){$\GenC$};
        \node[NodeST](4)at(3.00,-4.00){$\GenA$};
        \node[NodeST](6)at(5.00,0.00){$\GenA$};
        \draw[Edge](0)--(1);
        \draw[Edge](1)--(6);
        \draw[Edge](2)--(1);
        \draw[Edge](3)--(4);
        \draw[Edge](4)--(1);
        \draw[Edge](5)--(4);
        \draw[Edge](7)--(6);
        \node(r)at(5.00,2){};
        \draw[Edge](r)--(6);
    \end{tikzpicture}
    \enspace = \enspace
    \begin{tikzpicture}[Centering,xscale=0.17,yscale=0.17]
        \node(0)at(0.00,-5.00){};
        \node(2)at(1.00,-5.00){};
        \node(3)at(2.00,-7.50){};
        \node(5)at(4.00,-7.50){};
        \node(7)at(6.00,-5.00){};
        \node(9)at(8.00,-5.00){};
        \node[NodeST](1)at(1.00,-2.50){$\GenC$};
        \node[NodeST](4)at(3.00,-5.00){$\GenA$};
        \node[NodeST](6)at(5.00,0.00){$\GenA$};
        \node[NodeST](8)at(7.00,-2.50){$\GenA$};
        \draw[Edge](0)--(1);
        \draw[Edge](1)--(6);
        \draw[Edge](2)--(1);
        \draw[Edge](3)--(4);
        \draw[Edge](4)--(1);
        \draw[Edge](5)--(4);
        \draw[Edge](7)--(8);
        \draw[Edge](8)--(6);
        \draw[Edge](9)--(8);
        \node(r)at(5.00,2){};
        \draw[Edge](r)--(6);
    \end{tikzpicture}.
\end{equation}
\medbreak

\begin{Lemma} \label{lem:prefix_tree_poset_join}
    For any finite alphabet $\GeneratingSet$ and any $\GeneratingSet$-trees $\TreeT$ and
    $\TreeT'$ such that $\Bra{\TreeT, \TreeT'}$ admits an upper bound in the
    $\GeneratingSet$-prefix poset, $\TreeT \JJoin \TreeT'$ is well-defined and is the least
    upper bound of~$\Bra{\TreeT, \TreeT'}$.
\end{Lemma}
\begin{proof}
    We use here Proposition~\ref{prop:prefix_tree_poset} and its description of the order
    relation $\OrderPrefixes$ of the $\GeneratingSet$-prefix poset in terms of prefixes of
    $\GeneratingSet$-trees.  Let us denote by $U\Par{\TreeT, \TreeT'}$ the set of all upper
    bounds of $\Bra{\TreeT, \TreeT'}$. By structural induction on $\TreeT$ and $\TreeT'$, we
    show that $\min_{\OrderPrefixes} U\Par{\TreeT, \TreeT'}$ exists and that
    $\min_{\OrderPrefixes} U\Par{\TreeT, \TreeT'} = \TreeT \JJoin \TreeT'$.  First,
    immediately from the definition of $\JJoin$, we have
    \begin{math}
        \Leaf \JJoin \TreeT = \TreeT \JJoin \Leaf = \TreeT
    \end{math}
    and
    \begin{math}
        U\Par{\Leaf, \TreeT} = U\Par{\TreeT, \Leaf} = \Bra{\TreeT}
    \end{math}
    so that the statement of the lemma holds in this case.  Assume now that $\TreeT$ and
    $\TreeT'$ are both different from the leaf so that
    \begin{math}
        \TreeT = \GenA \BPar{\TreeT_1, \dots, \TreeT_{|\GenA|}}
    \end{math}
    and
    \begin{math}
        \TreeT' = \GenA' \BPar{\TreeT'_1, \dots, \TreeT'_{\left|\GenA'\right|}}
    \end{math}
    where $\GenA, \GenA' \in \GeneratingSet$ and $\TreeT_1$, \dots $\TreeT_{|\GenA|}$,
    $\TreeT'_1$, \dots, $\TreeT'_{\left|\GenA'\right|}$ are $\GeneratingSet$-trees.  Since
    $\Bra{\TreeT, \TreeT'}$ admits, by hypothesis, an upper bound, both $\TreeT$ and
    $\TreeT'$ have to be prefixes of a same $\GeneratingSet$-tree. This implies that $\GenA
    = \GenA'$.  It follows by induction hypothesis that
    \begin{equation} \label{equ:prefix_tree_poset_join_1}
        \TreeT \JJoin \TreeT'
        = \GenA \BPar{\TreeT_1 \JJoin \TreeT'_1, \dots,
        \TreeT_{|\GenA|} \JJoin \TreeT'_{|\GenA|}}
        = \GenA \BPar{\TreeS_1, \dots, \TreeS_{|\GenA|}},
    \end{equation}
    where for any $i \in [|\GenA|]$, $\TreeS_i := \min_{\OrderPrefixes} U\Par{\TreeT_i,
    \TreeT'_i}$.  Observe that the $\GeneratingSet$-tree specified
    by~\eqref{equ:prefix_tree_poset_join_1} is well-defined by induction hypothesis. Indeed,
    by calling $\TreeR$ an upper bound of $\Bra{\TreeT, \TreeT'}$, for any $i \in
    [|\GenA|]$, $\TreeR_i$ is an upper bound of $\Bra{\TreeT_i, \TreeT'_i}$.  Now, since for
    any $i \in [|\GenA|]$, both $\TreeT_i$ and $\TreeT'_i$ are prefixes of $\TreeS_i$, and
    since all trees of $U\Par{\TreeT, \TreeT'}$ have a root decorated by $\GenA$, by
    Lemma~\ref{lem:prefix_recursive} both $\TreeT$ and $\TreeT'$ are prefixes of $\TreeT
    \JJoin \TreeT'$ so that $\TreeT \JJoin \TreeT' \in U\Par{\TreeT, \TreeT'}$.  To show
    finally that $\TreeT \JJoin \TreeT'$ is the least element of $U\Par{\TreeT, \TreeT'}$,
    assume that $\TreeR$ is a $\GeneratingSet$-tree to $U\Par{\TreeT, \TreeT'}$.  Since for
    any $i \in [|\GenA|]$, both $\TreeT_i$ and $\TreeT'_i$ are prefixes of $\TreeR(i)$, and
    since $\TreeS_i$ is the smallest $\GeneratingSet$-tree admitting both $\TreeT_i$ and
    $\TreeT'_i$ as prefixes, $\TreeS_i$ is a prefix of $\TreeR(i)$.  Therefore, by
    Lemma~\ref{lem:prefix_recursive}, this implies that $\TreeT \JJoin \TreeT'$ is a prefix
    of $\TreeR$ and establishes the statement of the lemma.
\end{proof}
\medbreak

By seeing $\GeneratingSet$-trees as terms, the term encoded by $\TreeT \JJoin \TreeT'$ is in
fact, if it exists, the unification of the terms encoded by the $\GeneratingSet$-trees
$\TreeT$ and $\TreeT'$ (see~\cite{BN98,Ter03}).
\medbreak

\begin{Proposition} \label{prop:prefix_tree_poset_meet_semilattice}
    For any finite alphabet $\GeneratingSet$, the $\GeneratingSet$-prefix poset is a
    meet-semilattice for the operation~$\Meet$. Moreover, each interval $\Han{\TreeS,
    \TreeT}$ of this poset is a distributive lattice for the operations $\Meet$
    and~$\JJoin$.
\end{Proposition}
\begin{proof}
    Lemma~\ref{lem:prefix_tree_poset_meet} says that the $\GeneratingSet$-prefix poset is a
    meet-semilattice for the operation $\Meet$. Moreover, by
    Lemma~\ref{lem:prefix_tree_poset_join}, the operation $\JJoin$ is well-defined for any
    pair of elements of the interval $I := \Han{\TreeS, \TreeT}$ since $\TreeT$ is an upper
    bound of any pair of trees of $I$. Hence, $I$ is a join-semilattice and thus also a
    lattice for the operations $\Meet$ and~$\JJoin$.
    \smallbreak

    Let us now prove that $I$ is a distributive lattice. We proceed by structural induction
    on the three $\GeneratingSet$-trees $\TreeR_1$, $\TreeR_2$, and $\TreeR_3$ of $I$ to
    show that
    \begin{math}
        \alpha_{\TreeR_1, \TreeR_2, \TreeR_3} = \beta_{\TreeR_1, \TreeR_2, \TreeR_3}
    \end{math}
    where
    \begin{math}
        \alpha_{\TreeR_1, \TreeR_2, \TreeR_3}
        := \TreeR_1 \Meet \Par{\TreeR_2 \JJoin \TreeR_3}
    \end{math}
    and
    \begin{math}
        \beta_{\TreeR_1, \TreeR_2, \TreeR_3}
        := \Par{\TreeR_1 \Meet \TreeR_2} \JJoin \Par{\TreeR_1 \Meet \TreeR_3}.
    \end{math}
    First, we have
    \begin{equation}
        \alpha_{\Leaf, \TreeR_2, \TreeR_3}
        = \Leaf \Meet \Par{\TreeR_2 \JJoin \TreeR_3}
        = \Leaf,
    \end{equation}
    and
    \begin{equation}
        \beta_{\Leaf, \TreeR_2, \TreeR_3}
        = \Par{\Leaf \Meet \TreeR_2} \JJoin \Par{\Leaf \Meet \TreeR_3}
        = \Leaf \JJoin \Leaf
        = \Leaf.
    \end{equation}
    Second, we have
    \begin{equation}
        \alpha_{\TreeR_1, \Leaf, \TreeR_3}
        = \TreeR_1 \Meet \Par{\Leaf \JJoin \TreeR_3}
        = \TreeR_1 \Meet \TreeR_3,
    \end{equation}
    and
    \begin{equation}
        \beta_{\TreeR_1, \Leaf, \TreeR_3}
        = \Par{\TreeR_1 \Meet \Leaf} \JJoin \Par{\TreeR_1 \Meet \TreeR_3}
        = \Leaf \JJoin \Par{\TreeR_1 \Meet \TreeR_3}
        = \TreeR_1 \Meet \TreeR_3.
    \end{equation}
    Similarly, the relation
    \begin{math}
        \alpha_{\TreeR_1, \TreeR_2, \Leaf}
        = \TreeR_1 \Meet \TreeR_2
        = \beta_{\TreeR_1, \TreeR_2, \Leaf}
    \end{math}
    holds. We can now assume that $\TreeR_1$, $\TreeR_2$, and $\TreeR_3$ are different from
    the leaf. Moreover, since $\TreeR_1 \OrderPrefixes \TreeT$, $\TreeR_2 \OrderPrefixes
    \TreeT$, and $\TreeR_3 \OrderPrefixes \TreeT$, the roots of $\TreeR_1$, $\TreeR_2$, and
    $\TreeR_3$ are decorated by the same letter $\GenA$ of $\GeneratingSet$. Therefore,
    \begin{equation}\begin{split}
        \alpha_{\TreeR_1, \TreeR_2, \TreeR_3}
        & = \TreeR_1 \Meet
        \GenA \BPar{\TreeR_2(1) \JJoin \TreeR_3(1), \dots
            \TreeR_2(|\GenA|) \JJoin \TreeR_3(|\GenA|)} \\
        & = \GenA \BPar{\TreeR_1(1) \Meet
            \Par{\TreeR_2(1) \JJoin \TreeR_3(1)}, \dots,
            \TreeR_1(|\GenA|) \Meet
            \Par{\TreeR_2(|\GenA|) \JJoin \TreeR_3(|\GenA|)}} \\
        & =
        \GenA \BPar{\alpha_{\TreeR_1(1), \TreeR_2(1), \TreeR_3(1)},
        \dots, \alpha_{\TreeR_1(|\GenA|), \TreeR_2(|\GenA|), \TreeR_3(|\GenA|)}},
    \end{split}\end{equation}
    and
    \begin{equation}\begin{split}
        \beta_{\TreeR_1, \TreeR_2, \TreeR_3}
        & =
        \GenA \BPar{\TreeR_1(1) \Meet \TreeR_2(1), \dots,
            \TreeR_1(|\GenA|) \Meet \TreeR_2(|\GenA|)}
        \JJoin
        \GenA \BPar{\TreeR_1(1) \Meet \TreeR_3(1), \dots,
            \TreeR_1(|\GenA|) \Meet \TreeR_3(|\GenA|)} \\
        & =
        \GenA \BPar{
            \Par{\TreeR_1(1) \Meet \TreeR_2(1)}
                \JJoin \Par{\TreeR_1(1) \Meet \TreeR_3(1)},
            \dots,
            \Par{\TreeR_1(|\GenA|) \Meet \TreeR_2(|\GenA|)}
                \JJoin \Par{\TreeR_1(|\GenA|) \Meet \TreeR_3(|\GenA|)}} \\
        & =
        \GenA \BPar{\beta_{\TreeR_1(1), \TreeR_2(1), \TreeR_3(1)},
        \dots, \beta_{\TreeR_1(|\GenA|), \TreeR_2(|\GenA|), \TreeR_3(|\GenA|)}}.
    \end{split}\end{equation}
    By induction hypothesis, the relation
    \begin{math}
        \alpha_{\TreeR_1, \TreeR_2, \TreeR_3} = \beta_{\TreeR_1, \TreeR_2, \TreeR_3}
    \end{math}
    follows.
\end{proof}
\medbreak

%%%%%%%%%%%%%%%%%%%%%%%%%%%%%%%%%%%%%%%%%%%%%%%%%%%%%%%%%%%%%%%%%%%%%%%%%%%%%%%%%%%%%%%%%%%%
\subsubsection{Structure of the intervals}
A $\GeneratingSet$-tree $\TreeT$ is \Def{stringy} if any internal node of $\TreeT$ has at
most one child which is an internal node.  A $\GeneratingSet$-tree $\TreeT$ is
\Def{co-irreducible} in $\Par{\SyntaxTreesInternalNode(\GeneratingSet), \OrderPrefixes}$ if
$\TreeT$ covers at most one element.
\medbreak

\begin{Proposition} \label{prop:prefix_tree_poset_co_irreducible}
    For any finite alphabet $\GeneratingSet$, the set of co-irreducible elements of the
    $\GeneratingSet$-prefix poset is the set of all stringy $\GeneratingSet$-trees.
    Moreover, the number of such elements of degree $d \geq 1$ is
    \begin{math}
        \RankSeries_\GeneratingSet(1) \RankSeries_\GeneratingSet'(1)^{d - 1},
    \end{math}
    where $\RankSeries_\GeneratingSet'$ is the derivative of $\RankSeries_\GeneratingSet(t)$
    with respect to~$t$.
\end{Proposition}
\begin{proof}
    We use here Proposition~\ref{prop:prefix_tree_poset} and its description of the covering
    relation $\Covered_\Up$ of the $\GeneratingSet$-prefix poset in terms of deletion of
    maximal nodes.  First, if $\TreeT$ is a stringy $\GeneratingSet$-tree different from the
    leaf, by definition of stringy trees, $\TreeT$ admits exactly one maximal internal node
    $u$.  Therefore, $\DelNode{\TreeT}{u}$ is the only tree covered by $\TreeT$. Conversely,
    if $\TreeT$ covers exactly one $\GeneratingSet$-tree $\TreeT'$, then $\TreeT$ has only
    one maximal internal node.  This implies that $\TreeT$ is stringy. This establishes the
    first part of the statement.
    \smallbreak

    By definition, a stringy $\GeneratingSet$-tree $\TreeT$ decomposes as
    \begin{equation}
        \TreeT =
        \GenA_1 \circ_{i_1} \Par{\GenA_2 \circ_{i - 2} \Par {\dots
         \Par{\GenA_{d - 1} \circ_{i_{d - 1}} \GenA_d} \dots}}
    \end{equation}
    where $\Par{\GenA_1, \dots, \GenA_d}$ is a sequence of elements of $\GeneratingSet$ and
    $\Par{i_1, \dots, i_{d - 1}}$ is a sequence of indices satisfying $i_j \in
    \Han{\Brr{\GenA_j}}$ for any $j \in [d - 1]$.  This tree $\TreeT$ is moreover entirely
    specified by these two sequences.  For this reason, by denoting by $\theta(d)$ the
    number of stringy $\GeneratingSet$-trees of degree $d \geq 0$, we have
    \begin{equation} \begin{split}
        \theta(d)
        & =
        \sum_{\Par{\GenA_1, \dots, \GenA_d} \in \GeneratingSet^d}
        \Brr{\GenA_1} \dots \Brr{\GenA_{d - 1}} \\
        & =
        \Par{\sum_{\GenA \in \GeneratingSet} |\GenA|}^{d- 1}
        (\# \GeneratingSet) \\
        & =
        \RankSeries_\GeneratingSet'(1)^{d - 1}
        \RankSeries_\GeneratingSet(1).
    \end{split} \end{equation}
    This shows the second part of the statement.
\end{proof}
\medbreak

Here are the sequences of the first numbers of stringy $\GeneratingSet$-trees:
\begin{subequations}
\begin{equation}
    1, 1, 2, 4, 8, 16, 32, 64,
    \qquad \mbox{ for } \GeneratingSet = \{\GenA\} \mbox{ with } |\GenA| = 2,
\end{equation}
\begin{equation}
    1, 1, 3, 9, 27, 81, 243, 729,
    \qquad \mbox{ for } \GeneratingSet = \{\GenC\} \mbox{ with } |\GenC| = 3,
\end{equation}
\begin{equation}
    1, 2, 8, 32, 128, 512, 2048, 8192,
    \qquad \mbox{ for } \GeneratingSet = \{\GenA, \GenB\} \mbox{ with } |\GenA| = |\GenB|
    = 2,
\end{equation}
\begin{equation}
    1, 2, 10, 50, 250, 1250, 6250, 31250,
    \qquad \mbox{ for } \GeneratingSet = \{\GenA, \GenC\} \mbox{ with } |\GenA| =2, |\GenC|
    = 3,
\end{equation}
\end{subequations}
\medbreak

A \Def{$\GeneratingSet$-forest} is a nonempty word of $\GeneratingSet$-trees. The
\Def{length} of a $\GeneratingSet$-forest is the number of trees it contains. If $\TreeS
\OrderPrefixes \TreeT$, the \Def{difference} between $\TreeT$ and $\TreeS$ is the
$\GeneratingSet$-forest $\TreeT \setminus \TreeS := (\TreeR_1, \dots, \TreeR_{|\TreeS|})$
such that $\TreeR_1$, \dots, $\TreeR_{|\TreeS|}$ are the unique $\GeneratingSet$-trees such
that
\begin{math}
    \TreeT = \TreeS \circ \Han{\TreeR_1, \dots, \TreeR_{|\TreeS|}}.
\end{math}
Moreover, from any $\GeneratingSet$-forest $\Par{\TreeR_1, \dots, \TreeR_k}$, we denote by
$\lozenge_k \BPar{\TreeR_1, \dots, \TreeR_k}$ the $\GeneratingSet_{\lozenge_k}$-tree
obtained by grafting the $\GeneratingSet$-trees $\TreeR_1$, \dots, $\TreeR_k$ to a root
decorated by the letter $\lozenge_k$ of arity $k$, where $\GeneratingSet_{\lozenge_k}$ is the
alphabet~$\GeneratingSet \sqcup \Bra{\lozenge_k}$.
\medbreak

\begin{Proposition} \label{prop:prefix_tree_poset_interval_decomposition}
    Let $\GeneratingSet$ be a finite alphabet, and $\TreeS$ and $\TreeT$ be two
    $\GeneratingSet$-trees such that $\TreeS \OrderPrefixes \TreeT$. As subposets of
    $\Par{\SyntaxTreesInternalNode(\GeneratingSet), \OrderPrefixes}$, one has the poset
    isomorphisms
    \begin{equation}
        \Han{\TreeS, \TreeT}
        \simeq \Han{\Leaf, \TreeR_1} \times \dots \times \Han{\Leaf, \TreeR_{|\TreeS|}}
        \simeq \Han{\Corolla\Par{\lozenge_{|\TreeS|}}, \lozenge_{|\TreeS|}
        \BPar{\TreeR_1, \dots, \TreeR_{|\TreeS|}}},
    \end{equation}
    where $\Par{\TreeR_1, \dots, \TreeR_{|\TreeS|}}$ is the $\GeneratingSet$-forest
    $\TreeT \setminus \TreeS$.
\end{Proposition}
\begin{proof}
    We use here Proposition~\ref{prop:prefix_tree_poset} and its description of the order
    relation $\OrderPrefixes$ of the $\GeneratingSet$-prefix poset in terms of prefixes of
    $\GeneratingSet$-trees.  Let us call $\PosetP$ the poset
    \begin{math}
        \Han{\Leaf, \TreeR_1} \times \dots \times \Han{\Leaf, \TreeR_{|\TreeS|}}
    \end{math}
    and let us denote by $\Leq$ its partial order relation. Let
    \begin{math}
        \psi : \Han{\TreeS, \TreeT} \to \PosetP
    \end{math}
    be the map defined for any $\GeneratingSet$-tree $\TreeU \in \Han{\TreeS, \TreeT}$ by
    $\psi(\TreeU) := \Par{\TreeU_1, \dots, \TreeU_{|\TreeS|}}$ where
    \begin{math}
        \Par{\TreeU_1, \dots, \TreeU_{|\TreeS|}} = \TreeU \setminus \TreeS.
    \end{math}
    This map is well-defined because by Lemma~\ref{lem:prefix_recursive}, $\TreeU_i$ is a
    prefix of $\TreeR_i$ for any $i \in \Han{|\TreeS|}$. Since $\psi$ admits as inverse the
    map $\psi^{-1}$ satisfying
    \begin{math}
        \psi^{-1}\Par{\Par{\TreeU_1, \dots, \TreeU_{|\TreeS|}}}
        = \TreeS \circ \Han{\TreeU_1, \dots, \TreeU_{|\TreeS|}},
    \end{math}
    $\psi$ is a bijection.  Assume that $x := \Par{\TreeU_1, \dots, \TreeU_{|\TreeS|}}$ and
    $y := \Par{\TreeV_1, \dots, \TreeV_{|\TreeS|}}$ are elements of $\PosetP$. Now, $x \Leq
    y$ is equivalent to the fact that $\TreeU_i \OrderPrefixes \TreeV_i$ for all $i \in
    [|\TreeS|]$. This, again by Lemma~\ref{lem:prefix_recursive}, is in turn equivalent to
    the fact that
    \begin{math}
        \TreeS \circ \Han{\TreeU_1, \dots, \TreeU_{|\TreeS|}}
        \OrderPrefixes
        \TreeS \circ \Han{\TreeV_1, \dots, \TreeV_{|\TreeS|}},
    \end{math}
    that is
    \begin{math}
        \psi^{-1}(x) \OrderPrefixes \psi^{-1}(y).
    \end{math}
    Therefore, this establishes the first isomorphism of the statement of the proposition.
    \smallbreak

    Let us call $\PosetQ$ the poset
    \begin{math}
        \Han{\Corolla\Par{\lozenge_{|\TreeS|}},
        \lozenge_{|\TreeS|} \BPar{\TreeR_1, \dots, \TreeR_{|\TreeS|}}}
    \end{math}
    and let $\psi' : \PosetP \to \PosetQ$ be the map defined for any $\Par{\TreeU_1, \dots,
    \TreeU_{|\TreeS|}} \in \PosetP$ by
    \begin{math}
        \psi'\Par{\Par{\TreeU_1, \dots, \TreeU_{|\TreeS|}}}
        := \lozenge_{|\TreeS|} \BPar{\TreeU_1, \dots, \TreeU_{|\TreeS|}}.
    \end{math}
    Again by Lemma~\ref{lem:prefix_recursive}, it follows that $\psi'$ is a well-defined
    map, which is additionally  a bijection, and  a poset embedding.
\end{proof}
\medbreak

For instance, by considering the same alphabet $\GeneratingSet$ as in the previous examples,
Proposition~\ref{prop:prefix_tree_poset_interval_decomposition} says that one has the
isomorphism
\begin{equation}
    \Han{
    \begin{tikzpicture}[Centering,xscale=0.18,yscale=0.2]
        \node(0)at(0.00,-4.00){};
        \node(2)at(2.00,-6.00){};
        \node(4)at(4.00,-6.00){};
        \node(6)at(5.00,-2.00){};
        \node(7)at(6.00,-2.00){};
        \node[NodeST](1)at(1.00,-2.00){$\GenA$};
        \node[NodeST](3)at(3.00,-4.00){$\GenA$};
        \node[NodeST](5)at(5.00,0.00){$\GenC$};
        \draw[Edge](0)--(1);
        \draw[Edge](1)--(5);
        \draw[Edge](2)--(3);
        \draw[Edge](3)--(1);
        \draw[Edge](4)--(3);
        \draw[Edge](6)--(5);
        \draw[Edge](7)--(5);
        \node(r)at(5.00,2){};
        \draw[Edge](r)--(5);
    \end{tikzpicture},
    \begin{tikzpicture}[Centering,xscale=0.19,yscale=0.12]
        \node(0)at(0.00,-10.80){};
        \node(10)at(7.00,-3.60){};
        \node(11)at(8.00,-10.80){};
        \node(13)at(9.00,-10.80){};
        \node(14)at(10.00,-10.80){};
        \node(17)at(12.00,-10.80){};
        \node(2)at(1.00,-10.80){};
        \node(3)at(2.00,-10.80){};
        \node(5)at(4.00,-10.80){};
        \node(8)at(6.00,-14.40){};
        \node[NodeST](1)at(1.00,-7.20){$\GenC$};
        \node[NodeST](12)at(9.00,-7.20){$\GenC$};
        \node[NodeST](15)at(11.00,-3.60){$\GenA$};
        \node[NodeST](16)at(12.00,-7.20){$\GenE$};
        \node[NodeST](4)at(3.00,-3.60){$\GenA$};
        \node[NodeST](6)at(5.00,-7.20){$\GenA$};
        \node[NodeST](7)at(6.00,-10.80){$\GenE$};
        \node[NodeST](9)at(7.00,0.00){$\GenC$};
        \draw[Edge](0)--(1);
        \draw[Edge](1)--(4);
        \draw[Edge](10)--(9);
        \draw[Edge](11)--(12);
        \draw[Edge](12)--(15);
        \draw[Edge](13)--(12);
        \draw[Edge](14)--(12);
        \draw[Edge](15)--(9);
        \draw[Edge](16)--(15);
        \draw[Edge](17)--(16);
        \draw[Edge](2)--(1);
        \draw[Edge](3)--(1);
        \draw[Edge](4)--(9);
        \draw[Edge](5)--(6);
        \draw[Edge](6)--(4);
        \draw[Edge](7)--(6);
        \draw[Edge](8)--(7);
        \node(r)at(7.00,3.25){};
        \draw[Edge](r)--(9);
    \end{tikzpicture}}
    \simeq
    \Han{
    \begin{tikzpicture}[Centering,xscale=0.19,yscale=0.17]
        \node(0)at(0.00,-3.00){};
        \node(1)at(1.00,-3.00){};
        \node(3)at(2.00,-3.00){};
        \node(4)at(3.00,-3.00){};
        \node(5)at(4.00,-3.00){};
        \node[NodeST](2)at(2.00,0.00){$\lozenge_5$};
        \draw[Edge](0)--(2);
        \draw[Edge](1)--(2);
        \draw[Edge](3)--(2);
        \draw[Edge](4)--(2);
        \draw[Edge](5)--(2);
        \node(r)at(2.00,2.5){};
        \draw[Edge](r)--(2);
    \end{tikzpicture},
    \begin{tikzpicture}[Centering,xscale=0.2,yscale=0.12]
        \node(0)at(0.00,-8.00){};
        \node(11)at(7.00,-12.00){};
        \node(12)at(8.00,-12.00){};
        \node(15)at(10.00,-12.00){};
        \node(2)at(1.00,-8.00){};
        \node(3)at(2.00,-8.00){};
        \node(4)at(3.00,-4.00){};
        \node(7)at(4.00,-8.00){};
        \node(8)at(5.00,-4.00){};
        \node(9)at(6.00,-12.00){};
        \node[NodeST](1)at(1.00,-4.00){$\GenC$};
        \node[NodeST](10)at(7.00,-8.00){$\GenC$};
        \node[NodeST](13)at(9.00,-4.00){$\GenA$};
        \node[NodeST](14)at(10.00,-8.00){$\GenE$};
        \node[NodeST](5)at(4.00,0.00){$\lozenge_5$};
        \node[NodeST](6)at(4.00,-4.00){$\GenE$};
        \draw[Edge](0)--(1);
        \draw[Edge](1)--(5);
        \draw[Edge](10)--(13);
        \draw[Edge](11)--(10);
        \draw[Edge](12)--(10);
        \draw[Edge](13)--(5);
        \draw[Edge](14)--(13);
        \draw[Edge](15)--(14);
        \draw[Edge](2)--(1);
        \draw[Edge](3)--(1);
        \draw[Edge](4)--(5);
        \draw[Edge](6)--(5);
        \draw[Edge](7)--(6);
        \draw[Edge](8)--(5);
        \draw[Edge](9)--(10);
        \node(r)at(4.00,4){};
        \draw[Edge](r)--(5);
    \end{tikzpicture}}
\end{equation}
between respectively an interval of $\Par{\SyntaxTreesInternalNode(\GeneratingSet),
\OrderPrefixes}$ and an interval of
$\Par{\SyntaxTreesInternalNode\Par{\GeneratingSet_{\lozenge_5}}, \OrderPrefixes}$.
\medbreak

A \Def{shadow} is defined recursively as being a (possibly empty) finite multiset $\lbag
s_1, \dots, s_k \rbag$ of shadows. A shadow encodes hence a nonplanar undecorated rooted
tree. For any $\GeneratingSet$-tree $\TreeT$ different from the leaf, we construct the
shadow $\Shadow(\TreeT)$ recursively by
\begin{equation}
    \Shadow(\TreeT) :=
    \lbag \Shadow(\TreeT(i)) : i \in [k] \mbox{ and } \TreeT(i) \ne \Leaf \rbag,
\end{equation}
where $k$ is the arity of the root of $\TreeT$. For instance,
\begin{equation} \label{equ:example_shadow}
    \Shadow\Par{
    \begin{tikzpicture}[Centering,xscale=0.19,yscale=0.12]
        \node(0)at(0.00,-7.00){};
        \node(10)at(7.00,-10.50){};
        \node(11)at(7.00,-7.00){};
        \node(13)at(9.00,-7.00){};
        \node(2)at(2.00,-7.00){};
        \node(5)at(3.00,-10.50){};
        \node(7)at(4.50,-7.00){};
        \node(8)at(5.00,-10.50){};
        \node[NodeST](1)at(1.00,-3.50){$\GenA$};
        \node[NodeST](12)at(8.00,-3.50){$\GenA$};
        \node[NodeST](3)at(4.50,0.00){$\GenC$};
        \node[NodeST](4)at(3.00,-7.00){$\GenE$};
        \node[NodeST](6)at(4.50,-3.50){$\GenC$};
        \node[NodeST](9)at(6.00,-7.00){$\GenA$};
        \draw[Edge](0)--(1);
        \draw[Edge](1)--(3);
        \draw[Edge](10)--(9);
        \draw[Edge](11)--(12);
        \draw[Edge](12)--(3);
        \draw[Edge](13)--(12);
        \draw[Edge](2)--(1);
        \draw[Edge](4)--(6);
        \draw[Edge](5)--(4);
        \draw[Edge](6)--(3);
        \draw[Edge](7)--(6);
        \draw[Edge](8)--(9);
        \draw[Edge](9)--(6);
        \node(r)at(4.50,3.5){};
        \draw[Edge](r)--(3);
    \end{tikzpicture}}
    = \lbag \lbag \emptyset, \emptyset\rbag, \emptyset, \emptyset \rbag
    =
    \begin{tikzpicture}[Centering,xscale=0.17,yscale=0.19]
        \node[Node](0)at(0.00,-4.00){};
        \node[Node](2)at(2.00,-4.00){};
        \node[Node](4)at(3.00,-2.00){};
        \node[Node](5)at(5.00,-2.00){};
        \node[Node](1)at(1.00,-2.00){};
        \node[Node](3)at(3.00,0.00){};
        \draw[Edge](0)--(1);
        \draw[Edge](1)--(3);
        \draw[Edge](2)--(1);
        \draw[Edge](4)--(3);
        \draw[Edge](5)--(3);
        \node(r)at(3,1.75){};
        \draw[Edge](r)--(3);
    \end{tikzpicture}.
\end{equation}
\medbreak

Given a shadow $s$, the \Def{poset induced} by $s$ is the poset $\ShadowPoset{s}$ on its set
of nodes different from the root wherein a node $u$ is smaller than a node $v$ if $u$ is an
ancestor of $v$. In other words, $\ShadowPoset{s}$ is the poset having as Hasse diagram the
nonplanar rooted tree $s$ without its root.  Besides, we say that a poset $(\PosetP, \Leq)$
is a \Def{forest poset} if $x \Leq y$ and $x' \Leq y$ imply $x \Leq x'$ or $x' \Leq x$ for
all $x, x', y \in \PosetP$. Observe that any poset induced by a shadow is a forest poset and
conversely, for any forest poset $\PosetP$, there is a shadow $s$ such that $\PosetP$ and
$\ShadowPoset{s}$ are isomorphic.
\medbreak

For any poset $(\PosetP, \Leq)$, let $(\JLattice{\PosetP}, \Meet, \JJoin)$ be the lattice
wherein
\begin{equation}
    \JLattice{\PosetP} := \Bra{X \subseteq \PosetP :
    x \in X \mbox{ and } y \in \PosetP \mbox{ and }y \Leq x \mbox{ imply } y \in X},
\end{equation}
and the operation $\Meet$ (resp. $\JJoin$) is the intersection (resp. the union) of the
sets. In other words, $\JLattice{\PosetP}$ is the set of all order ideals of $\PosetP$
ordered by inclusion. The \Def{Fundamental theorem for distributive lattices}
(see~\cite{Sta11}) states that for any finite distributive lattice $\LatticeL$, there exists
a unique finite poset $\PosetP$ such that $\LatticeL$ and $\JLattice{\PosetP}$ are
isomorphic as lattices.  An element $x$ of a finite lattice $\LatticeL$ is
\Def{join-irreducible} if $x$ covers exactly one element. It is known that the set of all
join-irreducible elements of $\JLattice{\PosetP}$ forms a subposet of $\JLattice{\PosetP}$
which is isomorphic as a poset to~$\PosetP$.
\medbreak

\begin{Lemma} \label{lem:join_irreducible_shadow_lattices}
    Let $s$ be a shadow. The set of join-irreducible elements of the lattice
    $\JLattice{\ShadowPoset{s}}$ is the set of the nonempty saturated chains of
    $\ShadowPoset{s}$.
\end{Lemma}
\begin{proof}
    Let us denote by $\Covered$ the covering relation of the poset $\ShadowPoset{s}$. First,
    $\emptyset$ is not a join-irreducible element of $\JLattice{\ShadowPoset{s}}$ since
    $\emptyset$ covers no elements.  Any nonempty saturated chain
    \begin{math}
        x_1 \Covered \cdots \Covered x_{\ell - 1} \Covered x_\ell
    \end{math}
    of $\ShadowPoset{s}$ covers exactly the chain
    \begin{math}
        x_1 \Covered \cdots \Covered x_{\ell - 1},
    \end{math}
    so that any nonempty saturated chain is join-irreducible. Finally, if $X$ is an element
    of $\JLattice{\ShadowPoset{s}}$ which is not a chain, there are $x, x' \in X$ such that
    $x$ and $x'$ are incomparable in $\ShadowPoset{s}$. Since $\ShadowPoset{s}$ is a forest
    poset, we can assume that $x$ and $x'$ are maximal elements of $X$. Hence, $X$ covers
    the elements $X \setminus \{x\}$ and $X \setminus \Bra{x'}$ of
    $\JLattice{\ShadowPoset{s}}$. This shows that $X$ is not join-irreducible and
    establishes the statement of the lemma.
\end{proof}
\medbreak

\begin{Lemma} \label{lem:join_irreducible_intervals}
    For any finite alphabet $\GeneratingSet$ and any $\GeneratingSet$-trees $\TreeR_1$,
    \dots, $\TreeR_k$, the set of join-irreducible elements of the lattice
    \begin{math}
        \Han{\Corolla\Par{\lozenge_k}, \lozenge_k \BPar{\TreeR_1, \dots, \TreeR_k}}
    \end{math}
    is the set of all $\GeneratingSet_{\lozenge_k}$-trees of the form
    $\Corolla\Par{\lozenge_k} \circ_i \TreeR'_i$ where $i \in [k]$ and $\TreeR'_i$ is a
    stringy tree, different from the leaf, and a prefix of $\TreeR_i$.
\end{Lemma}
\begin{proof}
    We use here Proposition~\ref{prop:prefix_tree_poset} and its descriptions of the order
    relation $\OrderPrefixes$ and of the covering relation $\Covered_\Up$ of the
    $\GeneratingSet$-prefix poset respectively in terms of prefixes of
    $\GeneratingSet$-trees and of deletion of maximal nodes.  Let $\TreeT :=
    \Corolla\Par{\lozenge_k} \circ_i \TreeR'_i$.  Since $\TreeR'_i$ is stringy, $\TreeT$
    also is. For this reason, $\TreeT$ covers at most one element in
    \begin{math}
        \Han{\Corolla\Par{\lozenge_k}, \lozenge_k \BPar{\TreeR_1, \dots, \TreeR_k}}.
    \end{math} 
    Moreover, due to the fact that $\TreeR'_i$ is by hypothesis different from the leaf,
    there is a $\GeneratingSet$-tree $\TreeR''$, a $j \in [|\TreeR''|]$, and a letter $\GenA
    \in \GeneratingSet$ such that $\TreeR'_i = \TreeR'' \circ_j \GenA$. Now, by using
    Relation~\eqref{equ:operad_axiom_1} satisfied by the partial composition maps of
    $\SyntaxTreesInternalNode\Par{\GeneratingSet_{\lozenge_k}}$, we have
    \begin{equation}
        \TreeT
        = \Corolla\Par{\lozenge_k} \circ_i \TreeR'_i
        = \Corolla\Par{\lozenge_k} \circ_i \Par{\TreeR'' \circ_j \GenA}
        = \Par{\Corolla\Par{\lozenge_k} \circ_i \TreeR''} \circ_{i + j - 1} \GenA.
    \end{equation}
    This shows that $\TreeT$ covers only $\Corolla\Par{\lozenge_k} \circ_i \TreeR''$. It
    remains to show that when $\TreeT$ is a tree different from the description of the
    statement of the lemma, $\TreeT$ covers zero or two or more elements. First, if $\TreeT
    = \Corolla\Par{\lozenge_k}$, $\TreeT$ covers no elements. Second, if $\TreeT =
    \Corolla\Par{\lozenge_k} \circ_i \TreeR'_i$ where $\TreeR'_i$ is not stringy, there are
    at least two maximal internal nodes in $\TreeR'_i$. By removing one of these internal
    nodes, one obtains at least two different trees $\TreeR''_1$ and $\TreeR''_2$ covered by
    $\TreeR'_i$. Thus, $\Corolla\Par{\lozenge_k} \circ_i \TreeR''_1$ and
    $\Corolla\Par{\lozenge_k} \circ_i \TreeR''_2$ are both covered by $\TreeT$. Finally, it
    remains to consider the case where $\TreeT = \Corolla\Par{\lozenge_k} \circ
    \Han{\TreeR'_1, \dots, \TreeR'_k}$ where $\TreeR'_1$, \dots, $\TreeR'_k$ are
    $\GeneratingSet$-trees such that for any $i \in [k]$, $\TreeR'_i \OrderPrefixes
    \TreeR_i$, and there are at least two indices $j, \ell \in [k]$ such that $j \ne \ell$,
    $\TreeR'_j \ne \Leaf$, and $\TreeR'_\ell \ne \Leaf$.  By
    Lemma~\ref{lem:prefix_recursive},
    \begin{math}
        \TreeT \in
        \Han{\Corolla\Par{\lozenge_k}, \lozenge_k \BPar{\TreeR_1, \dots, \TreeR_k}}.
    \end{math}
    These assumptions on $\TreeT$ lead to the fact that $\TreeT$ is covered by two different
    trees, respectively obtained by replacing $\TreeR'_j$ (resp.  $\TreeR'_\ell$) by any
    tree covered by $\TreeR'_j$ (resp.  $\TreeR'_\ell$). All this establishes the statement
    of the lemma.
\end{proof}
\medbreak

\begin{Proposition} \label{prop:shadow_intervals}
    For any finite alphabet $\GeneratingSet$ and any $\GeneratingSet$-trees $\TreeS$,
    $\TreeT$ such that $\TreeS \OrderPrefixes \TreeT$, the interval $[\TreeS, \TreeT]$ of
    $\Par{\SyntaxTreesInternalNode(\GeneratingSet), \OrderPrefixes}$ is isomorphic as a
    lattice to $\JLattice{\ShadowPoset{s}}$ where
    $s := \Shadow\Par{\lozenge_{|\TreeS|} \BPar{\TreeT \setminus \TreeS}}$.
\end{Proposition}
\begin{proof}
    By Proposition~\ref{prop:prefix_tree_poset_interval_decomposition}, the interval
    $\Han{\TreeS, \TreeT}$ is isomorphic as a poset to the interval
    \begin{math}
        \Han{\Corolla\Par{\lozenge_{|\TreeS|}},
        \lozenge_{|\TreeS|} \BPar{\TreeR_1, \dots, \TreeR_{|\TreeS|}}}
    \end{math}
    where $\Par{\TreeR_1, \dots, \TreeR_{|\TreeS|}}$ is the $\GeneratingSet$-forest $\TreeT
    \setminus \TreeS$. Therefore, the statement of the proposition is equivalent to saying
    that
    \begin{equation}
        \Han{\Corolla\Par{\lozenge_{|\TreeS|}},
        \lozenge_{|\TreeS|} \BPar{\TreeR_1, \dots, \TreeR_{|\TreeS|}}}
        \simeq
        \JLattice{\ShadowPoset{\Shadow\Par{\lozenge_{|\TreeS|}
        \BPar{\TreeR_1, \dots, \TreeR_{|\TreeS|}}}}}.
    \end{equation}
    By Lemmas~\ref{lem:join_irreducible_shadow_lattices}
    and~\ref{lem:join_irreducible_intervals}, the respective sets of join-irreducible
    elements of the two lattices
    \begin{math}
        \JLattice{\ShadowPoset{\Shadow\Par{\lozenge_{|\TreeS|}
        \BPar{\TreeR_1, \dots, \TreeR_{|\TreeS|}}}}}
    \end{math}
    and
    \begin{math}
        \Han{\Corolla\Par{\lozenge_{|\TreeS|}},
        \lozenge_{|\TreeS|} \BPar{\TreeR_1, \dots, \TreeR_{|\TreeS|}}}
    \end{math}
    are in one-to-one correspondence. They also preserves the ordering so that these two
    subposets are isomorphic as posets. Now, since by
    Proposition~\ref{prop:prefix_tree_poset_meet_semilattice},
    \begin{math}
        \Han{\Corolla\Par{\lozenge_{|\TreeS|}},
        \lozenge_{|\TreeS|} \BPar{\TreeR_1, \dots, \TreeR_{|\TreeS|}}}
    \end{math}
    is a finite distributive lattice, by the Fundamental theorem for finite distributive
    lattices, the statement of the proposition follows.
\end{proof}
\medbreak

\begin{Theorem} \label{thm:prefix_tree_poset_interval_isomorphism}
    For any finite alphabet $\GeneratingSet$ and any $\GeneratingSet$-trees $\TreeS$,
    $\TreeT$, $\TreeS'$, and $\TreeT'$ such that $\TreeS \OrderPrefixes \TreeT$ and $\TreeS'
    \OrderPrefixes \TreeT'$, the intervals $[\TreeS, \TreeT]$ and $\Han{\TreeS', \TreeT'}$
    of $\Par{\SyntaxTreesInternalNode(\GeneratingSet), \OrderPrefixes}$ are isomorphic as
    posets if and only if
    \begin{equation}
        \Shadow\Par{\lozenge_{|\TreeS|} \BPar{\TreeT \setminus \TreeS}}
        = \Shadow\Par{\lozenge_{\Brr{\TreeS'}} \BPar{\TreeT' \setminus \TreeS'}}.
    \end{equation}
\end{Theorem}
\begin{proof}
    Proposition~\ref{prop:shadow_intervals} brings a one-to-one correspondence between
    shadows and intervals of $\Par{\SyntaxTreesInternalNode(\GeneratingSet),
    \OrderPrefixes}$ up to lattice isomorphism. This is equivalent to the statement of the
    theorem.
\end{proof}
\medbreak

For instance, by considering the same alphabet $\GeneratingSet$ as in the previous example,
Theorem~\ref{thm:prefix_tree_poset_interval_isomorphism} says that since
\begin{equation}
    \Shadow\Par{
    \lozenge_3
    \BPar{
    \begin{tikzpicture}[Centering,xscale=0.2,yscale=0.17]
        \node(1)at(0.00,-7.50){};
        \node(3)at(2.00,-7.50){};
        \node(5)at(4.00,-7.50){};
        \node(7)at(5.00,-2.50){};
        \node(9)at(6.00,-5.00){};
        \node[NodeST](0)at(0.00,-5.00){$\GenE$};
        \node[NodeST](2)at(1.00,-2.50){$\GenA$};
        \node[NodeST](4)at(3.00,-5.00){$\GenA$};
        \node[NodeST](6)at(5.00,0.00){$\GenC$};
        \node[NodeST](8)at(6.00,-2.50){$\GenE$};
        \draw[Edge](0)--(2);
        \draw[Edge](1)--(0);
        \draw[Edge](2)--(6);
        \draw[Edge](3)--(4);
        \draw[Edge](4)--(2);
        \draw[Edge](5)--(4);
        \draw[Edge](7)--(6);
        \draw[Edge](8)--(6);
        \draw[Edge](9)--(8);
        \node(r)at(5.00,2.25){};
        \draw[Edge](r)--(6);
    \end{tikzpicture}
    \setminus
    \begin{tikzpicture}[Centering,xscale=0.24,yscale=0.24]
        \node(0)at(0.00,-1.67){};
        \node(2)at(1.00,-1.67){};
        \node(4)at(2.00,-3.33){};
        \node[NodeST](1)at(1.00,0.00){$\GenC$};
        \node[NodeST](3)at(2.00,-1.67){$\GenE$};
        \draw[Edge](0)--(1);
        \draw[Edge](2)--(1);
        \draw[Edge](3)--(1);
        \draw[Edge](4)--(3);
        \node(r)at(1.00,1.75){};
        \draw[Edge](r)--(1);
    \end{tikzpicture}}}
    =
    \begin{tikzpicture}[Centering,xscale=0.2,yscale=0.19]
        \node[Node](1)at(0.00,-2.67){};
        \node[Node](3)at(2.00,-2.67){};
        \node[Node](0)at(1.00,0.00){};
        \node[Node](2)at(1.00,-1.33){};
        \draw[Edge](1)--(2);
        \draw[Edge](2)--(0);
        \draw[Edge](3)--(2);
        \node(r)at(1.00,1.75){};
        \draw[Edge](r)--(0);
    \end{tikzpicture}
    =
    \Shadow\Par{
    \lozenge_2
    \BPar{
    \begin{tikzpicture}[Centering,xscale=0.2,yscale=0.19]
        \node(0)at(0.00,-2.25){};
        \node(3)at(2.00,-6.75){};
        \node(5)at(4.00,-6.75){};
        \node(7)at(5.00,-6.75){};
        \node(8)at(6.00,-6.75){};
        \node[NodeST](1)at(1.00,0.00){$\GenA$};
        \node[NodeST](2)at(2.00,-4.50){$\GenE$};
        \node[NodeST](4)at(3.00,-2.25){$\GenA$};
        \node[NodeST](6)at(5.00,-4.50){$\GenC$};
        \draw[Edge](0)--(1);
        \draw[Edge](2)--(4);
        \draw[Edge](3)--(2);
        \draw[Edge](4)--(1);
        \draw[Edge](5)--(6);
        \draw[Edge](6)--(4);
        \draw[Edge](7)--(6);
        \draw[Edge](8)--(6);
        \node(r)at(1.00,2){};
        \draw[Edge](r)--(1);
    \end{tikzpicture}
    \setminus
    \begin{tikzpicture}[Centering,xscale=0.2,yscale=0.27]
        \node(0)at(0.00,-1.50){};
        \node(2)at(2.00,-1.50){};
        \node[NodeST](1)at(1.00,0.00){$\GenA$};
        \draw[Edge](0)--(1);
        \draw[Edge](2)--(1);
        \node(r)at(1.00,1.5){};
        \draw[Edge](r)--(1);
    \end{tikzpicture}}},
\end{equation}
one has the isomorphism
\begin{equation}
    \Han{
    \begin{tikzpicture}[Centering,xscale=0.24,yscale=0.24]
        \node(0)at(0.00,-1.67){};
        \node(2)at(1.00,-1.67){};
        \node(4)at(2.00,-3.33){};
        \node[NodeST](1)at(1.00,0.00){$\GenC$};
        \node[NodeST](3)at(2.00,-1.67){$\GenE$};
        \draw[Edge](0)--(1);
        \draw[Edge](2)--(1);
        \draw[Edge](3)--(1);
        \draw[Edge](4)--(3);
        \node(r)at(1.00,1.75){};
        \draw[Edge](r)--(1);
    \end{tikzpicture},
    \begin{tikzpicture}[Centering,xscale=0.2,yscale=0.17]
        \node(1)at(0.00,-7.50){};
        \node(3)at(2.00,-7.50){};
        \node(5)at(4.00,-7.50){};
        \node(7)at(5.00,-2.50){};
        \node(9)at(6.00,-5.00){};
        \node[NodeST](0)at(0.00,-5.00){$\GenE$};
        \node[NodeST](2)at(1.00,-2.50){$\GenA$};
        \node[NodeST](4)at(3.00,-5.00){$\GenA$};
        \node[NodeST](6)at(5.00,0.00){$\GenC$};
        \node[NodeST](8)at(6.00,-2.50){$\GenE$};
        \draw[Edge](0)--(2);
        \draw[Edge](1)--(0);
        \draw[Edge](2)--(6);
        \draw[Edge](3)--(4);
        \draw[Edge](4)--(2);
        \draw[Edge](5)--(4);
        \draw[Edge](7)--(6);
        \draw[Edge](8)--(6);
        \draw[Edge](9)--(8);
        \node(r)at(5.00,2.25){};
        \draw[Edge](r)--(6);
    \end{tikzpicture}}
    \simeq
    \Han{
    \begin{tikzpicture}[Centering,xscale=0.2,yscale=0.27]
        \node(0)at(0.00,-1.50){};
        \node(2)at(2.00,-1.50){};
        \node[NodeST](1)at(1.00,0.00){$\GenA$};
        \draw[Edge](0)--(1);
        \draw[Edge](2)--(1);
        \node(r)at(1.00,1.5){};
        \draw[Edge](r)--(1);
    \end{tikzpicture},
    \begin{tikzpicture}[Centering,xscale=0.2,yscale=0.19]
        \node(0)at(0.00,-2.25){};
        \node(3)at(2.00,-6.75){};
        \node(5)at(4.00,-6.75){};
        \node(7)at(5.00,-6.75){};
        \node(8)at(6.00,-6.75){};
        \node[NodeST](1)at(1.00,0.00){$\GenA$};
        \node[NodeST](2)at(2.00,-4.50){$\GenE$};
        \node[NodeST](4)at(3.00,-2.25){$\GenA$};
        \node[NodeST](6)at(5.00,-4.50){$\GenC$};
        \draw[Edge](0)--(1);
        \draw[Edge](2)--(4);
        \draw[Edge](3)--(2);
        \draw[Edge](4)--(1);
        \draw[Edge](5)--(6);
        \draw[Edge](6)--(4);
        \draw[Edge](7)--(6);
        \draw[Edge](8)--(6);
        \node(r)at(1.00,2){};
        \draw[Edge](r)--(1);
    \end{tikzpicture}}
\end{equation}
between these two intervals of $\Par{\SyntaxTreesInternalNode(\GeneratingSet),
\OrderPrefixes}$.
\medbreak

%%%%%%%%%%%%%%%%%%%%%%%%%%%%%%%%%%%%%%%%%%%%%%%%%%%%%%%%%%%%%%%%%%%%%%%%%%%%%%%%%%%%%%%%%%%%
%%%%%%%%%%%%%%%%%%%%%%%%%%%%%%%%%%%%%%%%%%%%%%%%%%%%%%%%%%%%%%%%%%%%%%%%%%%%%%%%%%%%%%%%%%%%
\subsection{Enumerative properties}
We end the study of the $\GeneratingSet$-prefix posets by describing a way to count the
elements of a given interval and by enumerating all its intervals with respect to the
degrees of their minimal and maximal elements.
\medbreak

%%%%%%%%%%%%%%%%%%%%%%%%%%%%%%%%%%%%%%%%%%%%%%%%%%%%%%%%%%%%%%%%%%%%%%%%%%%%%%%%%%%%%%%%%%%%
\subsubsection{Cardinalities of intervals}
The \Def{load} $\Load(s)$ of a shadow $s := \lbag s_1, \dots, s_k \rbag$ is the integer
$\Load(s)$ recursively defined by
\begin{equation}
    \Load\Par{\lbag s_1, \dots, s_k \rbag} := \prod_{i \in [k]} \Par{1 + \Load\Par{s_i}}.
\end{equation}
For instance, the load of the shadow $s$ appearing in~\eqref{equ:example_shadow} is~$20$.
Indeed, by labeling each node $u$ of $s$ by the load of the subtree of $s$ rooted at $u$, we
have
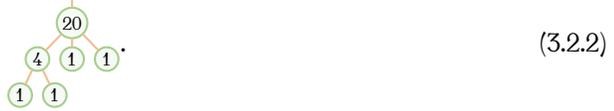
\begin{equation}
    \begin{tikzpicture}[Centering,xscale=0.23,yscale=0.24]
        \node[Node](0)at(0.00,-4.00){$1$};
        \node[Node](2)at(2.00,-4.00){$1$};
        \node[Node](4)at(3.00,-2.00){$1$};
        \node[Node](5)at(5.00,-2.00){$1$};
        \node[Node](1)at(1.00,-2.00){$4$};
        \node[Node](3)at(3.00,0.00){$20$};
        \draw[Edge](0)--(1);
        \draw[Edge](1)--(3);
        \draw[Edge](2)--(1);
        \draw[Edge](4)--(3);
        \draw[Edge](5)--(3);
        \node(r)at(3,1.75){};
        \draw[Edge](r)--(3);
    \end{tikzpicture}.
\end{equation}
\medbreak

\begin{Proposition} \label{prop:prefix_tree_poset_interval_cardinality}
    For any finite alphabet $\GeneratingSet$ and any $\GeneratingSet$-trees $\TreeS$ and
    $\TreeT$ such that $\TreeS \OrderPrefixes \TreeT$, in
    $\Par{\SyntaxTreesInternalNode(\GeneratingSet), \OrderPrefixes}$,
    \begin{equation} \label{equ:prefix_tree_poset_interval_cardinality}
        \# \Han{\TreeS, \TreeT}
        = \Load\Par{\Shadow\Par{\lozenge_{|\TreeS|} \BPar{\TreeT \setminus \TreeS}}}.
    \end{equation}
\end{Proposition}
\begin{proof}
    Let $\theta(\TreeT) := \# \Han{\Leaf, \TreeT}$. By
    Proposition~\ref{prop:prefix_tree_poset}, $\theta(\TreeT)$ is the number of prefixes of
    $\TreeT$. By Lemma~\ref{lem:prefix_recursive}, any prefix $\TreeR$ of $\TreeT$ is either
    the leaf, or when $\TreeT$ is not the leaf, the roots of $\TreeR$ and $\TreeT$ have the
    same label $\GenA \in \GeneratingSet$ and each $\TreeR(i)$ is a prefix of $\TreeT(i)$
    for all $i \in [|\GenA|]$. Hence,
    \begin{equation}
        \theta(\TreeT) = 1 + \prod_{i \in [|\GenA|]} \theta\Par{\TreeT(i)}.
    \end{equation}
    Moreover, by definition of the maps $\Shadow$ and $\Load$, we have
    \begin{equation}
        \Load\Par{\Shadow\Par{\lozenge_1 \BPar{\TreeT \setminus \Leaf}}}
        = \Load\Par{\Shadow\Par{\lozenge_1 \circ_1 \TreeT}}
        = \theta(\TreeT).
    \end{equation}
    Finally, by Proposition~\ref{prop:prefix_tree_poset_interval_decomposition},
    Equation~\eqref{equ:prefix_tree_poset_interval_cardinality} is a consequence of the fact
    the cardinality of $\Han{\TreeS, \TreeT}$ is the product of the cardinalities of
    $\Han{\Leaf, \TreeR_1}$, \dots, $\Han{\Leaf, \TreeR_{|\TreeS|}}$ where $\Par{\TreeR_1,
    \dots, \TreeR_{|\TreeS|}}$ is the forest~$\TreeT \setminus \TreeS$.
\end{proof}
\medbreak

%%%%%%%%%%%%%%%%%%%%%%%%%%%%%%%%%%%%%%%%%%%%%%%%%%%%%%%%%%%%%%%%%%%%%%%%%%%%%%%%%%%%%%%%%%%%
\subsubsection{Generating series of the intervals}
Let us consider now the generating series
\begin{equation}
    \IntervalSeries_{\SyntaxTreesInternalNode(\GeneratingSet)}(q, t) :=
    \sum_{\substack{
        \TreeS, \TreeT \in \SyntaxTreesInternalNode(\GeneratingSet) \\
        \TreeS \OrderPrefixes \TreeT
    }}
    q^{\Deg(\TreeS)} t^{\Deg(\TreeT)}
\end{equation}
enumerating all intervals $\Han{\TreeS, \TreeT}$ of
$\Par{\SyntaxTreesInternalNode(\GeneratingSet), \OrderPrefixes}$ with respect to the degree
of $\TreeS$ (parameter $q$) and the degree of $\TreeT$ (parameter $t$).
\medbreak

\begin{Proposition} \label{prop:prefix_tree_poset_interval_enumeration}
    For any finite alphabet $\GeneratingSet$, the series
    $\IntervalSeries_{\SyntaxTreesInternalNode(\GeneratingSet)}(q, t)$ satisfies
    \begin{equation}
        \IntervalSeries_{\SyntaxTreesInternalNode(\GeneratingSet)}(q, t) =
        1
        + t \; \RankSeries_\GeneratingSet\Par{
        \IntervalSeries_{\SyntaxTreesInternalNode(\GeneratingSet)}(q, t)
        - q t \; \RankSeries_\GeneratingSet\Par{
        \IntervalSeries_{\SyntaxTreesInternalNode(\GeneratingSet)}(q, t)
        }}
        + q t \; \RankSeries_\GeneratingSet\Par{
        \IntervalSeries_{\SyntaxTreesInternalNode(\GeneratingSet)}(q, t)
        }.
    \end{equation}
\end{Proposition}
\begin{proof}
    Let $\Han{\TreeS, \TreeT}$ be an interval of
    $\Par{\SyntaxTreesInternalNode(\GeneratingSet), \OrderPrefixes}$.  By
    Proposition~\ref{prop:prefix_tree_poset}, $\TreeS$ is a prefix of $\TreeT$.  Moreover,
    Proposition~\ref{prop:prefix_tree_poset_interval_decomposition} implies that this
    interval $\Han{\TreeS, \TreeT}$ can be encoded as the tree obtained by marking in
    $\TreeT$ the common internal nodes between $\TreeS$ and $\TreeT$. Since $\TreeS$ is a
    prefix of $\TreeT$, if a node different from the root is marked, its father is also
    marked.  Now, let $F(q, t)$ and $G(q, t)$ be the two series satisfying
    \begin{subequations}
    \begin{equation}
        F(q, t) = G(q, t) + q t \sum_{\GenA \in \GeneratingSet} F(q, t)^{|\GenA|},
    \end{equation}
    \begin{equation}
        G(q, t) = 1 + t \sum_{\GenA \in \GeneratingSet} G(q, t)^{|\GenA|}.
    \end{equation}
    \end{subequations}
    The series $G(q, t)$ enumerates the $\GeneratingSet$-trees with respect to their degree
    by the parameter $t$, and due to the previous description of the encoding of intervals,
    $F(q, t)$ enumerates the marked trees with respect to their degree by the parameter $t$
    and their number of marked nodes by the parameter $q$.  Now, by bringing in play the
    series $\RankSeries_\GeneratingSet(t)$, these two series express as
    \begin{subequations}
    \begin{equation}
        F(q, t) = G(q, t) + q t \; \RankSeries_\GeneratingSet\Par{F(q, t)},
    \end{equation}
    \begin{equation}
        G(q, t) = 1 + t \; \RankSeries_\GeneratingSet\Par{G(q, t)}.
    \end{equation}
    \end{subequations}
    This implies that
    \begin{subequations}
    \begin{equation}
        F(q, t) =
        1 + t \; \RankSeries_\GeneratingSet\Par{G(q, t)}
        + q t \; \RankSeries_\GeneratingSet\Par{F(q, t)},
    \end{equation}
    \begin{equation}
        G(q, t) = F(q, t) - q t \; \RankSeries_\GeneratingSet\Par{F(q, t)},
    \end{equation}
    \end{subequations}
    and since by construction
    \begin{math}
        \IntervalSeries_{\SyntaxTreesInternalNode(\GeneratingSet)}(q, t) = F(q, t),
    \end{math}
    the stated relation for $\IntervalSeries_{\SyntaxTreesInternalNode(\GeneratingSet)}(q,
    t)$ follows.
\end{proof}
\medbreak

For instance, when $\GeneratingSet$ consists in one binary letter,
$\RankSeries_\GeneratingSet(t) = t^2$ and
$\IntervalSeries_{\SyntaxTreesInternalNode(\GeneratingSet)}(q, t)$ satisfies
\begin{equation}
    1 - \IntervalSeries_{\SyntaxTreesInternalNode(\GeneratingSet)}(q, t)
    + t \IntervalSeries_{\SyntaxTreesInternalNode(\GeneratingSet)}(q, t)^2
    + q t \IntervalSeries_{\SyntaxTreesInternalNode(\GeneratingSet)}(q, t)^2
    - 2 q t^2 \IntervalSeries_{\SyntaxTreesInternalNode(\GeneratingSet)}(q, t)^3
    + q^2 t^3 \IntervalSeries_{\SyntaxTreesInternalNode(\GeneratingSet)}(q, t)^4
    = 0
\end{equation}
and
\begin{equation} \begin{split}
    \IntervalSeries_{\SyntaxTreesInternalNode(\GeneratingSet)}(q, t)
    & = 1 + (1 + q)t + 2\Par{1 + q + q^2}t^2 + \Par{5 + 6q + 5q^2 + 5q^3}t^3
    \\
    & \quad + 2\Par{7 + 10q + 9q^2 + 7q^3 + 7q^4}t^4
    \\
    & \quad + 14\Par{3 + 5q + 5q^2 + 4q^3 + 3q^4 + 3q^5}t^5 + \cdots.
\end{split} \end{equation}
The coefficients of $\IntervalSeries_{\SyntaxTreesInternalNode(\GeneratingSet)}(1, t)$ are
\begin{equation}
    1, 2, 6, 21, 80, 322, 1348, 5814
\end{equation}
and form Sequence~\OEIS{A121988} of~\cite{Slo}, enumerating the vertices of the
multiplihedra, which form a sequence of posets~\cite{Sta70}.
\medbreak

%%%%%%%%%%%%%%%%%%%%%%%%%%%%%%%%%%%%%%%%%%%%%%%%%%%%%%%%%%%%%%%%%%%%%%%%%%%%%%%%%%%%%%%%%%%%
%%%%%%%%%%%%%%%%%%%%%%%%%%%%%%%%%%%%%%%%%%%%%%%%%%%%%%%%%%%%%%%%%%%%%%%%%%%%%%%%%%%%%%%%%%%%
%%%%%%%%%%%%%%%%%%%%%%%%%%%%%%%%%%%%%%%%%%%%%%%%%%%%%%%%%%%%%%%%%%%%%%%%%%%%%%%%%%%%%%%%%%%%
\section{Graded graphs from operads} \label{sec:operad_prefix_graphs}
The aim of this section is to extend the previous definitions of $\GeneratingSet$-prefix
graded graphs and $\GeneratingSet$-twisted prefix graded graphs so that vertices of the
graphs can be any combinatorial objects endowed with the structure of an operad subjected to
some conditions.  This generalization, applied on free operads ---that are operad of
$\GeneratingSet$-trees endowed with the partial composition of trees--- gives back the
previous graded graphs. As we shall see, the pairs of graded graphs thus obtained are not
always $\phi$-diagonal dual.  We end this section by presenting some examples of such pairs
of graded graphs.
\medbreak

%%%%%%%%%%%%%%%%%%%%%%%%%%%%%%%%%%%%%%%%%%%%%%%%%%%%%%%%%%%%%%%%%%%%%%%%%%%%%%%%%%%%%%%%%%%%
%%%%%%%%%%%%%%%%%%%%%%%%%%%%%%%%%%%%%%%%%%%%%%%%%%%%%%%%%%%%%%%%%%%%%%%%%%%%%%%%%%%%%%%%%%%%
\subsection{Prefix and twisted prefix graded graphs}
We start by introducing the notion of homogeneous and finitely generated operad. We then
describe the construction of a pair of graded graphs from any homogeneous and finitely
generated operad.
\medbreak

%%%%%%%%%%%%%%%%%%%%%%%%%%%%%%%%%%%%%%%%%%%%%%%%%%%%%%%%%%%%%%%%%%%%%%%%%%%%%%%%%%%%%%%%%%%%
\subsubsection{Homogeneous operads and degrees}
Let us begin by an elementary result on the minimal generating sets of operads satisfying
some conditions.
\medbreak

\begin{Lemma} \label{lem:unique_minimal_generating_set}
    If $\Operad$ is an operad such that $\Operad(0) = \emptyset$ and $\Operad(1) =
    \Bra{\Unit}$, then $\Operad$ admits a unique minimal generating set.
\end{Lemma}
\begin{proof}
    This follows from the following algorithm to compute a minimal set $\GeneratingSet$ of
    generators of $\Operad$  up to a given arity. First, since $\Operad(0) = \emptyset$ and
    $\Operad(1) = \Bra{\Unit}$, then $\GeneratingSet(1) = \emptyset$.  Now, assume that
    there is an $m \geq 1$ such that we know the sets $\GeneratingSet(n)$ for all
    $n \in [m]$. A candidate for $\GeneratingSet(m + 1)$ is the set
    \begin{math}
        \Operad(m + 1) \setminus \Operad^{\GeneratingSet'}
    \end{math}
    where $\GeneratingSet'$ is the graded set consisting exactly in the elements of
    $\GeneratingSet$ up to arity $m$. In other terms, this candidate for $\GeneratingSet(m +
    1)$ contains the elements of arity $m + 1$ of $\Operad$ which cannot be obtained by
    composing elements of arity $k \leq m$ of $\GeneratingSet$.  The fact that
    $\GeneratingSet(1) = \emptyset$ ensures that the arity of any partial composition $x
    \circ_i y$ where $x \in \Operad$ and $y \in \GeneratingSet$ (resp.  $x \in
    \GeneratingSet$ and $y \in \Operad$) is greater than the arity of $x$ (resp. $y$). For
    this reason, the set $\GeneratingSet(m + 1)$ is unique, so that the given candidate for
    this set is the only possible one.  Finally, since by construction $\GeneratingSet$ is
    minimal, the statement of lemma follows.
\end{proof}
\medbreak

By Lemma~\ref{lem:unique_minimal_generating_set}, we shall denote by
$\GeneratingSet_\Operad$ the unique minimal generating set of any operad $\Operad$
satisfying $\Operad(0) = \emptyset$ and $\Operad(1) = \Bra{\Unit}$.  Moreover, given the
alphabet $\GeneratingSet_\Operad$, there is a unique operad congruence $\Congr_\Operad$ such
that
\begin{math}
    \SyntaxTreesLeaf\Par{\GeneratingSet_\Operad}/_{\Congr_\Operad} \simeq \Operad.
\end{math}
Indeed, $\Congr_\Operad$ satisfies necessarily $\TreeT \Congr_\Operad \TreeT'$ for all
$\TreeT, \TreeT' \in \SyntaxTreesLeaf\Par{\GeneratingSet_\Operad}$ such that $\Eval(\TreeT)
= \Eval\Par{\TreeT'}$. For this reason, $\Operad$ admits $\Par{\GeneratingSet_\Operad,
\Congr_\Operad}$ as unique presentation. These properties arising from the fact that
$\Operad(0) = \emptyset$ and $\Operad(1) = \Bra{\Unit}$ are consequences of the previous
lemma, used implicitly in the sequel.
\medbreak

Let $\Operad$ be an operad such that $\Operad(0) = \emptyset$ and $\Operad(1) =
\Bra{\Unit}$. If the presentation $\Par{\GeneratingSet_\Operad, \Congr_\Operad}$ of
$\Operad$ is such that for any $\TreeT, \TreeT' \in
\SyntaxTreesLeaf\Par{\GeneratingSet_\Operad}$, $\TreeT \Congr_\Operad \TreeT'$ implies
$\Deg(\TreeT) = \Deg\Par{\TreeT'}$, then $\Operad$ is \Def{homogeneous}. Besides, if
$\GeneratingSet_\Operad$ is finite, then $\Operad$ is \Def{finitely generated}. In the
sequel, we shall mainly consider homogeneous and finitely generated operads.
\medbreak

Given an homogeneous operad $\Operad$, the \Def{degree} $\Deg(x)$ of $x \in \Operad$ is the
degree of a treelike expression of $x$ on $\GeneratingSet_\Operad$. Observe that since
$\Operad$ is homogeneous, if $x$ admits two treelike expressions $\TreeT$ and $\TreeT'$, we
necessarily have $\Deg(\TreeT) = \Deg\Par{\TreeT'}$ so that $\Deg(x)$ is well-defined.
Moreover, we denote by $\OperadDegree$ the graded set wherein for any $d \geq 0$,
$\OperadDegree(d)$ is the set of all elements of $\Operad$ having $d$ as degree. Remark that
if $\Operad$ is finitely generated, since $\Operad$ is by definition a quotient of
$\SyntaxTreesLeaf\Par{\GeneratingSet_\Operad}$, and since
$\SyntaxTreesInternalNode\Par{\GeneratingSet_\Operad}$ is combinatorial, $\OperadDegree$ is
also combinatorial.
\medbreak

%%%%%%%%%%%%%%%%%%%%%%%%%%%%%%%%%%%%%%%%%%%%%%%%%%%%%%%%%%%%%%%%%%%%%%%%%%%%%%%%%%%%%%%%%%%%
\subsubsection{Prefix graded graphs}
For any homogeneous and finitely generated operad $\Operad$, let $\Par{\OperadDegree, \Up}$
be the graded graph wherein, for any $x \in \OperadDegree$,
\begin{equation} \label{equ:up_operads}
    \Up(x) :=
    \sum_{\substack{
        \GenA \in \GeneratingSet_\Operad \\
        i \in [|x|]}}
    x \circ_i \GenA.
\end{equation}
In words, this says that $y \in \OperadDegree$ appears in $\Up(x)$ with a coefficient
$\lambda$ where $\lambda$ is the number of pairs $\Par{\GenA, i} \in \GeneratingSet_\Operad
\times \N$ such that $y = x \circ_i \GenA$.  We call $\Par{\OperadDegree, \Up}$ the
\Def{prefix graph} of~$\Operad$. As a side remark, this graph is the derivation graph of the
so-called monochrome bud generating system of $\Operad$ introduced in~\cite{Gir19}.
\medbreak

\begin{Proposition} \label{prop:prefix_graph_operad}
    Let $\Operad$ be an homogeneous and finitely generated operad. The prefix graph of
    $\Operad$ is a natural rooted graded graph.
\end{Proposition}
\begin{proof}
    Since $\Operad$ is finitely generated, for any $x \in \OperadDegree$, $\Up(x)$ is an
    $\OperadDegree$-polynomial, so that $\Up$ is a well-defined map from $\K
    \Angle{\OperadDegree}$ to $\K \Angle{\OperadDegree}$.  Moreover, since $\Operad$ is
    homogeneous, each element of $\Operad$ has a well-defined degree. For any $x \in
    \OperadDegree$, the degree of $x \circ_i \GenA$ where $i \in [|x|]$ and $\GenA \in
    \GeneratingSet_\Operad$ is $\Deg(x) + 1$. Therefore, $\Up$ sends any element of
    $\OperadDegree$ of degree $d \geq 0$ to a sum of elements of degrees $d + 1$. This shows
    that $\Par{\OperadDegree, \Up}$ is a graded graph.  Moreover, since
    by~\eqref{equ:operad_axiom_1}, any $x \in \Operad$ writes as
    \begin{equation}
        x = \Par{\dots \Par{\Par{\Unit \circ_{i_1} \GenA_1}
        \circ_{i_2} \GenA_2} \dots} \circ_{i_d} \GenA_d
    \end{equation}
    where $d \geq 0$, $\GenA_1, \dots, \GenA_d \in \GeneratingSet_\Operad$, and $i_1, i_2,
    \dots, i_d \in \N$, there is at least a path from $\Unit$ to $x$ in $\Par{\OperadDegree,
    \Up}$.  Hence, this graded graph admits $\Leaf$ as root.  Finally, since all
    coefficients of $\Up(x)$ are obviously nonnegative, the statement of the proposition is
    established.
\end{proof}
\medbreak

Observe that when $\Operad$ is free, then $\Operad \simeq
\SyntaxTreesLeaf\Par{\GeneratingSet_\Operad}$, and $\Par{\OperadDegree, \Up}$ and
$\Par{\SyntaxTreesInternalNode\Par{\GeneratingSet_\Operad}, \Up}$ coincide as graded graphs.
\medbreak

%%%%%%%%%%%%%%%%%%%%%%%%%%%%%%%%%%%%%%%%%%%%%%%%%%%%%%%%%%%%%%%%%%%%%%%%%%%%%%%%%%%%%%%%%%%%
\subsubsection{Twisted prefix graded graphs}
For any homogeneous and finitely generated operad $\Operad$, let $\Par{\OperadDegree, \Vp}$
be the graded graph wherein, for any $x \in \OperadDegree$,
\begin{equation}
    \Vp(x) := \CharacteristicSeries{\Support{\Vp'(x)}},
\end{equation}
where $\Vp' : \K \Angle{\OperadDegree} \to \K \Angle{\OperadDegree}$ is the linear map
defined recursively by
\begin{equation} \label{equ:twisted_prefix_operads}
    \Vp'(x) := 
    \Par{\sum_{\GenA \in \GeneratingSet_\Operad} \GenA \circ_1 x}
    +
    \Par{
    \sum_{\substack{
        \GenB \in \GeneratingSet_\Operad \\
        y_1, \dots, y_{|\GenB|} \in \Operad \\
        x = \GenB \circ \Han{y_1, \dots, y_{|\GenB|}}
    }}
    \sum_{j \in \Han{2, |\GenB|}}
    \GenB \circ \Han{y_1, \dots, y_{j - 1}, \Vp'\Par{y_j}, y_{j + 1}, \dots, y_{|\GenB|}}}.
\end{equation}
In words, this says that $y \in \OperadDegree$ appears in $\Vp(x)$ if there is a treelike
expression $\TreeT_y$ on $\GeneratingSet_\Operad$ of $y$ and a treelike expression
$\TreeT_x$ on $\GeneratingSet_\Operad$ of $x$ such that $\TreeT_y$ appears in
$\Vp\Par{\TreeT_x}$, where this last occurrence of $\Vp$ is the linear map of the
$\GeneratingSet_\Operad$-twisted prefix graph
$\Par{\SyntaxTreesInternalNode\Par{\GeneratingSet_\Operad}, \Vp}$ (see
Section~\ref{subsubsec:twisted_prefix_graphs}).  We call $\Par{\OperadDegree, \Vp}$ the
\Def{twisted prefix graph} of~$\Operad$.
\medbreak

\begin{Proposition} \label{prop:twisted_prefix_graph_operad}
    Let $\Operad$ be an homogeneous and finitely generated operad. The twisted prefix graph
    of $\Operad$ is a simple rooted graded graph.
\end{Proposition}
\begin{proof}
    Since $\Operad$ is finitely generated, $\GeneratingSet_\Operad$ is finite and thus, the
    sum appearing in~\eqref{equ:twisted_prefix_operads} is finite. Therefore,
    $\Par{\SyntaxTreesInternalNode\Par{\GeneratingSet_\Operad}, \Vp}$ is a well-defined
    graded graph. Let $x, y \in \OperadDegree$ such that $y$ appears in $\Vp(x)$. Then,
    there are $\TreeS \in \TreeLikeExpressions_{\GeneratingSet_\Operad}(x)$, $\TreeT \in
    \TreeLikeExpressions_{\GeneratingSet_\Operad}(y)$ such that $\TreeT$ appears in
    $\Vp(\TreeS)$. Since $\GeneratingSet_\Operad$-twisted prefix graphs are ranked by the
    degrees of the $\GeneratingSet_\Operad$-trees, we have $\Deg(\TreeT) = \Deg(\TreeS) +
    1$. Therefore, and since $\Operad$ is homogeneous so that each element of
    $\OperadDegree$ as a well-defined degree, $\Deg(y) = \Deg(x) + 1$. Hence, $\Vp$ sends
    any element of $\OperadDegree$ of degree $d \geq 0$ to a sum of elements of degrees $d +
    1$. This implies that $\Par{\OperadDegree, \Vp}$ is a graded graph.  Moreover, as
    noticed in Section~\ref{subsubsec:twisted_prefix_graphs}, for each
    $\GeneratingSet_\Operad$-tree $\TreeT$ different from the leaf, $\Vp^\Dual(\TreeT) \ne
    0$. This implies that for any $\GeneratingSet_\Operad$-tree $\TreeS$ different from the
    leaf, there is a $\GeneratingSet_\Operad$-tree $\TreeT$ such that $\TreeS$ appears in
    $\Vp(\TreeT)$. For this reason, for any $y \in \OperadDegree$ such that $y \ne \Unit$,
    there is an $x \in \OperadDegree$ such that $\Vp(x) = y$. Therefore, and because $\Unit$
    is the only element of $\Operad$ of degree $0$, $\Par{\OperadDegree, \Vp}$ admits
    $\Unit$ as root.  Finally, directly by definition, all coefficients of $\Vp(x)$ are $0$
    or $1$.  This establishes the statement of the proposition.
\end{proof}
\medbreak

Observe that when $\Operad$ is free, $\Par{\OperadDegree, \Vp}$ and
$\Par{\SyntaxTreesInternalNode\Par{\GeneratingSet_\Operad}, \Vp}$ coincide as graded graphs.
Besides, when $\Operad$ is an homogeneous and finitely generated operad, by
Propositions~\ref{prop:prefix_graph_operad} and~\ref{prop:twisted_prefix_graph_operad}, the
two graded graphs $\Par{\OperadDegree, \Up}$ and $\Par{\OperadDegree, \Vp}$ are both ranked
by the degrees of their elements. Hence, $\Par{\OperadDegree, \Up, \Vp}$ is a pair of graded
graphs.
\medbreak

%%%%%%%%%%%%%%%%%%%%%%%%%%%%%%%%%%%%%%%%%%%%%%%%%%%%%%%%%%%%%%%%%%%%%%%%%%%%%%%%%%%%%%%%%%%%
%%%%%%%%%%%%%%%%%%%%%%%%%%%%%%%%%%%%%%%%%%%%%%%%%%%%%%%%%%%%%%%%%%%%%%%%%%%%%%%%%%%%%%%%%%%%
\subsection{Posets from operads} \label{subsec:operad_prefix_posets}
Let us study the posets of the prefix graphs of prefix graphs built from operads.
\medbreak

%%%%%%%%%%%%%%%%%%%%%%%%%%%%%%%%%%%%%%%%%%%%%%%%%%%%%%%%%%%%%%%%%%%%%%%%%%%%%%%%%%%%%%%%%%%%
\subsubsection{Prefix posets}
Let $\Operad$ be a homogeneous and finitely generated operad. The \Def{prefix poset} of
$\Operad$ is the poset $\Par{\OperadDegree, \OrderPrefixes}$ of $\Par{\OperadDegree, \Up}$.
Observe that $\GeneratingSet_\Operad$ is the set of the atoms of the prefix poset
of~$\Operad$.
\medbreak

\begin{Proposition} \label{prop:prefix_poset_operad_treelike_expressions}
    Let $\Operad$ be an homogeneous and finitely generated operad.  For any $x, y \in
    \OperadDegree$, we have $x \OrderPrefixes y$ if and only if there exist $\TreeS \in
    \TreeLikeExpressions_{\GeneratingSet_\Operad}(x)$ and $\TreeT \in
    \TreeLikeExpressions_{\GeneratingSet_\Operad}(y)$ such that $\TreeS \OrderPrefixes
    \TreeT$ in the $\GeneratingSet$-prefix poset
    $\Par{\SyntaxTreesInternalNode(\GeneratingSet), \OrderPrefixes}$.
\end{Proposition}
\begin{proof}
    By definition of the prefix poset of $\Operad$, $x \OrderPrefixes y$ is equivalent to
    the fact that there exist an integer $k \geq 0$, generators $\GenA_1$, \dots, $\GenA_k$
    of $\GeneratingSet_\Operad$, and positive integers $i_1$, \dots, $i_k$ such that
    \begin{equation} \label{equ:prefix_poset_operad_treelike_expressions}
        y = \Par{\dots \Par{\Par{x \circ_{i_1} \GenA_1} \circ_{i_2}
        \GenA_2} \dots } \circ_{i_k} \GenA_k.
    \end{equation}
    Let $\TreeS$ be any treelike expression on $\GeneratingSet_\Operad$ of $x$.
    From~\eqref{equ:prefix_poset_operad_treelike_expressions}, the tree
    \begin{equation}
        \TreeT := \Par{\dots \Par{\Par{\TreeS \circ_{i_1}
        \GenA_1} \circ_{i_2} \GenA_2} \dots } \circ_{i_k} \GenA_k,
    \end{equation}
    is a treelike expression on $\GeneratingSet_\Operad$ of $y$.  Moreover, by definition of
    the $\GeneratingSet_\Operad$-prefix graded graph, this is equivalent to the fact that
    there is a path from $\TreeS$ to $\TreeT$ in
    $\Par{\SyntaxTreesInternalNode\Par{\GeneratingSet_\Operad}, \Up}$.  Therefore,~$\TreeS
    \OrderPrefixes \TreeT$.
\end{proof}
\medbreak

For any $x, y \in \OperadDegree$, $x$ is a \Def{prefix} of $y$ if there exist some elements
$z_1$, \dots, $z_{|x|}$ of $\OperadDegree$ such that
\begin{math}
    y = x \circ \Han{z_1, \dots, z_{|x|}}.
\end{math}
\medbreak

\begin{Proposition} \label{prop:prefix_poset_operad}
    Let $\Operad$ be an homogeneous and finitely generated operad. The order relation
    $\OrderPrefixes$ of the prefix poset of $\Operad$ satisfies $x \OrderPrefixes y$ if and
    only if $x$ is a prefix of $y$ for any $x, y \in \OperadDegree$.  Moreover, the covering
    relation $\Covered_\Up$ of the prefix poset of $\Operad$ satisfies $x \Covered_\Up y$
    for any $x, y \in \Operad_\InternalNode$ if and only if there is an $\GenA \in
        \GeneratingSet_\Operad$ and an $i \in [|x|]$ such that $y = x \circ_i \GenA$.
\end{Proposition}
\begin{proof}
    By Proposition~\ref{prop:prefix_poset_operad_treelike_expressions}, we have $x
    \OrderPrefixes y$ if and only if $\TreeS \OrderPrefixes \TreeT$ where $\TreeS$ (resp.
    $\TreeT$) is a treelike expression on $\GeneratingSet_\Operad$ of $x$ (resp. $y$). By
    Proposition~\ref{prop:prefix_tree_poset}, this is equivalent to the fact that $\TreeS$
    is a prefix of $\TreeT$. Hence, we have
    \begin{math}
        \TreeT = \TreeS \circ \Han{\TreeR_1, \dots, \TreeR_{|\TreeS|}}
    \end{math}
    for some $\Operad_\GeneratingSet$-trees $\TreeR_1$, \dots, $\TreeR_{|\TreeS|}$. This
        says that
    \begin{math}
        \Eval\Par{\TreeT} = \Eval\Par{\TreeS \circ \Han{\TreeR_1, \dots, \TreeR_{|\TreeS|}}}
    \end{math}
    and, since $\Eval$ is an operad morphism, that
    \begin{math}
        y = x \circ \Han{\Eval\Par{\TreeR_1}, \dots, \Eval\Par{\TreeR_{|\TreeS|}}}.
    \end{math}
    This is, by definition of the order relation $\OrderPrefixes$ on
    $\Operad_\InternalNode$, equivalent to the fact that~$x \OrderPrefixes y$.
    \smallbreak

    The second part of the statement is a direct consequence of the definition of the
    map~$\Up$.
\end{proof}
\medbreak

%%%%%%%%%%%%%%%%%%%%%%%%%%%%%%%%%%%%%%%%%%%%%%%%%%%%%%%%%%%%%%%%%%%%%%%%%%%%%%%%%%%%%%%%%%%%
\subsubsection{Functorial construction}
\begin{Theorem} \label{thm:functor_operads_to_posets}
    The construction sending any homogeneous and finitely generated operad $\Operad$ to its
    prefix poset $\Par{\OperadDegree, \OrderPrefixes}$ and sending any morphism $\psi :
    \Operad \to \Operad'$ of homogeneous and finitely generated operads $\Operad$ and
    $\Operad'$ to the same map between $\Operad_\InternalNode$ and $\Operad'_\InternalNode$,
    is a functor preserving injections and surjections from the category of homogeneous and
    finitely generated operads to the category of posets.
\end{Theorem}
\begin{proof}
    By Proposition~\ref{prop:prefix_graph_operad}, this construction produces from an
    homogeneous and finitely presented operad $\Operad$ a well-defined poset
    $\Par{\OperadDegree, \OrderPrefixes}$. It remains to prove that this construction sends
    operad morphisms to poset morphisms and preserves their injectivity and surjectivity.
    For this, let $\Operad$ and $\Operad'$ be two homogeneous and finitely presented
    operads, and $\psi : \Operad \to \Operad'$ be an operad morphism.  Let also $x, y \in
    \Operad_\InternalNode$ and assume that $x \OrderPrefixes y$. By
    Proposition~\ref{prop:prefix_poset_operad_treelike_expressions}, there are $\TreeS \in
    \TreeLikeExpressions_{\GeneratingSet_\Operad}(x)$ and $\TreeT \in
    \TreeLikeExpressions_{\GeneratingSet_\Operad}(y)$ such that $\TreeS \OrderPrefixes
    \TreeT$.  By the universality property of free operads (See
    Section~\ref{subsubsec:free_operads}), the map $\psi$ is entirely specified by the map
    $f : \GeneratingSet_\Operad \to \Operad'$ satisfying $f(\GenA) = \psi(\GenA)$ for all
    $\GenA \in \GeneratingSet_\Operad$.  Let
    \begin{math}
        \bar{f} : \GeneratingSet_\Operad \to \SyntaxTreesLeaf\Par{\GeneratingSet_{\Operad'}}
    \end{math}
    be a map sending any $\GenA \in \GeneratingSet_\Operad$ to a treelike expression on
    $\GeneratingSet_{\Operad'}$ of $f(\GenA)$. Let also $\TreeS'$ (resp. $\TreeT'$) be the
    $\GeneratingSet_{\Operad'}$-tree obtained by replacing each internal node $\GenA \in
    \GeneratingSet_\Operad$ of $\TreeS$ (resp. $\TreeT$) by $\bar{f}(\GenA)$.  By
    construction, we have $\Eval\Par{\TreeS'} = \psi(x)$ and $\Eval\Par{\TreeT'} = \psi(y)$.
    Moreover, since $\TreeS \OrderPrefixes \TreeT$, by
    Proposition~\ref{prop:prefix_tree_poset}, $\TreeS$ is a prefix of $\TreeT$.  This
    implies  by construction of $\TreeS'$ and $\TreeT'$ that $\TreeS' \OrderPrefixes
    \TreeT'$. Hence, by Proposition~\ref{prop:prefix_poset_operad_treelike_expressions}, we
    have $\psi(x) \OrderPrefixes \psi(y)$. Therefore, $\psi$ is a poset morphism.  Finally,
    injections and surjections are preserved since operad morphisms are sent to poset
    morphisms without any change.
\end{proof}
\medbreak

As a consequence of Theorem~\ref{thm:functor_operads_to_posets}, if $\Operad$ is an
homogeneous and finitely generated operad, then the operad surjection $\Eval :
\SyntaxTreesLeaf\Par{\GeneratingSet_\Operad} \to \Operad$ is a surjective poset morphism
from $\Par{\SyntaxTreesInternalNode\Par{\GeneratingSet_\Operad}, \OrderPrefixes}$ to
$\Par{\OperadDegree, \OrderPrefixes}$.
\medbreak

%%%%%%%%%%%%%%%%%%%%%%%%%%%%%%%%%%%%%%%%%%%%%%%%%%%%%%%%%%%%%%%%%%%%%%%%%%%%%%%%%%%%%%%%%%%%
%%%%%%%%%%%%%%%%%%%%%%%%%%%%%%%%%%%%%%%%%%%%%%%%%%%%%%%%%%%%%%%%%%%%%%%%%%%%%%%%%%%%%%%%%%%%
\subsection{Examples} \label{subsec:examples_graph_operads}
Before ending this paper, we consider here some examples of pairs of graded graphs
constructed from some homogeneous and finitely generated operads. Some of these pairs of
graded graphs are $\phi$-diagonal dual and some other not.  Most of the considered operads
arise in a combinatorial context.
\medbreak

%%%%%%%%%%%%%%%%%%%%%%%%%%%%%%%%%%%%%%%%%%%%%%%%%%%%%%%%%%%%%%%%%%%%%%%%%%%%%%%%%%%%%%%%%%%%
\subsubsection{Associative operad}
The \Def{associative operad} $\As$ is the operad wherein $\As(n) := \Bra{\Product_n}$ for
all $n \geq 1$, and $\Product_n \circ_i \Product_m := \Product_{n + m - 1}$ for all $n \geq
1$, $m \geq 1$, and $i \in [n]$. This operad admits the presentation
$\Par{\GeneratingSet_\As, \Congr_\As}$ where
\begin{equation}
    \GeneratingSet_\As := \Bra{\Product_2},
\end{equation}
and  $\Congr_\As$ is the smallest operad congruence of
$\SyntaxTreesLeaf\Par{\GeneratingSet_\As}$ satisfying
\begin{equation}
    \Corolla\Par{\Product_2} \circ_1 \Corolla\Par{\Product_2}
    \Congr_\As
    \Corolla\Par{\Product_2} \circ_2 \Corolla\Par{\Product_2}.
\end{equation}
Therefore, $\As$ is homogeneous and finitely presented.
\medbreak

The pair $\Par{\As_\InternalNode, \Up, \Vp}$ of graded graphs satisfies $\Up\Par{\Product_n}
= n\, \Product_{n + 1}$ and $\Vp\Par{\Product_n} = \Product_{n + 1}$ for any $n \geq 1$ (see
Figure~\ref{fig:operad_graphs_As}).
\begin{figure}[ht]
    \centering
    \subfloat[][The graph $\Par{\As_\InternalNode, \Up}$.]{
    \begin{minipage}[c]{.4\textwidth}
        \centering
        \begin{tikzpicture}[Centering,xscale=1,yscale=.8]
            \tikzset{every node}=[font=\scriptsize]
            \node(Unit)at(0,0){$\Product_1$};
            \node(a2)at(0,-1){$\Product_2$};
            \node(a3)at(0,-2){$\Product_3$};
            \node(a4)at(0,-3){$\Product_4$};
            \draw[EdgeGraph](Unit)--(a2);
            \draw[EdgeGraph](a2)edge[]node[EdgeLabel]{$2$}(a3);
            \draw[EdgeGraph](a3)edge[]node[EdgeLabel]{$3$}(a4);
            \draw[EdgeGraph,dotted,shorten >=5mm](a4)--(0,-4);
        \end{tikzpicture}
    \end{minipage}
    \label{subfig:operad_graph_As_U}}
    \hfill
    \subfloat[][The graph $\Par{\As_\InternalNode, \Vp}$.]{
    \begin{minipage}[c]{.4\textwidth}
        \centering
        \begin{tikzpicture}[Centering,xscale=1,yscale=.8]
            \tikzset{every node}=[font=\scriptsize]
            \node(Unit)at(0,0){$\Product_1$};
            \node(a2)at(0,-1){$\Product_2$};
            \node(a3)at(0,-2){$\Product_3$};
            \node(a4)at(0,-3){$\Product_4$};
            \draw[EdgeGraph](Unit)--(a2);
            \draw[EdgeGraph](a2)--(a3);
            \draw[EdgeGraph](a3)--(a4);
            \draw[EdgeGraph,dotted,shorten >=5mm](a4)--(0,-4);
        \end{tikzpicture}
    \end{minipage}
    \label{subfig:operad_graph_As_V}}
    \caption{The pair $\Par{\As_\InternalNode, \Up, \Vp}$ of graded graphs.}
    \label{fig:operad_graphs_As}
\end{figure}
This very elementary example of pair of graded graphs is dual and is known as the
\Def{chain} in~\cite{Fom94}. The hook series of $\Par{\As_\InternalNode, \Up}$ satisfies
\begin{equation}
    \HookSeries{\Up} = \sum_{n \geq 1} n! \Product_n.
\end{equation}
\medbreak

%%%%%%%%%%%%%%%%%%%%%%%%%%%%%%%%%%%%%%%%%%%%%%%%%%%%%%%%%%%%%%%%%%%%%%%%%%%%%%%%%%%%%%%%%%%%
\subsubsection{Diassociative operad}
The \Def{diassociative operad} $\Dias$ is the operad wherein $\Dias(n)$ is the set of all
words of length $n \geq 1$ on the alphabet $\{\mathtt{0}, \mathtt{1}\}$ having exactly one
occurrence of $\mathtt{0}$. The partial composition $u \circ_i v$ of two such words $u$ and
$v$ consists in replacing the $i$-th letter of $u$ by $v'$, where $v'$ is the word obtained
from $v$ by replacing all its letters $a$ by $\max \Bra{u_i, a}$. This operad has been
introduced in~\cite{Lod01} under a slightly different form (see also~\cite{Cha05,Gir16}).
This operad admits the presentation $\Par{\GeneratingSet_\Dias, \Congr_\Dias}$ where
\begin{equation}
    \GeneratingSet_\Dias = \Bra{\mathtt{01}, \mathtt{10}}
\end{equation}
and $\Congr_\Dias$ is the smallest operad congruence of
$\SyntaxTreesLeaf\Par{\GeneratingSet_\Dias}$ satisfying
\begin{subequations}
\begin{equation}
    \Corolla\Par{\mathtt 01} \circ_1 \Corolla\Par{\mathtt 01}
    \Congr_\Dias
    \Corolla\Par{\mathtt 01} \circ_2 \Corolla\Par{\mathtt 01}
    \Congr_\Dias
    \Corolla\Par{\mathtt 01} \circ_2 \Corolla\Par{\mathtt 10},
\end{equation}
\begin{equation}
    \Corolla\Par{\mathtt 01} \circ_1 \Corolla\Par{\mathtt 10}
    \Congr_\Dias
    \Corolla\Par{\mathtt 10} \circ_2 \Corolla\Par{\mathtt 01},
\end{equation}
\begin{equation}
    \Corolla\Par{\mathtt 10} \circ_1 \Corolla\Par{\mathtt 01}
    \Congr_\Dias
    \Corolla\Par{\mathtt 10} \circ_1 \Corolla\Par{\mathtt 10}
    \Congr_\Dias
    \Corolla\Par{\mathtt 10} \circ_2 \Corolla\Par{\mathtt 10}.
\end{equation}
\end{subequations}
Observe that these relations describe respectively the treelike expressions for the elements
$011$, $101$, and $110$ of $\Dias$. Therefore, $\Dias$ is homogeneous and finitely generated.
\medbreak

The pair of graded graphs $\Par{\Dias_\InternalNode, \Up, \Vp}$
satisfies
\begin{equation} \label{equ:up_dias}
    \Up\Par{\mathtt{1}^k \mathtt{0} \mathtt{1}^\ell} =
    (2k + 1) \, \mathtt{1}^{k + 1} \mathtt{0} \mathtt{1}^\ell
    + (2\ell + 1) \, \mathtt{1}^k \mathtt{0} \mathtt{1}^{\ell + 1},
\end{equation}
and
\begin{equation}
    \Vp\Par{\mathtt{1}^k \mathtt{0} \mathtt{1}^\ell} =
    \mathtt{1}^k \mathtt{0} \mathtt{1}^{\ell + 1}
    + \mathtt{1}^{k + 1 + \ell} \mathtt{0},
\end{equation}
for any $k, \ell \in \N$ (see Figure~\ref{fig:operad_graphs_Dias}).
\begin{figure}[ht]
    \centering
    \subfloat[][The graph $\Par{\Dias_\InternalNode, \Up}$ up to elements of degree~$4$.]{
    \begin{minipage}[c]{.45\textwidth}
        \centering
        \begin{tikzpicture}[Centering,xscale=.7,yscale=.7]
            \tikzset{every node}=[font=\scriptsize]
            \node(0)at(0,0){$\mathtt{0}$};
            \node(01)at(-1,-1){$\mathtt{01}$};
            \node(10)at(1,-1){$\mathtt{10}$};
            \node(011)at(-2,-2){$\mathtt{011}$};
            \node(101)at(0,-2){$\mathtt{101}$};
            \node(110)at(2,-2){$\mathtt{110}$};
            \node(0111)at(-3,-3){$\mathtt{0111}$};
            \node(1011)at(-1,-3){$\mathtt{1011}$};
            \node(1101)at(1,-3){$\mathtt{1101}$};
            \node(1110)at(3,-3){$\mathtt{1110}$};
            \node(01111)at(-4,-4){$\mathtt{01111}$};
            \node(10111)at(-2,-4){$\mathtt{10111}$};
            \node(11011)at(0,-4){$\mathtt{11011}$};
            \node(11101)at(2,-4){$\mathtt{11101}$};
            \node(11110)at(4,-4){$\mathtt{11110}$};
            \draw[EdgeGraph](0)--(01);
            \draw[EdgeGraph](0)--(10);
            \draw[EdgeGraph](01)edge[]node[EdgeLabel]{$3$}(011);
            \draw[EdgeGraph](01)--(101);
            \draw[EdgeGraph](10)--(101);
            \draw[EdgeGraph](10)edge[]node[EdgeLabel]{$3$}(110);
            \draw[EdgeGraph](011)edge[]node[EdgeLabel]{$5$}(0111);
            \draw[EdgeGraph](011)--(1011);
            \draw[EdgeGraph](101)edge[]node[EdgeLabel]{$3$}(1011);
            \draw[EdgeGraph](101)edge[]node[EdgeLabel]{$3$}(1101);
            \draw[EdgeGraph](110)--(1101);
            \draw[EdgeGraph](110)edge[]node[EdgeLabel]{$5$}(1110);
            \draw[EdgeGraph](0111)edge[]node[EdgeLabel]{$7$}(01111);
            \draw[EdgeGraph](0111)--(10111);
            \draw[EdgeGraph](1011)edge[]node[EdgeLabel]{$5$}(10111);
            \draw[EdgeGraph](1011)edge[]node[EdgeLabel]{$3$}(11011);
            \draw[EdgeGraph](1101)edge[]node[EdgeLabel]{$3$}(11011);
            \draw[EdgeGraph](1101)edge[]node[EdgeLabel]{$5$}(11101);
            \draw[EdgeGraph](1110)--(11101);
            \draw[EdgeGraph](1110)edge[]node[EdgeLabel]{$7$}(11110);
        \end{tikzpicture}
    \end{minipage}
    \label{subfig:operad_graph_Dias_U}}
    \hfill
    \subfloat[][The graph $\Par{\Dias_\InternalNode, \Vp}$ up to elements of degree~$4$.]{
    \begin{minipage}[c]{.45\textwidth}
        \centering
        \begin{tikzpicture}[Centering,xscale=.7,yscale=.7]
            \tikzset{every node}=[font=\scriptsize]
            \node(0)at(0,0){$\mathtt{0}$};
            \node(01)at(-1,-1){$\mathtt{01}$};
            \node(10)at(1,-1){$\mathtt{10}$};
            \node(011)at(-2,-2){$\mathtt{011}$};
            \node(101)at(0,-2){$\mathtt{101}$};
            \node(110)at(2,-2){$\mathtt{110}$};
            \node(0111)at(-3,-3){$\mathtt{0111}$};
            \node(1011)at(-1,-3){$\mathtt{1011}$};
            \node(1101)at(1,-3){$\mathtt{1101}$};
            \node(1110)at(3,-3){$\mathtt{1110}$};
            \node(01111)at(-4,-4){$\mathtt{01111}$};
            \node(10111)at(-2,-4){$\mathtt{10111}$};
            \node(11011)at(0,-4){$\mathtt{11011}$};
            \node(11101)at(2,-4){$\mathtt{11101}$};
            \node(11110)at(4,-4){$\mathtt{11110}$};
            \draw[EdgeGraph](0)--(01);
            \draw[EdgeGraph](0)--(10);
            \draw[EdgeGraph](10)--(101);
            \draw[EdgeGraph](10)--(110);
            \draw[EdgeGraph](01)--(110);
            \draw[EdgeGraph](01)--(011);
            \draw[EdgeGraph](101)--(1011);
            \draw[EdgeGraph](101)--(1110);
            \draw[EdgeGraph](110)--(1101);
            \draw[EdgeGraph](110)--(1110);
            \draw[EdgeGraph](011)--(1110);
            \draw[EdgeGraph](011)--(0111);
            \draw[EdgeGraph](1011)--(10111);
            \draw[EdgeGraph](1011)--(11110);
            \draw[EdgeGraph](1101)--(11011);
            \draw[EdgeGraph](1101)--(11110);
            \draw[EdgeGraph](1110)--(11110);
            \draw[EdgeGraph](1110)--(11101);
            \draw[EdgeGraph](0111)--(11110);
            \draw[EdgeGraph](0111)--(01111);
        \end{tikzpicture}
    \end{minipage}
    \label{subfig:operad_graph_Dias_V}}
    \caption{The pair $\Par{\Dias_\InternalNode, \Up, \Vp}$ of graded graphs.}
    \label{fig:operad_graphs_Dias}
\end{figure}
This pair of graded graphs is not $\phi$-diagonal dual since for instance,
\begin{math}
    (\Vp^\Dual \Up - \Up \Vp^\Dual)(\mathtt{10}) = 3\, (\mathtt{10}) + 2\, (\mathtt{01}).
\end{math}
Nevertheless, we have the following property.
\medbreak

\begin{Proposition} \label{prop:prefix_operad_graphs_UU_duality}
    The graded graph $\Par{\Dias_\InternalNode, \Up}$ is $\phi$-diagonal self-dual for the
    linear map $\phi : \K \Angle{\Dias_\InternalNode} \to \K \Angle{\Dias_\InternalNode}$
    satisfying
    \begin{equation}
        \phi\Par{\mathtt{1}^k \mathtt{0} \mathtt{1}^\ell}
        =
        \Par{\Han{k = 0} + \Han{\ell = 0} + 8 \Han{k \geq 1} k + 8 \Han{\ell \geq 1} \ell}
        \, \mathtt{1}^k \mathtt{0} \mathtt{1}^\ell
    \end{equation}
    for any $k, \ell \in \N$.
\end{Proposition}
\begin{proof}
    By a straightforward computation, by using~\eqref{equ:up_dias}, we can show that the
    relation
    \begin{math}
        \Par{\Up^\Dual \Up - \Up \Up^\Dual}\Par{\mathtt{1}^k \mathtt{0}\mathtt{1}^\ell}
        = \phi\Par{\mathtt{1}^k \mathtt{0}\mathtt{1}^\ell}
    \end{math}
    holds for all $k, \ell \geq 0$, establishing the statement of the proposition.
\end{proof}
\medbreak

The hook series of $\Par{\Dias_\InternalNode, \Up}$ satisfies
\begin{equation} \begin{split}
    \HookSeries{\Up}
    & =
    (\mathtt{0}) + (\mathtt{01}) + (\mathtt{10}) + 3\, (\mathtt{011}) + 2\, (\mathtt{101})
    + 3\, (\mathtt{110}) + 15\, (\mathtt{0111}) + 9\, (\mathtt{1011})
    \\
    & \quad + 9\, (\mathtt{1101}) + 15\, (\mathtt{1110}) + 105\, (\mathtt{01111})
    + 60\, (\mathtt{10111}) + 54\, (\mathtt{11011})
    \\
    & \quad + 60\, (\mathtt{11101}) + 105\, (\mathtt{11110}) + \cdots.
\end{split} \end{equation}
These coefficients form Triangle~\OEIS{A059366} of~\cite{Slo}.
\medbreak

%%%%%%%%%%%%%%%%%%%%%%%%%%%%%%%%%%%%%%%%%%%%%%%%%%%%%%%%%%%%%%%%%%%%%%%%%%%%%%%%%%%%%%%%%%%%
\subsubsection{Operad of integer compositions}
The \Def{operad of integer compositions} $\Comp$ is the operad wherein $\Comp(n)$ is the set
of all words of length $n \geq 1$ on the alphabet $\Bra{\mathtt{0}, \mathtt{1}}$ beginning
by $\mathtt{0}$. The partial composition $u \circ_i v$ of two such words $u$ and $v$
consists in replacing the $i$-th letter of $u$ by $v$ if $u_i = \mathtt{0}$ and by $\bar{v}$
if $u_i = \mathtt{1}$ where $\bar{v}$ is the one complement of $v$. This operad has been
    introduced in~\cite{Gir15} and admits the presentation
\begin{math}
    \Par{\GeneratingSet_\Comp, \Congr_\Comp}
\end{math}
where
\begin{equation}
    \GeneratingSet_\Comp = \Bra{\mathtt{00}, \mathtt{01}}
\end{equation}
and $\Congr_\Comp$ is the smallest operad congruence of
$\SyntaxTreesLeaf\Par{\GeneratingSet_\Comp}$ satisfying
\begin{subequations}
    \begin{equation}
        \Corolla\Par{\mathtt 00} \circ_1 \Corolla\Par{\mathtt 00}
        \Congr_\Comp
        \Corolla\Par{\mathtt 00} \circ_2 \Corolla\Par{\mathtt 00},
    \end{equation}
    \begin{equation}
        \Corolla\Par{\mathtt 01} \circ_1 \Corolla\Par{\mathtt 00}
        \Congr_\Comp
        \Corolla\Par{\mathtt 00} \circ_2 \Corolla\Par{\mathtt 01},
    \end{equation}
    \begin{equation}
        \Corolla\Par{\mathtt 01} \circ_1 \Corolla\Par{\mathtt 01}
        \Congr_\Comp
        \Corolla\Par{\mathtt 01} \circ_2 \Corolla\Par{\mathtt 00},
    \end{equation}
    \begin{equation}
        \Corolla\Par{\mathtt 00} \circ_1 \Corolla\Par{\mathtt 01}
        \Congr_\Comp
        \Corolla\Par{\mathtt 01} \circ_2 \Corolla\Par{\mathtt 01}.
    \end{equation}
\end{subequations}
Observe that these relations describe respectively the treelike expressions for the elements
$000$, $001$, $011$, and $010$ of $\Comp$.  Therefore, $\Comp$ is homogeneous and finitely
generated.
\medbreak

The pair of graded graphs $\Par{\Comp_\InternalNode, \Up, \Vp}$ satisfies
\begin{equation} \label{equ:up_comp}
    \Up(u) =
    \sum_{i \in [|u|]}
    u_1 \dots u_i \, \mathtt{0} \, u_{i + 1} \dots u_{|u|}
    +
    u_1 \dots u_i \, \mathtt{1} \, u_{i + 1} \dots u_{|u|},
\end{equation}
and
\begin{equation} \label{equ:vp_comp}
    \Vp(u) = u \mathtt{0} + u \mathtt{1},
\end{equation}
for any $u \in \Comp$ (see  Figure~\ref{fig:operad_graphs_Comp}).
\begin{figure}[ht]
    \centering
    \subfloat[][The graph $\Par{\Comp_\InternalNode, \Up}$ up to elements of degree~$4$.]{
    \begin{minipage}[c]{.47\textwidth}
        \centering
        \begin{tikzpicture}[Centering,xscale=.85,yscale=1]
            \tikzset{every node}=[font=\scriptsize]
            \node(0)at(0,0){$\mathtt{0}$};
            \node(00)at(-1,-1){$\mathtt{00}$};
            \node(01)at(1,-1){$\mathtt{01}$};
            \node(000)at(-2.25,-2){$\mathtt{000}$};
            \node(001)at(-.75,-2){$\mathtt{001}$};
            \node(010)at(.75,-2){$\mathtt{010}$};
            \node(011)at(2.25,-2){$\mathtt{011}$};
            \node(0000)at(-3.5,-3){$\mathtt{0000}$};
            \node(0001)at(-2.5,-3){$\mathtt{0001}$};
            \node(0010)at(-1.5,-3){$\mathtt{0010}$};
            \node(0011)at(-.5,-3){$\mathtt{0011}$};
            \node(0100)at(.5,-3){$\mathtt{0100}$};
            \node(0101)at(1.5,-3){$\mathtt{0101}$};
            \node(0110)at(2.5,-3){$\mathtt{0110}$};
            \node(0111)at(3.5,-3){$\mathtt{0111}$};
            \draw[EdgeGraph](0)--(00);
            \draw[EdgeGraph](0)--(01);
            \draw[EdgeGraph](00)edge[]node[EdgeLabel]{$2$}(000);
            \draw[EdgeGraph](00)--(001);
            \draw[EdgeGraph](00)--(010);
            \draw[EdgeGraph](01)--(001);
            \draw[EdgeGraph](01)--(010);
            \draw[EdgeGraph](01)edge[]node[EdgeLabel]{$2$}(011);
            \draw[EdgeGraph](000)edge[]node[EdgeLabel]{$3$}(0000);
            \draw[EdgeGraph](000)--(0001);
            \draw[EdgeGraph](000)--(0010);
            \draw[EdgeGraph](000)--(0100);
            \draw[EdgeGraph](001)edge[]node[EdgeLabel]{$2$}(0001);
            \draw[EdgeGraph](001)--(0010);
            \draw[EdgeGraph](001)--(0101);
            \draw[EdgeGraph](001)edge[]node[EdgeLabel,near start]{$2$}(0011);
            \draw[EdgeGraph](010)--(0010);
            \draw[EdgeGraph](010)edge[]node[EdgeLabel,near start]{$2$}(0100);
            \draw[EdgeGraph](010)--(0101);
            \draw[EdgeGraph](010)edge[]node[EdgeLabel]{$2$}(0110);
            \draw[EdgeGraph](011)--(0011);
            \draw[EdgeGraph](011)--(0101);
            \draw[EdgeGraph](011)--(0110);
            \draw[EdgeGraph](011)edge[]node[EdgeLabel]{$3$}(0111);
        \end{tikzpicture}
    \end{minipage}
    \label{subfig:operad_graph_Comp_U}}
    \hfill
    \subfloat[][The graph $\Par{\Comp_\InternalNode, \Vp}$ up to elements of degree~$4$.]{
    \begin{minipage}[c]{.47\textwidth}
        \centering
        \begin{tikzpicture}[Centering,xscale=.85,yscale=1]
            \tikzset{every node}=[font=\scriptsize]
            \node(0)at(0,0){$\mathtt{0}$};
            \node(00)at(-1,-1){$\mathtt{00}$};
            \node(01)at(1,-1){$\mathtt{01}$};
            \node(000)at(-2.25,-2){$\mathtt{000}$};
            \node(001)at(-.75,-2){$\mathtt{001}$};
            \node(010)at(.75,-2){$\mathtt{010}$};
            \node(011)at(2.25,-2){$\mathtt{011}$};
            \node(0000)at(-3.5,-3){$\mathtt{0000}$};
            \node(0001)at(-2.5,-3){$\mathtt{0001}$};
            \node(0010)at(-1.5,-3){$\mathtt{0010}$};
            \node(0011)at(-.5,-3){$\mathtt{0011}$};
            \node(0100)at(.5,-3){$\mathtt{0100}$};
            \node(0101)at(1.5,-3){$\mathtt{0101}$};
            \node(0110)at(2.5,-3){$\mathtt{0110}$};
            \node(0111)at(3.5,-3){$\mathtt{0111}$};
            \draw[EdgeGraph](0)--(00);
            \draw[EdgeGraph](0)--(01);
            \draw[EdgeGraph](00)--(000);
            \draw[EdgeGraph](00)--(001);
            \draw[EdgeGraph](01)--(010);
            \draw[EdgeGraph](01)--(011);
            \draw[EdgeGraph](000)--(0000);
            \draw[EdgeGraph](000)--(0001);
            \draw[EdgeGraph](001)--(0010);
            \draw[EdgeGraph](001)--(0011);
            \draw[EdgeGraph](010)--(0100);
            \draw[EdgeGraph](010)--(0101);
            \draw[EdgeGraph](011)--(0110);
            \draw[EdgeGraph](011)--(0111);
        \end{tikzpicture}
    \end{minipage}
    \label{subfig:operad_graph_Comp_V}}
    \caption{The pair $\Par{\Comp_\InternalNode, \Up, \Vp}$ of graded graphs.}
    \label{fig:operad_graphs_Comp}
\end{figure}
The poset of $\Par{\Comp_\InternalNode, \Up}$ is the \Def{composition poset} introduced and
studied in~\cite{BS05}.
\medbreak

\begin{Proposition} \label{prop:prefix_operad_graphs_Comp_duality}
    The pair $(\Comp, \Up, \Vp)$ of graded graphs is $2$-dual.
\end{Proposition}
\begin{proof}
    By a straightforward computation, by using~\eqref{equ:up_comp} and~\eqref{equ:vp_comp},
    we can infer the relation $\Par{\Vp^\Dual \Up - \Up \Vp^\Dual}(u) = 2u$ for all
    $u \in \Comp$, establishing the statement of the proposition.
\end{proof}
\medbreak

The hook series of $\Par{\Comp_\InternalNode, \Up}$ satisfies
\begin{equation}
    \HookSeries{\Up} = \sum_{u \in \Comp} (|u| - 1)! \, u.
\end{equation}
\medbreak

%%%%%%%%%%%%%%%%%%%%%%%%%%%%%%%%%%%%%%%%%%%%%%%%%%%%%%%%%%%%%%%%%%%%%%%%%%%%%%%%%%%%%%%%%%%%
\subsubsection{Operad of Motzkin paths}
The \Def{operad of Motzkin} paths $\Motz$ is the operad wherein $\Motz(n)$ is the set of all
words of length $n \geq 1$ of nonnegative integers starting and finishing by $\mathtt{0}$
and such that the absolute difference between two consecutive letters is at most $1$. The
partial composition $u \circ_i v$ of two such words consists in replacing the $i$-th letter
of $u$ by $v'$ where $v'$ is the word obtained by incrementing by $u_i$ all its letters.
This operad has been introduced in~\cite{Gir15} and admits the presentation
$\Par{\GeneratingSet_\Motz, \Congr_\Motz}$ where
\begin{equation}
    \GeneratingSet_\Motz := \Bra{\mathtt{00}, \mathtt{010}}
\end{equation}
and $\Congr_\Motz$ is the smallest operad congruence of
$\SyntaxTreesLeaf\Par{\GeneratingSet_\Motz}$ satisfying
\begin{subequations}
    \begin{equation}
        \Corolla\Par{\mathtt 00} \circ_1 \Corolla\Par{\mathtt 00}
        \Congr_\Motz
        \Corolla\Par{\mathtt 00} \circ_2 \Corolla\Par{\mathtt 00},
    \end{equation}
    \begin{equation}
        \Corolla\Par{\mathtt 010} \circ_1 \Corolla\Par{\mathtt 00}
        \Congr_\Motz
        \Corolla\Par{\mathtt 00} \circ_2 \Corolla\Par{\mathtt 010},
    \end{equation}
    \begin{equation}
        \Corolla\Par{\mathtt 00} \circ_1 \Corolla\Par{\mathtt 010}
        \Congr_\Motz
        \Corolla\Par{\mathtt 010} \circ_3 \Corolla\Par{\mathtt 00},
    \end{equation}
    \begin{equation}
        \Corolla\Par{\mathtt 010} \circ_1 \Corolla\Par{\mathtt 010}
        \Congr_\Motz
        \Corolla\Par{\mathtt 010} \circ_3 \Corolla\Par{\mathtt 010}.
    \end{equation}
\end{subequations}
Observe that these relations describe respectively the treelike expressions for the elements
$000$, $0010$, $0100$, and $01010$ of $\Motz$.  Therefore, $\Motz$ is homogeneous and
finitely generated. Observe moreover that, since the generator $\mathtt{010}$ is ternary,
$\Motz$ is not a binary operad.
\medbreak

The pair of graded graphs $\Par{\Motz_\InternalNode, \Up, \Vp}$ satisfies
\begin{equation} \label{equ:up_motz}
    \Up(u) =
    \sum_{i \in [|u|]}
    u_1 \dots u_i u_i u_{i + 1} \dots u_{|u|}
    +
    u_1 \dots u_i \Par{u_i + 1} u_i u_{i + 1} \dots u_{|u|},
\end{equation}
and
\begin{equation} \label{equ:vp_motz}
    \Vp(u) =
    \sum_{\substack{
        i \in [|u|] \\
        i = |u| \mbox{ \footnotesize or } u_i > u_{i + 1}
    }}
    u_1 \dots u_i u_i u_{i + 1} \dots u_{|u|}
    +
    u_1 \dots u_i \Par{u_i + 1} u_i u_{i + 1} \dots u_{|u|},
\end{equation}
for any $u \in \Motz$ (see Figure~\ref{fig:operad_graphs_Motz}).
\begin{figure}[ht]
    \centering
    \subfloat[][The graph $\Par{\Motz_\InternalNode, \Up}$ up to elements of degree~$3$.]{
    \begin{minipage}[c]{.46\textwidth}
        \centering
        \begin{tikzpicture}[Centering,xscale=1.9,yscale=.82,rotate=90]
            \tikzset{every node}=[font=\scriptsize]
            \node(0)at(0,0){$\mathtt{0}$};
            \node(00)at(-2,-1){$\mathtt{00}$};
            \node(010)at(2,-1){$\mathtt{010}$};
            \node(000)at(-7.5,-2){$\mathtt{000}$};
            \node(0010)at(-4.5,-2){$\mathtt{0010}$};
            \node(0100)at(-1.5,-2){$\mathtt{0100}$};
            \node(01010)at(4.5,-2){$\mathtt{01010}$};
            \node(0110)at(1.5,-2){$\mathtt{0110}$};
            \node(01210)at(7.5,-2){$\mathtt{01210}$};
            \node(0000)at(-10.5,-3){$\mathtt{0000}$};
            \node(00010)at(-9.5,-3){$\mathtt{00010}$};
            \node(00100)at(-8.5,-3){$\mathtt{00100}$};
            \node(01000)at(-7.5,-3){$\mathtt{01000}$};
            \node(001010)at(-6.5,-3){$\mathtt{001010}$};
            \node(00110)at(-5.5,-3){$\mathtt{00110}$};
            \node(001210)at(-4.5,-3){$\mathtt{001210}$};
            \node(010010)at(-3.5,-3){$\mathtt{010010}$};
            \node(010100)at(-2.5,-3){$\mathtt{010100}$};
            \node(01100)at(-1.5,-3){$\mathtt{01100}$};
            \node(012100)at(-.5,-3){$\mathtt{012100}$};
            \node(0101010)at(.5,-3){$\mathtt{0101010}$};
            \node(010110)at(1.5,-3){$\mathtt{010110}$};
            \node(0101210)at(2.5,-3){$\mathtt{0101210}$};
            \node(011010)at(3.5,-3){$\mathtt{011010}$};
            \node(0121010)at(4.5,-3){$\mathtt{0121010}$};
            \node(01110)at(5.5,-3){$\mathtt{01110}$};
            \node(011210)at(6.5,-3){$\mathtt{011210}$};
            \node(012110)at(7.5,-3){$\mathtt{012110}$};
            \node(0121210)at(8.5,-3){$\mathtt{0121210}$};
            \node(012210)at(9.5,-3){$\mathtt{012210}$};
            \node(0123210)at(10.5,-3){$\mathtt{0123210}$};
            \draw[EdgeGraph](0)--(00);
            \draw[EdgeGraph](0)--(010);
            \draw[EdgeGraph](00)edge[]node[EdgeLabel]{$2$}(000);
            \draw[EdgeGraph](00)--(0010);
            \draw[EdgeGraph](00)--(0100);
            \draw[EdgeGraph](010)--(0010);
            \draw[EdgeGraph](010)--(0100);
            \draw[EdgeGraph](010)--(0110);
            \draw[EdgeGraph](010)edge[]node[EdgeLabel]{$2$}(01010);
            \draw[EdgeGraph](010)--(01210);
            \draw[EdgeGraph](000)edge[]node[EdgeLabel]{$3$}(0000);
            \draw[EdgeGraph](000)--(00010);
            \draw[EdgeGraph](000)--(00100);
            \draw[EdgeGraph](000)--(01000);
            \draw[EdgeGraph](0010)edge[]node[EdgeLabel]{$2$}(00010);
            \draw[EdgeGraph](0010)--(00100);
            \draw[EdgeGraph](0010)edge[]node[EdgeLabel,near start]{$2$}(001010);
            \draw[EdgeGraph](0010)--(00110);
            \draw[EdgeGraph](0010)--(001210);
            \draw[EdgeGraph](0010)--(010010);
            \draw[EdgeGraph](0100)--(00100);
            \draw[EdgeGraph](0100)edge[]node[EdgeLabel,near start]{$2$}(01000);
            \draw[EdgeGraph](0100)--(010010);
            \draw[EdgeGraph](0100)edge[]node[EdgeLabel,near start]{$2$}(010100);
            \draw[EdgeGraph](0100)--(01100);
            \draw[EdgeGraph](0100)--(012100);
            \draw[EdgeGraph](01010)--(001010);
            \draw[EdgeGraph](01010)--(010010);
            \draw[EdgeGraph](01010)--(010100);
            \draw[EdgeGraph](01010)edge[]node[EdgeLabel,very near start]{$3$}(0101010);
            \draw[EdgeGraph](01010)--(010110);
            \draw[EdgeGraph](01010)--(0101210);
            \draw[EdgeGraph](01010)--(011010);
            \draw[EdgeGraph](01010)--(0121010);
            \draw[EdgeGraph](0110)--(00110);
            \draw[EdgeGraph](0110)--(01100);
            \draw[EdgeGraph](0110)--(010110);
            \draw[EdgeGraph](0110)--(011010);
            \draw[EdgeGraph](0110)edge[]node[EdgeLabel,very near end]{$2$}(01110);
            \draw[EdgeGraph](0110)--(011210);
            \draw[EdgeGraph](0110)--(012110);
            \draw[EdgeGraph](01210)--(001210);
            \draw[EdgeGraph](01210)--(012100);
            \draw[EdgeGraph](01210)--(0101210);
            \draw[EdgeGraph](01210)--(0121010);
            \draw[EdgeGraph](01210)--(011210);
            \draw[EdgeGraph](01210)--(012110);
            \draw[EdgeGraph](01210)edge[]node[EdgeLabel]{$2$}(0121210);
            \draw[EdgeGraph](01210)--(012210);
            \draw[EdgeGraph](01210)--(0123210);
        \end{tikzpicture}
    \end{minipage}
    \label{subfig:operad_graph_Motz_U}}
    \hfill
    \subfloat[][The graph $\Par{\Motz_\InternalNode, \Vp}$ up to elements of degree~$3$.]{
    \begin{minipage}[c]{.46\textwidth}
        \centering
        \begin{tikzpicture}[Centering,xscale=1.9,yscale=.82,rotate=90]
            \tikzset{every node}=[font=\scriptsize]
            \node(0)at(0,0){$\mathtt{0}$};
            \node(00)at(-2,-1){$\mathtt{00}$};
            \node(010)at(2,-1){$\mathtt{010}$};
            \node(000)at(-7.5,-2){$\mathtt{000}$};
            \node(0010)at(-4.5,-2){$\mathtt{0010}$};
            \node(0100)at(-1.5,-2){$\mathtt{0100}$};
            \node(01010)at(4.5,-2){$\mathtt{01010}$};
            \node(0110)at(1.5,-2){$\mathtt{0110}$};
            \node(01210)at(7.5,-2){$\mathtt{01210}$};
            \node(0000)at(-10.5,-3){$\mathtt{0000}$};
            \node(00010)at(-9.5,-3){$\mathtt{00010}$};
            \node(00100)at(-8.5,-3){$\mathtt{00100}$};
            \node(01000)at(-7.5,-3){$\mathtt{01000}$};
            \node(001010)at(-6.5,-3){$\mathtt{001010}$};
            \node(00110)at(-5.5,-3){$\mathtt{00110}$};
            \node(001210)at(-4.5,-3){$\mathtt{001210}$};
            \node(010010)at(-3.5,-3){$\mathtt{010010}$};
            \node(010100)at(-2.5,-3){$\mathtt{010100}$};
            \node(01100)at(-1.5,-3){$\mathtt{01100}$};
            \node(012100)at(-.5,-3){$\mathtt{012100}$};
            \node(0101010)at(.5,-3){$\mathtt{0101010}$};
            \node(010110)at(1.5,-3){$\mathtt{010110}$};
            \node(0101210)at(2.5,-3){$\mathtt{0101210}$};
            \node(011010)at(3.5,-3){$\mathtt{011010}$};
            \node(0121010)at(4.5,-3){$\mathtt{0121010}$};
            \node(01110)at(5.5,-3){$\mathtt{01110}$};
            \node(011210)at(6.5,-3){$\mathtt{011210}$};
            \node(012110)at(7.5,-3){$\mathtt{012110}$};
            \node(0121210)at(8.5,-3){$\mathtt{0121210}$};
            \node(012210)at(9.5,-3){$\mathtt{012210}$};
            \node(0123210)at(10.5,-3){$\mathtt{0123210}$};
            \draw[EdgeGraph](0)--(00);
            \draw[EdgeGraph](0)--(010);
            \draw[EdgeGraph](00)--(000);
            \draw[EdgeGraph](00)--(0010);
            \draw[EdgeGraph](010)--(0100);
            \draw[EdgeGraph](010)--(0110);
            \draw[EdgeGraph](010)--(01010);
            \draw[EdgeGraph](010)--(01210);
            \draw[EdgeGraph](000)--(0000);
            \draw[EdgeGraph](000)--(00010);
            \draw[EdgeGraph](0010)--(00110);
            \draw[EdgeGraph](0010)--(001210);
            \draw[EdgeGraph](0010)--(00100);
            \draw[EdgeGraph](0010)--(001010);
            \draw[EdgeGraph](0110)--(011210);
            \draw[EdgeGraph](0110)--(01110);
            \draw[EdgeGraph](0110)--(01100);
            \draw[EdgeGraph](0110)--(011010);
            \draw[EdgeGraph](01210)--(012210);
            \draw[EdgeGraph](01210)--(0123210);
            \draw[EdgeGraph](01210)--(012110);
            \draw[EdgeGraph](01210)--(0121210);
            \draw[EdgeGraph](01210)--(012100);
            \draw[EdgeGraph](01210)--(0121010);
            \draw[EdgeGraph](0100)--(01100);
            \draw[EdgeGraph](0100)--(012100);
            \draw[EdgeGraph](0100)--(01000);
            \draw[EdgeGraph](0100)--(010010);
            \draw[EdgeGraph](01010)--(011010);
            \draw[EdgeGraph](01010)--(0121010);
            \draw[EdgeGraph](01010)--(0101010);
            \draw[EdgeGraph](01010)--(010110);
            \draw[EdgeGraph](01010)--(0101210);
            \draw[EdgeGraph](01010)--(010100);
        \end{tikzpicture}
    \end{minipage}
    \label{subfig:operad_graph_Motz_V}}
    \caption{The pair $\Par{\Motz_\InternalNode, \Up, \Vp}$ of graded graphs.}
    \label{fig:operad_graphs_Motz}
\end{figure}
\medbreak

\begin{Proposition} \label{prop:prefix_operad_graphs_Motz_duality}
    The pair $\Par{\Motz_\InternalNode, \Up, \Vp}$ of graded graphs is $\phi$-diagonal
    dual for the linear map
    $\phi : \K \Angle{\Motz_\InternalNode} \to \K \Angle{\Motz_\InternalNode}$ satisfying
    \begin{equation}
        \phi(u) = \Par{2 + \# \Bra{i \in [|u| - 1] : u_i \ne u_{i + 1}}} \, u
    \end{equation}
    for any $u \in \Motz$.
\end{Proposition}
\begin{proof}
    By a straightforward computation, by using~\eqref{equ:up_motz} and~\eqref{equ:vp_motz},
    we can infer the relation $\Par{\Vp^\Dual \Up - \Up \Vp^\Dual}(u) = \phi(u)$ for all
    $u \in \Motz$, establishing the statement of the proposition.
\end{proof}
\medbreak

The hook series of $\Par{\Motz_\InternalNode, \Up}$ satisfies
\begin{equation} \begin{split}
    \HookSeries{\Up}
    & =
    (\mathtt{0}) + (\mathtt{00}) + 2\, (\mathtt{000}) + (\mathtt{010}) + 6\, (\mathtt{0000})
    + 2\, (\mathtt{0010}) + 2\, (\mathtt{0100}) + (\mathtt{0110})
    \\
    & \quad + 6\, (\mathtt{00010}) + 6\, (\mathtt{00100}) + 3\, (\mathtt{00110})
    + 6\, (\mathtt{01000}) + 2\, (\mathtt{01010}) + 3\, (\mathtt{01100})
    \\
    & \quad + 2\, (\mathtt{01110}) + (\mathtt{01210}) + 6\, (\mathtt{001010})
    + 3\, (\mathtt{001210}) + 6\, (\mathtt{010010}) + 6\, (\mathtt{010100})
    \\
    & \quad + 3\, (\mathtt{010110}) + 3\, (\mathtt{011010}) + 2\, (\mathtt{011210})
    + 3\, (\mathtt{012100}) + 2\, (\mathtt{012110})  + (\mathtt{012210})
    \\
    & \quad + 6\, (\mathtt{0101010}) + 3\, (\mathtt{0101210}) + 3\, (\mathtt{0121010})
    + 2\, (\mathtt{0121210}) + (\mathtt{0123210}) + \cdots.
\end{split} \end{equation}
We do not have a concise combinatorial description for these hook coefficients.
\medbreak

Besides, by representing any element of $\Motz$ as a path in the quarter plane (that is, by
drawing points $\Par{i - 1, u_i}$ for all $i \in [|u|]$ and by connecting all pairs of
adjacent points by segments), in the prefix poset of $\Motz$, one has $u \OrderPrefixes v$
if and only if $u$ can be obtained from $v$ by collapsing into points some factors of $v$
that are Motzkin paths. For instance, one has
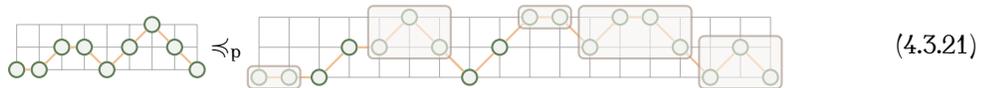
\begin{equation}
    \begin{tikzpicture}[Centering,scale=.3]
        \draw[Grid](0,0)grid(8,2);
        \node[PathNode](0)at(0,0){};
        \node[PathNode](1)at(1,0){};
        \node[PathNode](2)at(2,1){};
        \node[PathNode](3)at(3,1){};
        \node[PathNode](4)at(4,0){};
        \node[PathNode](5)at(5,1){};
        \node[PathNode](6)at(6,2){};
        \node[PathNode](7)at(7,1){};
        \node[PathNode](8)at(8,0){};
        \draw[PathStep](0)--(1);
        \draw[PathStep](1)--(2);
        \draw[PathStep](2)--(3);
        \draw[PathStep](3)--(4);
        \draw[PathStep](4)--(5);
        \draw[PathStep](5)--(6);
        \draw[PathStep](6)--(7);
        \draw[PathStep](7)--(8);
    \end{tikzpicture}
    \OrderPrefixes
    \begin{tikzpicture}[Centering,scale=.4]
        \draw[Grid](0,0)grid(17,2);
        \node[PathNode](0)at(0,0){};
        \node[PathNode](1)at(1,0){};
        \node[PathNode](2)at(2,0){};
        \node[PathNode](3)at(3,1){};
        \node[PathNode](4)at(4,1){};
        \node[PathNode](5)at(5,2){};
        \node[PathNode](6)at(6,1){};
        \node[PathNode](7)at(7,0){};
        \node[PathNode](8)at(8,1){};
        \node[PathNode](9)at(9,2){};
        \node[PathNode](10)at(10,2){};
        \node[PathNode](11)at(11,1){};
        \node[PathNode](12)at(12,2){};
        \node[PathNode](13)at(13,2){};
        \node[PathNode](14)at(14,1){};
        \node[PathNode](15)at(15,0){};
        \node[PathNode](16)at(16,1){};
        \node[PathNode](17)at(17,0){};
        \draw[PathStep](0)--(1);
        \draw[PathStep](1)--(2);
        \draw[PathStep](2)--(3);
        \draw[PathStep](3)--(4);
        \draw[PathStep](4)--(5);
        \draw[PathStep](5)--(6);
        \draw[PathStep](6)--(7);
        \draw[PathStep](6)--(7);
        \draw[PathStep](7)--(8);
        \draw[PathStep](8)--(9);
        \draw[PathStep](9)--(10);
        \draw[PathStep](10)--(11);
        \draw[PathStep](11)--(12);
        \draw[PathStep](12)--(13);
        \draw[PathStep](13)--(14);
        \draw[PathStep](14)--(15);
        \draw[PathStep](15)--(16);
        \draw[PathStep](16)--(17);
        \node[fit=(0)(1),Box,opacity=.6]{};
        \node[fit=(4)(5)(6),Box,opacity=.6]{};
        \node[fit=(9)(10),Box,opacity=.6]{};
        \node[fit=(11)(12)(13)(14),Box,opacity=.6]{};
        \node[fit=(15)(16)(17),Box,opacity=.6]{};
    \end{tikzpicture}
\end{equation}
where the factors to collapse are framed. Observe also that the prefix poset of $\Motz$ is
not a meet-semilattice since $00 \OrderPrefixes 0010$, $00 \OrderPrefixes 0100$, $010
\OrderPrefixes 0010$, $010 \OrderPrefixes 0100$, $00$ and $010$ are incomparable, and $0010$
and $0100$ are incomparable.
\medbreak

%%%%%%%%%%%%%%%%%%%%%%%%%%%%%%%%%%%%%%%%%%%%%%%%%%%%%%%%%%%%%%%%%%%%%%%%%%%%%%%%%%%%%%%%%%%%
\subsubsection{Operads of $m$-trees}
For any integer $m \geq 0$, the \Def{operad of $m$-trees} $\FCat{m}$ is the operad wherein
$\FCat{m}(n)$ is the set of all words $u$ of length $n \geq 1$ of nonnegative integers such
that $u_1 = 0$ and, for all $i \in [n - 1]$, $u_{i + 1} \leq u_i + m$. The partial
composition $u \circ_i v$ of two words $u$ and $v$ consists in replacing the $i$-th letter
of $u$ by $v'$ where $v'$ is the word obtained by incrementing by $u_i$ all its letters.
This operad has been introduced in~\cite{Gir15} and admits the presentation
$\Par{\GeneratingSet_{\FCat{m}}, \Congr_{\FCat{m}}}$ where
\begin{equation}
    \GeneratingSet_{\FCat{m}} := \Bra{00, 01, \dots, 0m}
\end{equation}
and $\Congr_{\FCat{m}}$ is the smallest operad congruence of
$\SyntaxTreesLeaf\Par{\GeneratingSet_{\FCat{m}}}$ satisfying
\begin{equation}
    \Corolla\Par{{\mathtt 0} k_3} \circ_1 \Corolla\Par{{\mathtt 0} k_1}
    \Congr_{\FCat{m}}
    \Corolla\Par{{\mathtt 0} k_1} \circ_2 \Corolla\Par{{\mathtt 0} k_2},
    \qquad
    k_3 = k_1 + k_2.
\end{equation}
Observe that this relation describes the treelike expressions for the element $0 k_1
\Par{k_1 + k_2}$ of $\FCat{m}$.  Therefore, $\FCat{m}$ is homogeneous and finitely
generated.
\medbreak

The pair of graded graphs $\Par{{\FCat{m}}_\InternalNode, \Up, \Vp}$ satisfies
\begin{equation} \label{equ:up_fcat}
    \Up(u) =
    \sum_{i \in [|u|]} \sum_{a \in [0, m]}
    u_1 \dots u_i \Par{u_i + a} u_{i + 1} \dots u_{|u|},
\end{equation}
and
\begin{equation} \label{equ:vp_fcat}
    \Vp(u) = \sum_{a \in \Han{0, u_{|u|} + m}} u a,
\end{equation}
for any $u \in \FCat{m}$ (see Figure~\ref{fig:operad_graphs_FCat_1}).
\begin{figure}[ht]
    \centering
    \subfloat[][The graph $\Par{{\FCat{1}}_\InternalNode, \Up}$ up to elements of
    degree~$3$.]{
    \begin{minipage}[c]{.46\textwidth}
        \centering
        \begin{tikzpicture}[Centering,xscale=1.4,yscale=1.1,rotate=90]
            \tikzset{every node}=[font=\scriptsize]
            \node(0)at(0,0){$0$};
            \node(00)at(-1,-1){$00$};
            \node(01)at(1,-1){$01$};
            \node(000)at(-4,-2){$000$};
            \node(010)at(-2,-2){$010$};
            \node(001)at(0,-2){$001$};
            \node(011)at(2,-2){$011$};
            \node(012)at(4,-2){$012$};
            \node(0000)at(-6.5,-3){$0000$};
            \node(0100)at(-5.5,-3){$0100$};
            \node(0010)at(-4.5,-3){$0010$};
            \node(0001)at(-3.5,-3){$0001$};
            \node(0110)at(-2.5,-3){$0110$};
            \node(0120)at(-1.5,-3){$0120$};
            \node(0101)at(-.5,-3){$0101$};
            \node(0011)at(.5,-3){$0011$};
            \node(0012)at(1.5,-3){$0012$};
            \node(0121)at(2.5,-3){$0121$};
            \node(0111)at(3.5,-3){$0111$};
            \node(0112)at(4.5,-3){$0112$};
            \node(0122)at(5.5,-3){$0122$};
            \node(0123)at(6.5,-3){$0123$};
            \draw[EdgeGraph](0)--(00);
            \draw[EdgeGraph](0)--(01);
            \draw[EdgeGraph](00)edge[]node[EdgeLabel]{$2$}(000);
            \draw[EdgeGraph](00)--(001);
            \draw[EdgeGraph](01)--(010);
            \draw[EdgeGraph](01)edge[]node[EdgeLabel]{$2$}(011);
            \draw[EdgeGraph](01)--(012);
            \draw[EdgeGraph](000)edge[]node[EdgeLabel]{$3$}(0000);
            \draw[EdgeGraph](000)--(0100);
            \draw[EdgeGraph](000)--(0010);
            \draw[EdgeGraph](000)--(0001);
            \draw[EdgeGraph](010)--(0100);
            \draw[EdgeGraph](010)--(0010);
            \draw[EdgeGraph](010)edge[]node[EdgeLabel]{$2$}(0110);
            \draw[EdgeGraph](010)--(0120);
            \draw[EdgeGraph](010)--(0101);
            \draw[EdgeGraph](001)edge[]node[EdgeLabel,near start]{$2$}(0001);
            \draw[EdgeGraph](001)--(0101);
            \draw[EdgeGraph](001)edge[]node[EdgeLabel]{$2$}(0011);
            \draw[EdgeGraph](001)--(0012);
            \draw[EdgeGraph](011)--(0011);
            \draw[EdgeGraph](011)--(0121);
            \draw[EdgeGraph](011)edge[]node[EdgeLabel,near start]{$3$}(0111);
            \draw[EdgeGraph](011)--(0112);
            \draw[EdgeGraph](012)--(0012);
            \draw[EdgeGraph](012)edge[]node[EdgeLabel]{$2$}(0112);
            \draw[EdgeGraph](012)edge[]node[EdgeLabel]{$2$}(0122);
            \draw[EdgeGraph](012)--(0123);
        \end{tikzpicture}
    \end{minipage}
    \label{subfig:operad_graph_FCat_1_U}}
    \hfill
    \subfloat[][The graph $\Par{{\FCat{1}}_\InternalNode, \Vp}$ up to elements of
    degree~$3$.]{
    \begin{minipage}[c]{.46\textwidth}
        \centering
        \begin{tikzpicture}[Centering,xscale=1.4,yscale=1.1,rotate=90]
            \tikzset{every node}=[font=\scriptsize]
            \node(0)at(0,0){$0$};
            \node(00)at(-1,-1){$00$};
            \node(01)at(1,-1){$01$};
            \node(000)at(-4,-2){$000$};
            \node(010)at(-2,-2){$010$};
            \node(001)at(0,-2){$001$};
            \node(011)at(2,-2){$011$};
            \node(012)at(4,-2){$012$};
            \node(0000)at(-6.5,-3){$0000$};
            \node(0100)at(-5.5,-3){$0100$};
            \node(0010)at(-4.5,-3){$0010$};
            \node(0001)at(-3.5,-3){$0001$};
            \node(0110)at(-2.5,-3){$0110$};
            \node(0120)at(-1.5,-3){$0120$};
            \node(0101)at(-.5,-3){$0101$};
            \node(0011)at(.5,-3){$0011$};
            \node(0012)at(1.5,-3){$0012$};
            \node(0121)at(2.5,-3){$0121$};
            \node(0111)at(3.5,-3){$0111$};
            \node(0112)at(4.5,-3){$0112$};
            \node(0122)at(5.5,-3){$0122$};
            \node(0123)at(6.5,-3){$0123$};
            \draw[EdgeGraph](0)--(00);
            \draw[EdgeGraph](0)--(01);
            \draw[EdgeGraph](00)--(000);
            \draw[EdgeGraph](00)--(001);
            \draw[EdgeGraph](01)--(010);
            \draw[EdgeGraph](01)--(011);
            \draw[EdgeGraph](01)--(012);
            \draw[EdgeGraph](000)--(0000);
            \draw[EdgeGraph](000)--(0001);
            \draw[EdgeGraph](001)--(0010);
            \draw[EdgeGraph](001)--(0011);
            \draw[EdgeGraph](001)--(0012);
            \draw[EdgeGraph](010)--(0100);
            \draw[EdgeGraph](010)--(0101);
            \draw[EdgeGraph](011)--(0110);
            \draw[EdgeGraph](011)--(0111);
            \draw[EdgeGraph](011)--(0112);
            \draw[EdgeGraph](012)--(0120);
            \draw[EdgeGraph](012)--(0121);
            \draw[EdgeGraph](012)--(0122);
            \draw[EdgeGraph](012)--(0123);
        \end{tikzpicture}
    \end{minipage}
    \label{subfig:operad_graph_FCat_1_V}}
    \caption{The pair $\Par{{\FCat{1}}_\InternalNode, \Up, \Vp}$ of graded graphs.}
    \label{fig:operad_graphs_FCat_1}
\end{figure}
\medbreak

\begin{Proposition} \label{prop:prefix_operad_graphs_FCat_duality}
    For any $m \geq 0$, the pair $\Par{{\FCat{m}}_\InternalNode, \Up, \Vp}$ of graded graphs
    is $\phi$-diagonal dual for the linear map $\phi : \K \Angle{{\FCat{m}}_\InternalNode}
    \to \K \Angle{{\FCat{m}}_\InternalNode}$ satisfying
    \begin{equation}
        \phi(u) = (m + 1)\, u
    \end{equation}
    for any $u \in \FCat{m}$.
\end{Proposition}
\begin{proof}
    By a straightforward computation, by using~\eqref{equ:up_fcat} and~\eqref{equ:vp_fcat},
    we can infer the relation $\Par{\Vp^\Dual \Up - \Up \Vp^\Dual}(u) = \phi(u)$ for all
    $u \in \FCat{m}$, establishing the statement of the proposition.
\end{proof}
\medbreak

The hook series of $\Par{{\FCat{1}}_\InternalNode, \Up}$ satisfies
\begin{equation} \begin{split}
    \HookSeries{\Up}
    & =
    (\mathtt{0}) + (\mathtt{00}) + (\mathtt{01}) + 2\, (\mathtt{000}) + 2\, (\mathtt{001})
    + (\mathtt{010}) + 2\, (\mathtt{011}) + (\mathtt{012}) + 6\, (\mathtt{0000})
    \\
    & \quad + 6\, (\mathtt{0001}) + 3\, (\mathtt{0010}) + 6\, (\mathtt{0011})
    + 3\, (\mathtt{0012})+ 3\, (\mathtt{0100}) + 3\, (\mathtt{0101}) + 2\, (\mathtt{0110})
    \\
    & \quad + 6\, (\mathtt{0111}) + 4\, (\mathtt{0112}) + (\mathtt{0120})
    + 2\, (\mathtt{0121}) + 2\, (\mathtt{0122}) + (\mathtt{0123}) + \cdots,
\end{split} \end{equation}
and the one of $\Par{{\FCat{2}}_\InternalNode, \Up}$ satisfies
\begin{equation} \begin{split}
    \HookSeries{\Up}
    & =
    (\mathtt{0}) + (\mathtt{00}) + (\mathtt{01}) + (\mathtt{02}) + 2\, (\mathtt{000})
    + 2\, (\mathtt{001}) + 2\, (\mathtt{002}) + (\mathtt{010}) + 2\, (\mathtt{011})
    + 2\, (\mathtt{012})
    \\
    & \quad + (\mathtt{013}) + (\mathtt{020}) + (\mathtt{021}) + 2\, (\mathtt{022})
    + (\mathtt{023}) + (\mathtt{024}) + 6\, (\mathtt{0000}) + 6\, (\mathtt{0001})
    \\
    & \quad + 6\, (\mathtt{0002}) + 3\, (\mathtt{0010}) + 6\, (\mathtt{0011})
    + 6\, (\mathtt{0012}) + 3\, (\mathtt{0013}) + 3\, (\mathtt{0020}) + 3\, (\mathtt{0021})
    \\
    & \quad + 6\, (\mathtt{0022}) + 3\, (\mathtt{0023}) + 3\, (\mathtt{0024})
    + 3\, (\mathtt{0100}) + 3\, (\mathtt{0101}) + 3\, (\mathtt{0102})  + 2\, (\mathtt{0110})
    \\
    & \quad + 6\, (\mathtt{0111}) + 6\, (\mathtt{0112}) + 4\, (\mathtt{0113})
    + 2\, (\mathtt{0120}) + 3\, (\mathtt{0121}) + 6\, (\mathtt{0122}) + 4\, (\mathtt{0123})
    \\
    & \quad
    + 3\, (\mathtt{0124}) + (\mathtt{0130}) + 2\, (\mathtt{0131}) + 2\, (\mathtt{0132})
    + 2\, (\mathtt{0133}) + (\mathtt{0134}) + (\mathtt{0135})
    \\
    & \quad + 3\, (\mathtt{0200}) + 3\, (\mathtt{0201}) + 3\, (\mathtt{0202})
    + (\mathtt{0210}) + 3\, (\mathtt{0211}) + 3\, (\mathtt{0212}) + 2\, (\mathtt{0213})
    \\
    & \quad + 2\, (\mathtt{0220}) + 2\, (\mathtt{0221}) + 6\, (\mathtt{0222})
    + 4\, (\mathtt{0223}) + 4\, (\mathtt{0224}) + (\mathtt{0230}) + (\mathtt{0231})
    \\
    & \quad + 2\, (\mathtt{0232}) + 2\, (\mathtt{0233}) + 2\, (\mathtt{0234})
    + (\mathtt{0235}) + (\mathtt{0240}) + (\mathtt{0241}) + 2\, (\mathtt{0242})
    + (\mathtt{0243})
    \\
    & \quad + 2\, (\mathtt{0244}) + (\mathtt{0245}) + (\mathtt{0246}) + \cdots.
\end{split} \end{equation}
We do not have a concise combinatorial description for these hook coefficients.
\medbreak

%%%%%%%%%%%%%%%%%%%%%%%%%%%%%%%%%%%%%%%%%%%%%%%%%%%%%%%%%%%%%%%%%%%%%%%%%%%%%%%%%%%%%%%%%%%%
%%%%%%%%%%%%%%%%%%%%%%%%%%%%%%%%%%%%%%%%%%%%%%%%%%%%%%%%%%%%%%%%%%%%%%%%%%%%%%%%%%%%%%%%%%%%
%%%%%%%%%%%%%%%%%%%%%%%%%%%%%%%%%%%%%%%%%%%%%%%%%%%%%%%%%%%%%%%%%%%%%%%%%%%%%%%%%%%%%%%%%%%%
\section*{Perspectives and open questions}
We finish this work by presenting three open questions and research directions.
\medbreak

As seen in Section~\ref{subsec:examples_graph_operads}, the pair of graded graphs associated
with the operads $\As$, $\Comp$, $\Motz$, and $\FCat{m}$, $m \geq 0$, are $\phi$-diagonal
dual, while the pair of graded graphs of $\Dias$ is not.  By computer exploration, we
conjecture that some classical operads appearing in the literature have also this property
of $\phi$-diagonal duality for their pair of graded graphs. This is the case for the
$2$-associative operad $\TwoAs$~\cite{LR06}, for the operad $\AsTwo$~\cite{Dot09}, for the
dipterous operad $\Dip$~\cite{LR03}, and for the duplicial operad $\Dup$~\cite{Lod08}.  The
first question is to obtain in general a necessary and sufficient condition for an
homogeneous and finitely presented operad $\Operad$ for the $\phi$-diagonal duality of its
pair of graded graphs $\Par{\Operad_\InternalNode, \Up, \Vp}$. Ideally, this condition
should relate to the presentation of $\Operad$.
\medbreak

The second question concerns applications of $\phi$-diagonal duality to enumerative
problems. We propose to understand to what extent this generalized version of graph duality
helps to obtain enumerative formulas. Recall that classical graph duality~\cite{Fom94}
leads, from the identity $\Vp^\Dual \Up^n = \Up^n \Vp^\Dual + n \Up^{n - 1}$, $n \geq 0$, to
a proof of~\eqref{equ:sum_square_Young_lattice} relating numbers of standard Young tableaux
and numbers of permutations. A starting point is to use
Proposition~\ref{prop:relation_U_V_phi_diagonal_duality} and the
Relation~\eqref{equ:relation_U_V_phi_diagonal_duality}, which is a generalization of the
previous identity, in order to relate other families of combinatorial objects in a similar
way.
\medbreak

A last research direction consists in, rather than considering operads to construct graded
graphs, use operad to construct trees. Roughly speaking, a tree can be built from a graded
graph by deleting some of its edges. To obtain such a tree, we can consider a variant of the
map $\Up$ (see~\eqref{equ:up_operads}) wherein the apparitions of certain terms are
forbidden.  This could be achieved by the use of colored operads~\cite{Yau16} since the
partial composition maps of these structures is restricted due to the use of colors. A
similar mechanism is used in~\cite{Gir19} wherein graphs of some colored versions of
combinatorial objects are built.  The main interest to search for trees instead of graded
graphs relies on the fact that trees can be thought as generating trees. These structures
can be used to design efficient algorithms for the exhaustive generation of the objects or
for random generation.  The aim is to build a framework leading to such generating trees
from any family of combinatorial objects endowed with the structure of a homogeneous and
finitely presented operad.
\medbreak

%%%%%%%%%%%%%%%%%%%%%%%%%%%%%%%%%%%%%%%%%%%%%%%%%%%%%%%%%%%%%%%%%%%%%%%%%%%%%%%%%%%%%%%%%%%%
%%%%%%%%%%%%%%%%%%%%%%%%%%%%%%%%%%%%%%%%%%%%%%%%%%%%%%%%%%%%%%%%%%%%%%%%%%%%%%%%%%%%%%%%%%%%
%%%%%%%%%%%%%%%%%%%%%%%%%%%%%%%%%%%%%%%%%%%%%%%%%%%%%%%%%%%%%%%%%%%%%%%%%%%%%%%%%%%%%%%%%%%%
\begin{footnotesize}
\bibliographystyle{alpha}
\bibliography{Bibliography}

\begin{thebibliography}{HNT05}

\bibitem[BN98]{BN98}
F.~Baader and T.~Nipkow.
\newblock {\em Term rewriting and all that}.
\newblock Cambridge University Press, 1998.
\newblock xii+301.

\bibitem[BP05]{BS05}
A.~Bjöner and Stanley~R. P.
\newblock {An analogue of Young’s lattice for compositions}.
\newblock {\em \Arxiv{math/0508043v4}}, 2005.

\bibitem[Cha05]{Cha05}
F.~Chapoton.
\newblock {On some anticyclic operads}.
\newblock {\em Algebr. Geom. Topol.}, 5:53--69, 2005.

\bibitem[Dot09]{Dot09}
V.~Dotsenko.
\newblock {Compatible associative products and trees}.
\newblock {\em Algebr. Number Theory}, 3(5):567--586, 2009.

\bibitem[Fom94]{Fom94}
S.~Fomin.
\newblock {Duality of graded graphs}.
\newblock {\em J. Algebr. Comb.}, 3(4):357--404, 1994.

\bibitem[FRT54]{FRT54}
J.~S. Frame, G.~de~B. Robinson, and R.~M. Thrall.
\newblock {The hook graphs of the symmetric groups}.
\newblock {\em Canadian J. Math.}, 6:316--324, 1954.

\bibitem[Gir15]{Gir15}
S.~Giraudo.
\newblock {Combinatorial operads from monoids}.
\newblock {\em J. Algebr. Comb.}, 41(2):493--538, 2015.

\bibitem[Gir16]{Gir16}
S.~Giraudo.
\newblock {Pluriassociative algebras I: The pluriassociative operad}.
\newblock {\em Adv. Appli. Math.}, 77:1--42, 2016.

\bibitem[Gir18]{Gir18}
S.~Giraudo.
\newblock {\em {Nonsymmetric Operads in Combinatorics}}.
\newblock Springer Nature Switzerland AG, 2018.
\newblock ix+172.

\bibitem[Gir19]{Gir19}
S.~Giraudo.
\newblock {Colored operads, series on colored operads, and combinatorial
  generating systems}.
\newblock {\em Discrete Math.}, 342(6):1624--1657, 2019.

\bibitem[HNT05]{HNT05}
F.~Hivert, J.-C. Novelli, and J.-Y. Thibon.
\newblock {The Algebra of Binary Search Trees}.
\newblock {\em Theor. Comput. Sci.}, 339(1):129--165, 2005.

\bibitem[Knu97]{Knu97}
D.~Knuth.
\newblock {\em {The Art of Computer Programming, volume 1: Fundamental
  Algorithms}}.
\newblock Addison Wesley Longman, 3rd edition, 1997.
\newblock xx+650.

\bibitem[Knu98]{Knu98}
D.~Knuth.
\newblock {\em {The Art of Computer Programming, Sorting and Searching}},
  volume~3.
\newblock Addison Wesley, 2nd edition, 1998.

\bibitem[Lod01]{Lod01}
J.-L. Loday.
\newblock {Dialgebras}.
\newblock In {\em Dialgebras and related operads}, volume 1763 of {\em Lect.
  Notes Math.}, pages 7--66. Springer, Berlin, 2001.

\bibitem[Lod08]{Lod08}
J.-L. Loday.
\newblock {Generalized Bialgebras and Triples of Operads}.
\newblock {\em Ast{\'e}risque}, 320:1--114, 2008.

\bibitem[LR03]{LR03}
J.-L. Loday and M.~Ronco.
\newblock {Alg{\`e}bres de Hopf colibres}.
\newblock {\em C. R. Math.}, 337(3):153--158, 2003.

\bibitem[LR06]{LR06}
J.-L. Loday and M.~Ronco.
\newblock {On the structure of cofree Hopf algebras}.
\newblock {\em J. Reine Angew. Math.}, 592:123--155, 2006.

\bibitem[LS07]{LS07}
T.~F. Lam and M.~Shimozono.
\newblock {Dual graded graphs for Kac-Moody algebras}.
\newblock {\em Algebr. Number Theory}, 1(4):451--488, 2007.

\bibitem[LV12]{LV12}
J.-L. Loday and B.~Vallette.
\newblock {\em {Algebraic Operads}}, volume 346 of {\em Grundlehren der
  mathematischen Wissenschaften}.
\newblock Springer, Heidelberg, 2012.
\newblock xxiv+634.

\bibitem[Mé15]{Men15}
M.. Méndez.
\newblock {\em {Set Operads in Compbinatorics and Computer Science}}.
\newblock Springer International Publishing, 2015.
\newblock xv+129.

\bibitem[Nze06]{Nze06}
J.~Nzeutchap.
\newblock {Graded Graphs and Fomin's $r$-correspondences associated to the Hopf
  Algebras of Planar Binary Trees, Quasi-symmetric Functions and Noncommutative
  Symmetric Functions}.
\newblock {\em Formal Power Series and Algebraic Combinatorics}, 2006.

\bibitem[Sag01]{Sag01}
B.~E. Sagan.
\newblock {\em {The Symmetric Group}}, volume 203 of {\em Graduate Texts in
  Mathematics}.
\newblock Springer-Verlag New York, second edition, 2001.
\newblock {Representations, Combinatorial Algorithms, and Symmetric Functions}.

\bibitem[Slo]{Slo}
N.~J.~A. Sloane.
\newblock {The On-Line Encyclopedia of Integer Sequences}.
\newblock \url{https://oeis.org/}.

\bibitem[Sta70]{Sta70}
J.~Stasheff.
\newblock {\em {$H$-spaces from a homotopy point of view}}, volume 161 of {\em
  Lect. Notes in Math.}
\newblock Springer-Verlag, Berlin-New York, 1970.
\newblock v+95.

\bibitem[Sta88]{Sta88}
R.~P. Stanley.
\newblock {Differential posets}.
\newblock {\em J. Am. Math. Soc.}, 1(4):919--961, 1988.

\bibitem[Sta11]{Sta11}
R.~P. Stanley.
\newblock {\em {Enumerative Combinatorics}}, volume~1.
\newblock Cambridge University Press, second edition, 2011.

\bibitem[Ter03]{Ter03}
Terese.
\newblock {\em {Term Rewriting Systems}}.
\newblock Cambridge University Press, 2003.
\newblock 884.

\bibitem[Yau16]{Yau16}
D.~Yau.
\newblock {\em {Colored Operads}}.
\newblock Graduate Studies in Mathematics. American Mathematical Society, 2016.

\end{thebibliography}
\end{footnotesize}

\end{document}